\documentclass[leqno,12pt]{amsart} 
\usepackage{amssymb}
\usepackage{amsmath}
\usepackage{amsthm}
\usepackage{amscd}
\usepackage{graphicx}
\usepackage[all]{xy}
\usepackage[dvips]{color}
\theoremstyle{plain} 
\newtheorem{theorem}{\indent\sc Theorem}[section]
\newtheorem{lemma}[theorem]{\indent\sc Lemma}
\newtheorem{corollary}[theorem]{\indent\sc Corollary}
\newtheorem{proposition}[theorem]{\indent\sc Proposition}

\newtheorem{conjecture}[theorem]{\indent\sc Conjecture}
\theoremstyle{definition} 

\newtheorem{remark}[theorem]{\indent\sc Remark}
\newtheorem{example}[theorem]{\indent\sc Example}

\newtheorem{problem}[theorem]{\indent\sc Problem}
\newtheorem{question}[theorem]{\indent\sc Question}
\begin{document}

\title[Hermitian metrics on the anti-canonical bundle]
{Hermitian metrics on the anti-canonical bundle of the blow-up of the projective plane at nine points} 

\author[T. Koike]{Takayuki Koike} 

%
\address{
Department of Mathematics, Graduate School of Science, Osaka City University \endgraf
3-3-138, Sugimoto, Sumiyoshi-ku Osaka, 558-8585 \endgraf
Japan
}
\email{tkoike@sci.osaka-cu.ac.jp}

\maketitle

\begin{abstract}
We investigate Hermitian metrics on the anti-canonical bundle of a rational surface obtained by blowing up the projective plane at nine points. 
For that purpose, we pose a modified variant of an argument made by Ueda on the complex analytic structure of a neighborhood of a subvariety by considering the deformation of the complex structure. 
\end{abstract}

\section{Introduction}

In this paper, we investigate the complex analytic structure of a small neighborhood of a subvariety of a complex manifold. 
As our motivation comes from a study of Hermitian metrics of the anti-canonical bundle of some concrete examples of complex projective manifolds, 
we explain and describe our main results in this context in this section. 

First, let us explain our main interest on Hermitian metrics of line bundles, 
which is on semi-positivity criteria for nef line bundles. 
Let $X$ be a complex projective manifold and $L$ be a holomorphic line bundle on $X$. 
We say that $L$ is {\it nef} if the intersection number $(L. C)$ is non-negative for any compact complex curve $C$ of $X$, 
and that $L$ is {\it semi-positive} if $L$ admits a $C^\infty$'ly smooth Hermitian metric $h$ such that Chern curvature tensor $\sqrt{-1}\Theta_h$ is semi-positive. 
As is easily shown, $L$ is nef if it is semi-positive. 
The first example, as far as the author knows, of $(X, L)$ such that $L$ is nef however is not semi-positive was constructed by Demailly, Peternell, and Schneider in \cite[Example 1.7]{DPS}. 
In \cite{K6}, we also constructed another example of such $(X, L)$ by choosing suitable nine points $Z=\{p_ 1, p_2, \dots, p_9\}$ from the complex projective plane $\mathbb{P}^2$, letting $X$ be the blow-up ${\rm Bl}_Z\mathbb{P}^2$ of $\mathbb{P}^2$ at $Z$, and by letting $L$ be the anti-canonical line bundle $K_X^{-1}$. 
On the other hand, by the studies of Arnol'd, Ueda, and Brunella, it is known that the anti-canonical line bundle $K_X^{-1}$ is semi-positive when $X={\rm Bl}_Z\mathbb{P}^2$ is as above for almost every choice of the nine points $Z$ in the sense of the Lebesgue measure \cite{A}, \cite{U}, \cite{B} (see also  \cite[\S 1]{D2}). 
Let us note that, for any of the examples above, there exists a reduced divisor $Y\subset X$ which is included in the complete linear system $|L|$ such that the line bundle $L|_Y:=i_Y^*L$ is topologically trivial, where $i_Y\colon Y\to X$ is the inclusion. 
Motivated by such circumstance, we are interested in the following (see also Conjecture \ref{conj:main}): 

\begin{problem}
Let $X$ be a complex manifold and $Y\subset X$ be a (reduced) hypersurface. 
Assume that $Y$ is compact and $N_{Y/X}:=i_Y^*[Y]$ is topologically trivial, where $[Y]$ is the holomorphic line bundle on $X$ such that $\mathcal{O}_X([Y])=\mathcal{O}_X(Y)$. 
When is $[Y]$ semi-positive?
\end{problem}

We will see previous results on this problem generally in \S \ref{section:relation_sp_nbhd}. 
Here, let us focus on one of the most interesting cases, which is the case where $X=X_Z$ is the one obtained by blowing up $\mathbb{P}^2$ at nine points $Z=\{p_ 1, p_2, \dots, p_9\}\subset \mathbb{P}^2$. 
As nothing is unclear on (singular) Hermitian metrics on $K_{X_Z}^{-1}$ if it is not nef (see \cite[\S 7]{K6}), we assume that $K_{X_Z}^{-1}$ is nef. 
Then, according to \cite[Proposition 7.10]{K6}, there exists a reduced divisor $Y_Z\in |K_{X_Z}^{-1}|$ such that the restriction $K_{X_Z}^{-1}|_{Y_Z}$ admits a flat connection (i.e. all the transition functions are $\mathbb{C}^*$-constant for a suitable choice of local trivializations) and that $Y_Z$ is the strict transform of either a smooth elliptic curve, a cycle of rational curves, a curve with a cusp, or three lines intersecting at a point of $\mathbb{P}^2$. 
Here we mean by {\it a cycle of rational curves} a one-dimensional compact reduced variety with only nodal (i.e. normal crossing ) singularities whose normalization consists of finite numbers of $\mathbb{P}^1$'s and whose dual graph is a cycle graph 
(Note that a rational curve with a node is also a cycle of rational curves in our definition). 
We are mainly interested in the case where $Y_Z$ is a smooth elliptic curve or a cycle of rational curves, since the other cases have already been investigated in the proof of \cite[Proposition 7.10 (ii)]{K6}. 

First, let us describe our main result for $(X_Z, Y_Z)$ when $Y_Z$ is a cycle of rational curves. 
In this case, there exists an isomorphism $\alpha\colon {\rm Pic}^0(Y_Z)\to \mathbb{C}^*:=\mathbb{C}\setminus \{0\}$. 
Note that all the elements of ${\rm Pic}^0(Y_Z)$ admit flat connections (i.e. the natural map $H^1(Y, \mathbb{C}^*)\to {\rm Pic}^0(Y_Z)$ is surjective). 
Note also that an element $L\in {\rm Pic}^0(Y_Z)$ is unitary flat (i.e. $L$ admits a Hermitian metric $h$ with $\sqrt{-1}\Theta_j\equiv 0$) if and only if $|\alpha(L)|=1$. 
In the present paper, we show the following: 

\begin{theorem}\label{thm:main_9pt_b-up_of_P^2_nodal}
Let $(X_Z, Y_Z)$ be as above. 
Assume that the anti-canonical bundle $K_{X_Z}^{-1}$ is nef, and that $Y_Z$ is a cycle of rational curves. 
Then the followings are equivalent: \\
$(i)$ $K_{X_Z}^{-1}$ is semi-positive. \\
$(ii)$ $N_Z$ is unitary flat,  where $N_Z:=K_{X_Z}^{-1}|_{Y_Z}$ (i.e. $|\alpha(N_Z)|=1$).  \\
$(iii)$ $Y_Z$ admits a pseudoflat neighborhoods system (i.e. there exists a fundamental system $\{V_\varepsilon\}_\varepsilon$ of neighborhoods of $Y_Z$ in $X_Z$ such that the boundary $\partial V_\varepsilon$ is Levi-flat for each $\varepsilon$). \\
$(iv)$ The set $\{T\in c_1(K_{X_Z}^{-1})\mid T: \text{closed semi-positive}\ (1, 1)-\text{current}\}$ is not a singleton. 
\end{theorem}

Note that, 
when $K_{X_Z}^{-1}$ is nef and $Y_Z$ has a singular point which is not a node (i.e. $Y_Z$ is a curve with a cusp or three lines intersecting at a point), 
it follows from the argument in the proof of \cite[Proposition 7.10]{K6} that the assertions $(i), (iii)$, and $(iv)$ in Theorem \ref{thm:main_9pt_b-up_of_P^2_nodal} are equivalent. 
Note also that, under the assumption in the theorem above, it is known that $Y_Z$ admits a strongly pseudoconcave neighborhoods system when $N_Z$ is {\bf not} unitary flat (i.e. $|\alpha(N_Z)|\not=1$, \cite{U91} for a rational curve with a node, \cite[Theorem 1.6]{K6} in general). 
It is shown by Brunella that 
the assertions $(i)$ and $(iii)$ in Theorem \ref{thm:main_9pt_b-up_of_P^2_nodal} are equivalent to each other when $Y_Z$ is non-singular and $X_Z\setminus Y_Z$ does not contain any compact complex curve \cite[Theorem 1 (i)]{B}. 
Therefore one can regard Theorem \ref{thm:main_9pt_b-up_of_P^2_nodal} as a singular analogue of Brunella's theorem. 

Let us add some explanation on known results on the semi-positivity of $K_{X_Z}^{-1}$ when $Y_Z$ is a cycle of rational curves. 
As is easily observed, $K_{X_Z}^{-1}$ is semi-ample if and only if $K_{X_Z}^{-1}|_{Y_Z}$ is torsion in ${\rm Pic}^0(Y_Z)$ (i.e. $\alpha(N_Z)=e^{2\pi\sqrt{-1}\theta}$ for some rational number $\theta$). 
In  this case, $K_{X_Z}^{-1}$ is semi-positive by a standard argument (see \S \ref{section:relation_sp_nbhd}). 
In \cite{K6}, we showed that the conditions $(i), (iii)$ and $(iv)$ in Theorem \ref{thm:main_9pt_b-up_of_P^2_nodal} hold if $\alpha(N_Z)=e^{2\pi\sqrt{-1}\theta}$ holds for some Diophantine irrational number $\theta$ \cite[Theorem 1.4, Corollary 7.5]{K6}. 
Here we say that an irrational number $\theta$ is {\it Diophantine} if there exist positive numbers $A$ and $\alpha$ such that $\min_{n\in\mathbb{Z}}|n-m\theta|\geq A\cdot m^{-\alpha}$. 
For example, any algebraic irrational number is Diophantine according to Liouville's theorem. 
By combining \cite[Theorem 1.6 $(ii)$]{K6} and the argument in the proof of \cite[Theorem 1.1]{K3} that $K_{X_Z}^{-1}$ is not semi-positive if $|\alpha(N_Z)|\not=1$. 
Therefore, our contribution in the present paper is in the case where $\alpha(N_Z)=e^{2\pi\sqrt{-1}\theta}$ holds for some non-Diophantine irrational real number $\theta$. 

Next, let us describe our main result for $(X_Z, Y_Z)$ when $Y_Z$ is a smooth elliptic curve. 
In this case, any topologically trivial line bundle $L$ on $Y_Z$ admits a unitary flat structure  (see \cite[\S 1.1]{U} for example). 
Denote 
by $\rho_L\colon \pi_1(Y_Z, *)\to \mathrm{U}(1)$ the the unitary representation of the fundamental group which corresponds to $L$: i.e. $\rho_L$ is the one obtained by considering the holonomy of $\mathcal{F}_L$ along the zero section, where $\mathrm{U}(1):=\{t\in \mathbb{C}\mid |t|=1\}$ and $\mathcal{F}_L$ is the foliation of the total space of $L$ which corresponds to the connection defined by a flat metric (see \cite[\S 2.1]{K8} for the detail). 
Denote by ${\rm rank}(L)$ the rank of the image of $\rho_L$ as a finitely generated abelian group. 
Note that ${\rm rank}(L)=0$ if and only if $L$ is torsion in ${\rm Pic}^0(Y_Z)$ (i.e. there is a positive integer $m$ such that $L^{\otimes m}$ is holomorphically trivial), 
that ${\rm rank}(L)=1$ if and only if any leaf of $\mathcal{F}_L$ is biholomorphic to $\mathbb{C}^*$ except for the zero section, 
and that ${\rm rank}(L)=2$ if and only if any leaf of $\mathcal{F}_L$ is biholomorphic to $\mathbb{C}$ except for the zero section. 
Our main result in the case where $Y_Z$ is smooth is the following: 

\begin{theorem}\label{thm:main_9pt_b-up_of_P^2_smooth}
Let $(X_Z, Y_Z)$ be as above. 
Assume that $Y_Z$ is smooth and ${\rm rank}(N_Z)<2$, where $N_Z:=K_{X_Z}^{-1}|_{Y_Z}$. 
Then $K_{X_Z}^{-1}$ is semi-positive, and $Y_Z$ admits a pseudoflat neighborhoods system. 
\end{theorem}

Note that, as is classically known, $K_{X_Z}^{-1}$ is semi-ample if and only if ${\rm rank}(N_Z)=0$. As $K_{X_Z}^{-1}$ is semi-positive in this case, we are interested in the case where ${\rm rank}(N_Z)>0$ (i.e. $N_Z$ is non-torsion). 
It follows from the studies of Arnol'd, Ueda, and Brunella \cite{A}, \cite{U}, \cite{B} that 
$K_{X_Z}^{-1}$ is semi-positive if $N_Z$ is Diophantine. 
Here we say that a topologically trivial line bundle $L$ on $Y_Z$ is {\it Diophantine} if there exist positive numbers $A$ and $\alpha$ such that $d_{{\rm Pic}^0(Y_Z)}(\mathbb{I}_Z, N_Z^{\otimes m})\geq A\cdot m^{-\alpha}$, where $\mathbb{I}_Z$ is the holomorphically trivial line bundle on $Y_Z$ and $d_{{\rm Pic}^0(Y_Z)}$ is an invariant distance on ${\rm Pic}^0(Y_Z)$ (see \cite[\S 4.1]{U}). 
Therefore, our contribution in the present paper when $Y_Z$ is smooth is in the case where ${\rm rank}(N_Z)=1$ and $N_Z$ is not Diophantine. 

Let us briefly explain the idea how to construct a $C^\infty$ Hermitian metric with semi-positive curvature under the condition $(ii)$ of Theorem \ref{thm:main_9pt_b-up_of_P^2_nodal}, or the assumption of Theorem \ref{thm:main_9pt_b-up_of_P^2_smooth}. 
Based on an argument in the proof of \cite[Theorem 1]{B}, or the argument which we described in \cite[\S 5]{K8} (see also \S \ref{section:relation_sp_nbhd}), the problem is reduced to show the existence of an open covering $\{V_j\}_j$ of a neighborhood of $Y_Z$ and a system $\{\widehat{w}_j\}_j$ of holomorphic functions on $V_j$ such that 
$\widehat{w}_j$ is a defining function of $Y_Z\cap V_j$ 
and that $|\widehat{w}_j/\widehat{w}_k|\equiv 1$ holds on each $V_j\cap V_k$. 
Take a sufficiently fine covering $\{V_j\}$. 
Then, it follows from a simple argument that there exists a defining function $w_j$ on $V_j$ of $Y_Z\cap V_j$ for each $j$ such that $|w_j/w_k|\equiv 1$ holds on $Y_Z\cap V_j\cap V_k$. 
Following the strategy of \cite{A} or \cite[\S 4]{U}, we will modify $w_j$ by solving a functional equation in the form of 
\[
w_j=\widehat{w}_j+\sum_{m=2}^\infty F_{j, m}\cdot \widehat{w}_j^m
\]
on each $V_j$ (after shrinking $V_j$ suitably), where each $F_{j, m}$ is a suitably constructed holomorphic function on $V_j$. 
As is described in \cite[\S 4.2]{U} or \cite[\S 4.2.1]{K6}, one can actually construct coefficient functions $\{F_{j, m}\}$ inductively on $m$ in our setting (see also \cite{N}). 
Thus the problem is reduced to the $L^\infty$ estimate of each $A_m:=\max_j\sup_{V_j}|F_{j, m}|$ so that the formal power series $\sum_{m=2}^\infty A_m X^m\in \mathbb{C}[[X]]$ has a positive radius of convergence (then one can actually solve the functional equation above by using the implicit function theorem, see \S \ref{section:proof_outline} for details). 

For such purpose, one need to estimate the $L^\infty$ operator norm of the coboundary map 
$\delta\colon \check{C}^0(\{U_j\}, \mathcal{O}_{Y_Z}(N_Z^{-m}))\to \check{C}^1(\{U_j\}, \mathcal{O}_{Y_Z}(N_Z^{-m}))$ of the \v{C}ech cocycles, where $U_j:=Y_Z\cap V_j$ and $N_Z^{-m}$ is the $m$-th tensor power of the dual line bundle $N_Z^{-1}$ of $N_Z$. 
According to Ueda's lemma \cite[Lemma 4]{U} or its singular analogue we will describe as Lemma \ref{lem:ueda_4_new_nodal}, the operator norm of the coboundary map can be estimated from above by 
$K\cdot |1-\alpha(N_Z^m)|^{-1}$ when $Y_Z$ is a cycle of rational curves, 
and by $K\cdot (d_{{\rm Pic}^0(Y_Z)}(\mathbb{I}_Z, N_Z^{\otimes m}))^{-1}$ when $Y_Z$ is smooth, where $K$ is a constant which does not depend on $m$. 
This type of estimate is not enough for the case where $N_Z$ can be ``too-well" approximated by a sequences of torsion line bundles (for example when $\alpha(N_Z^m)=e^{2\pi\sqrt{-1}\theta}$ for some Liouville number $\theta\in \mathbb{R}\setminus\mathbb{Q}$). 
In order to overcome this difficulty, 
we alternatively 
consider a deformation family $\pi\colon \mathcal{X}\to S$ such that each fiber is isomorphic to $X_Z$ for some nine points configuration $Z$. 
Then, by using the maximum principle suitably, we can improve the estimate in our situation (see also \S \ref{section:our_strategy_in_toy_model}). 

The organization of the paper is as follows. 
In \S 2, we will review the relationship between neighborhood theories and the semi-positivity of nef line bundles. Here we will also review Ueda theory, give an explanation on the original proof of Ueda's theorem, and explain our basic strategy to prove the main theorems. 
In \S 3, we describe a generalized configuration and state our main result as Theorem \ref{thm:main_general}. 
In \S 4, we give two more concrete configurations as examples of the generalized configuration we describe in \S 3. 
In \S 5, we prove Theorem \ref{thm:main_general}. 
In \S 6,  we give some examples and prove Theorem \ref{thm:main_9pt_b-up_of_P^2_nodal} and \ref{thm:main_9pt_b-up_of_P^2_smooth}. 
\vskip3mm
{\bf Acknowledgment. } 
This work was supported by JSPS Grant-in-Aid for Research Activity Start-up 
18H05834, 
by MEXT Grant-in-Aid for Leading Initiative for Excellent Young Researchers（LEADER) No. J171000201,  
and partly by Osaka City University Advanced Mathematical Institute (MEXT Joint Usage/Research Center on Mathematics and Theoretical Physics).

\section{Preliminaries}

\subsection{Semi-positivity of a nef line bundles and a neighborhood of the stable base locus}\label{section:relation_sp_nbhd}

Let $X$ be a complex manifold and $L$ be a holomorphic line bundle on $X$. 
For a positive integer $m$, we denote by $L^m$ the $m$-th tensor power $L^{\otimes m}$ and by $L^{-m}$ the $m$-th tensor power $(L^*)^{\otimes m}$ of the dual bundle $L^*$ of $L$. 
A line bundle $L$ is said to be {\it semi-ample} if $\mathcal{O}_X(L^m)$ is globally generated for some positive integer $m$. 
If $L$ is semi-positive, then $L$ is semi-positive. 
Indeed, $h:=\Phi_{|L^m|}^*h_{\rm FS}$ is a $C^\infty$ Hermitian metric on $L^m$ with $\sqrt{-1}\Theta_h\geq 0$, where $m$ is a positive integer such that $\mathcal{O}_X(L^m)$ is globally generated, $\Phi_{|L^m|}\colon X\to \mathbb{P}(H^0(X, \mathcal{O}_X(L^m)))$ is the map defined by the complete linear system $|L^m|$, and $h_{\rm FS}$ is Fubini-Study metric on 
$\mathcal{O}_{\mathbb{P}(H^0(X, \mathcal{O}_X(L^m)))}(1)$. 
By considering the metric $h^{1/m}$, one can attach a $C^\infty$ Hermitian metric on $L$ with semi-positive curvature. 

In what follows, we drop the condition that $X$ is projective and assume that $X$ is just a complex manifold. 
A holomorphic line bundle $L$ on $X$ is said to be {\it effective} if $H^0(X, \mathcal{O}_X(L))\not=0$. 
When $L$ is effective, there exists a non-zero global section $f\in H^0(X, \mathcal{O}_X(L))$. 
Let $D$ be the corresponding divisor: $D:={\rm div}(f)$. 
Note that $L\cong [D]$ and $\mathcal{O}_X(L)=\mathcal{O}_X(D)$ hold, where $[D]$ is the holomorphic line bundle on $X$ which corresponds to the divisor $D$. 
In this case, the current $T_D$ defined by 
$\langle T_D, \varphi\rangle:=\int_D\varphi$ (for a compactly supported $C^\infty$ form $\varphi$) is a closed semi-positive $(1, 1)$-current such that $T_D\in c_1(L)$. 
Therefore we have that  $\{T\in c_1(K_{X_Z}^{-1})\mid T: \text{closed semi-positive}\ (1, 1)-\text{current}\}$ is not empty. 
Note that this set is not a singleton if $L$ is semi-positive, since $\sqrt{-1}\Theta_h\in c_1(L)$ is also a closed semi-positive $(1, 1)$-current for any $C^\infty$ Hermitian metric $h$ of $L$ with semi-positive curvature. 

We say that $L$ is {\it $\mathbb{C}^*$-flat}, or $L$ {\it admits a flat connection}, if $L$ is an element of the image of the natural map 
$H^1(X, \mathbb{C}^*)\to H^1(X, \mathcal{O}_X^*)$, or equivalently, all the transition functions are taken as constant functions valued in $\mathbb{C}^*$ for a suitable choice of local trivializations. 
We say that $L$ is {\it unitary flat} if $L$ is an element of the image of the natural map 
$H^1(X, \mathrm{U}(1))\to H^1(X, \mathcal{O}_X^*)$, or equivalently, all the transition functions are taken as constant functions valued in $\mathrm{U}(1)$ for a suitable choice of local trivializations. 
If $L$ is unitary flat, then $L$ is semi-positive. 
Actually one can easily construct a $C^\infty$ Hermitian metric $h$ on a unitary flat line bundle $L$ such that $\sqrt{-1}\Theta_h\equiv 0$. 
Note that any topologically trivial holomorphic line bundle on a compact K\"ahler manifold is unitary flat (see \cite[\S 1.1]{U}). 

Let $Y$ be a reduced divisor of $X$. 
In this paper, we denote by $N_{Y/X}$ the line bundle $[Y]|_Y:=i_Y^*[Y]$ on $Y$, where $i_Y\colon Y\to X$ is the inclusion. 
As $N_{Y/X}$ coincides with the holomorphic normal bundle of $Y$ when $Y$ is non-singular, we call this bundle the normal bundle of $Y$ in general. 
Even when $N_{Y/X}$ is unitary flat, it may possible that $L:=[Y]$ is not semi-positive. 
Indeed, \cite[Example 1.7]{DPS} gives such an example. 

Assume that there exists a neighborhood $V$ of $Y$ in $X$ such that $L|_V$ is unitary flat, where $L:=[Y]$. 
In this case, $L|_V$ admits a $C^\infty$ Hermitian metric $h_V$ with semi-positive curvature. 
On the other hand, $L$ also admits a singular Hermitian metric $h_{\rm sing}$ such that $h_{\rm sing}|_{X\setminus Y}$ is a $C^\infty$ Hermitian metric on $L|_{X\setminus Y}$, $h_{\rm sing}\to\infty$ holds when a point approaches to $Y$, and that Chern curvature current of $h_{\rm sing}$ is semi-positive. 
Indeed, the singular Hermitian metric $h_{\rm sing}$ defined by $|f_Y|^2_{h_{\rm sing}}\equiv 1$ enjoys these properties, where $f_Y\in H^0(X, \mathcal{O}_X(Y))$ is a canonical section. 
A $C^\infty$ Hermitian metric $h$ on $L$ with semi-positive curvature can be constructed by using {\it the regularized minimum construction} for these two metrics $h_V$ and $h_{\rm sing}$, which is the same construction as we used for proving \cite[Corollary 3.4]{K2} (see also \cite[\S 5]{K8}). 
Here we briefly explain this construction. 
Fix a relatively compact open neighborhood $V_0$ of $Y$ in $V$. 
For a sufficiently large constant $C$, one can easily see that the continuous Hermitian metric $h$ on $L$ defined by 
\[
h_x:=\begin{cases}
    \min\{C\cdot (h_V)_x,\ (h_{\rm sing})_x\} & \text{if}\ x\in V_0 \\
    (h_{\rm sing})_x & \text{if}\ \in X\setminus \overline{V_0}
  \end{cases}
\]
is well-defined, and that the curvature current $\sqrt{-1}\Theta_h$ is semi-positive. 
By replacing the function ``$\min$'' in the construction above with ``a regularized minimum function" (see \cite[\S 5.E]{Dbook}), one can make $h$ smooth. 

Again, let $X$ be a complex manifold and $Y$ be a reduced divisor of $X$. 
For a singular Hermitian metric $h$ on $L=[Y]$ with semi-positive curvature current, 
one can construct a plurisubharmonic function $\Phi_h$ on $X\setminus Y$ by $\Phi_h :=-\log |f_Y|_h^2$, where $f_Y\in H^0(X, \mathcal{O}_X(Y))$ is a canonical section. 
As $\Phi_h(x)=O(-\log d(Y, x))$ as $x$ approaches to $Y$ if $h$ is smooth,  one can prove the non semi-positivity of $L$ if one can deny the existence of such a plurisubharmonic function on $X\setminus Y$, which is the strategy of the proof of the main theorem in \cite{K3}. 

For investigating the flatness of $L|_V$ or the plurisubharmonic function on $V\setminus Y$ for a neighborhood $V$ of $Y$, we apply Ueda theory \cite{U}, \cite{U91} and its analogues \cite{K5}, \cite{K6}, \cite{K8}, which will be reviewed in the next subsection. 

Motivated by the argument as above together with some results such as \cite[Theorem 1 $(i)$]{B} and Theorem \ref{thm:main_9pt_b-up_of_P^2_nodal}, we pose the following: 

\begin{conjecture}\label{conj:main}
Let $X$ be a projective manifold and $Y\subset X$ be a reduced subvariety of codimension $1$. 
Assume that the restriction $L|_Y (:=i_Y^*L)$ of the line bundle $L:=[Y]$ to $Y$ is unitary flat (or topologically trivial), where $i_Y\colon Y\to X$ is the inclusion. 
Then $L$ is semi-positive if and only if there exists a neighborhood $V$ of $Y$ such that $L|_V$ is semi-positive. 
\end{conjecture}

Note that, if the conjecture above is affirmative, then $L=[Y]$ is not semi-positive when $Y$ is a smooth compact curve and the pair $(Y, X)$ is of type $(\gamma)$ (see \S \ref{section:review_ueda_theory} for the definition), which is actually the case for some examples as we showed in \cite{K7} and \cite{KO2}. 

\subsection{Short review for Ueda theory and its analogues}\label{section:review_ueda_theory}

Let $X$ be a complex manifold and $Y\subset X$ be a holomorphically embedded compact complex subvariety with topologically trivial normal bundle. 

In \cite{U}, Ueda investigated the complex analytic structure of a neighborhood of $Y$ when $X$ is surface and $Y$ be a smooth (i.e. non-singular) complex curve (see also \cite{N}). 
He defined the obstruction class $u_m(Y, X)\in H^1(Y, \mathcal{O}_X(N_{Y/X}^{-m}))$ ($m\geq 1$), whose definition will be explained below in a generalized configuration. 
In broad strokes, he defined the obstruction classes by comparing $[Y]|_V$ and $\widetilde{N}$ on a small tubular neighborhood $V$ of $Y$, where $\widetilde{N}$ is the flat extension of $N_{Y/X}$ (i.e. $\widetilde{N}$ is the unitary flat line bundle on $V$ which corresponds to $N_{Y/X}$ via the natural isomorphism $H^1(V, \mathrm{U}(1))\to H^1(Y, \mathrm{U}(1))$). 
The $m$-th obstruction class $u_m(Y, X)$, which is called {\it $m$-th Ueda class}, reflects the difference between $[Y]|_V$ and $\widetilde{N}$ in $m$-th order jet along $Y$. 
By using Ueda classes, he classified the pair $(Y, X)$ into the following four cases: 
$(\alpha)$: The case where $u_m(Y, X)\not=0$ for some $m\geq 0$, which means that $[Y]|_V$ and $\widetilde{N}$ are different to each other in a finite order jet along $Y$. In this case, the pair $(Y, X)$ is said to be of {\it finite type}. 
Otherwise, the pair is said to be of {\it infinite type}. 
$(\beta')$: The case where there exists a proper holomorphic map $\pi\colon V\to \mathbb{D}$ onto the unit disk $\mathbb{D}:=\{w\in\mathbb{C}\mid |w|<1\}$ by shrinking $V$ if necessary such that $\pi^*\{0\}=mY$ holds as divisors for some  positive integer $m$. 
$(\beta'')$: The case where $[Y]|_V$ is non-torsion and unitary flat by shrinking $V$ if necessary. In these two cases, $Y$ admits a pseudoflat fundamental neighborhoods system. 
$(\gamma)$: The remaining case. 

Ueda showed that $Y$ admits a strongly pseudoconcave neighborhoods system if $(Y, X)$ is of type $(\alpha)$ \cite[Theorem 1]{U}. In this case, he investigated the detailed grouth properties of a plurisubharmonic function defined on $V\setminus Y$ \cite[Theorem 2]{U}. 
He also established a singular analogue of these theorems for the case where $Y$ is a rational curve with a node and $N_{Y/X}$ is not unitary flat \cite{U91}, which was slightly generalized by the author to, for example, the case where $Y$ is a cycle of rational curves \cite[Theorem 1.6]{K6}. By combining these and the argument we explained in the previous subsection, one obtain many examples of nef line bundles which are not semi-positive (see \cite{K3}, \cite[Corollary 1.3]{K6}). 
In the infinite type case, Ueda gave a sufficient condition for the pair $(Y, X)$ to be of type $(\beta')$ or $(\beta'')$ \cite[Theorem 3]{U}, whose proof also makes sense even when $Y$ is a manifold of higher dimension. 
Note that \cite[Theorem 3]{U} can be regarded as a generalization of Arnol'd's theorem \cite{A}. 
The following theorem is a generalized variant of Ueda's theorem: 

\begin{theorem}[{\cite[Theorem 3]{U}, \cite[Theorem 1.4]{K6}}, see also {\cite[Theorem 1.1]{K8}}]\label{thm:ueda_thm_for_infin_type}
Let $X$ be a complex manifold and $Y\subset X$ be a compact reduced subvariety of codimension $1$ such that the normal bundle is unitary flat. 
Assume one of the following two conditions: $Y$ is non-singular, or $Y$ is a cycle of rational curves. 
Then there exists a neighborhood $V$ of $Y$ such that the line bundle $[Y]|_V$ is unitary flat if the pair $(Y, X)$ is of infinite type and the norml bundle $N_{Y/X}:=[Y]|_Y$ is either torsion or Diophantine. 
\end{theorem}
Here we say that a topologically trivial line bundle $L$ on a cycle of rational curve is Diophantine if $\alpha(N_{Y/X})=e^{2\pi\sqrt{-1}\theta}$ for a Diophantine number $\theta\in\mathbb{R}\setminus \mathbb{Q}$ (see \S 1 for the definition when $Y$ is non-singular). 
Especially, note that the pair $(Y, X)$ is of type $(\beta')$ if and only if it is of infinite type and $N_{Y/X}$ is a torsion element of ${\rm Pic}^0(Y)$. 
We will roughly review the strategy of the proof of Theorem \ref{thm:ueda_thm_for_infin_type} in the next subsection. 

In \cite{K5}, \cite{KO} and \cite{K8}, we investigated a higher codimensional analogue of Ueda theory. According to \cite{K6} and \cite{K8}, we explain the definition of a generalized variant of Ueda classes. 
Let $X$ be a complex manifold, and $Y\subset X$ be a compact reduced subvariety of codimension $r\geq 1$ such that $N_{Y/X}$ is unitary flat. 
Assume $Y$ is a cycle of rational curves and $r=1$ whenever $Y$ is singular for simplicity. 
Take a finite open covering $\{U_j\}$ of $Y$ and a neighborhood $V_j$ of $U_j$ in $X$. 
When $Y$ is singular, one may assume the following condition by refining them if necessary: 
$U_j\cap U_k\not=\emptyset$ and $U_k\cap Y_{\rm sing}\not=\emptyset$ imply $U_j\cap Y_{\rm sing}=\emptyset$. 
Take a defining function $w_j\colon V_j\to \mathbb{C}^r$ of $U_j$ for each $j$: 
i.e. $w_j$ is a holomorphic function on $V_j$ such that ${\rm div}(w_j^{(\lambda)})$'s transversally intersect along $U_j$, where $w_j^{(\lambda)}\colon V_j\to\mathbb{C}$ is the composition of $w_j$ and $\lambda$-th projection map $\mathbb{C}^r\to\mathbb{C}$. 
By a simple argument, one may assume that $dw_j=T_{jk}dw_k$ holds on each $U_{jk}:=U_j\cap U_k$ for some unitary matrix $T_{jk}\in\mathrm{U}(r)$ by changing $w_j$'s if necessary, where 
\[
dw_j:=\left(
    \begin{array}{c}
      dw_j^{(1)} \\
      dw_j^{(2)} \\
      \vdots \\
      dw_j^{(r)}
    \end{array}
  \right).  
\]
By shrinking $V_j$'s if necessary again, we assume that, for each $j$ with $U_j\cap Y_{\rm sing}=\emptyset$, there exists a holomorphic surjection ${\rm Pr}_{U_j}\colon V_j\to U_j$ such that $(w_j, z_j\circ {\rm Pr}_{U_j})$ are coordinates of $V_j$, where $z_j$ is a coordinate of $U_j$. 
In what follows, for any holomorphic function $f$ on $U_j$, we denote by the same letter $f$ the pull-buck ${\rm Pr}_{U_j}^*f:=f\circ {\rm Pr}_{U_j}$. 
Take $U_j$ and $U_k$ such that $U_{jk}\not=\emptyset$. 
We may assume that $U_j\cap Y_{\rm sing}=\emptyset$. 
In this case, one obtain the series expansion 
\[
T_{jk}\cdot \left(
    \begin{array}{c}
      w_k^{(1)} \\
      w_k^{(2)} \\
      \vdots \\
      w_k^{(r)}
    \end{array}
  \right)
=
\left(
    \begin{array}{c}
      w_j^{(1)} \\
      w_j^{(2)} \\
      \vdots \\
      w_j^{(r)}
    \end{array}
  \right)
+\sum_{|a|\geq 2} \left(
    \begin{array}{c}
      f_{kj, a}^{(1)}(z_j) \\
      f_{kj, a}^{(2)}(z_j) \\
      \vdots \\
      f_{kj, a}^{(r)}(z_j)
    \end{array}  \right)
\cdot w_j^a, 
\]
where $a=(a_1, a_2, \dots, a_r)$ is the multiple index running all the elements of $(\mathbb{Z}_{\geq 0})^r$ with $|a|:=\sum_{\lambda=1}^ra_r\geq2$, 
$f_{kj, a}^{(\lambda)}$'s are holomorphic functions on $U_{jk}$ (we regard this also as a function defined by $({\rm Pr}_{U_j}|_{{\rm Pr}_{U_j}^{-1}(U_{jk})})^*f_{kj, a}^{(\lambda)}$), 
and $w_j^a:=\prod_{\lambda=1}^r(w_j^{(\lambda)})^{a_\lambda}$. 
For a positive integer $m$, we say that $\{(V_j, w_j)\}$ is of {\it type $m$} if $f_{kj, a}\equiv 0$ holds for any $a$ with $|a|\leq m$ and any $j, k$ such that $U_{jk}\not=\emptyset$ and $U_j\cap Y_{\rm sing}=\emptyset$. 
If $\{(V_j, w_j)\}$ is of type $m$, it follows that 
$\{(U_{jk}, u_{jk})\}$ satisfies the $1$-cocycle condition, where 
\[
u_{jk}:=\sum_{\lambda=1}^r\sum_{|a|=m+1}f_{kj, a}^{(\lambda)} \, \frac{\partial}{\partial w_j^{(\lambda)}}\otimes dw_j^a, 
\]
and thus it defines an element of $H^1(Y, \mathcal{O}_Y(N_{Y/X}\otimes S^{m+1}N_{Y/X}^*))$. 
We denote by $u_m(Y, X)$ the class $[\{(U_{jk}, u_{jk})\}]$, which is the definition of the $m$-th Ueda class. 
Note that we define $u_{kj}:=-u_{jk}$ for $j$ with $U_j\cap Y_{\rm sing}\not=\emptyset$ when $Y$ is singular. 
Ueda class $u_m(Y, X)$ is well-defined up to the $\mathrm{U}(r)$-action of $H^1(Y, \mathcal{O}_Y(N_{Y/X}\otimes S^{m+1}N_{Y/X}^*))$: 
i.e. it does not depend on the choice of the system $\{(V_j, w_j)\}$ of type $m$ as an element of $H^1(Y, \mathcal{O}_Y(N_{Y/X}\otimes S^{m+1}N_{Y/X}^*))/\mathrm{U}(r)$ \cite[p. 588]{U}, 
\cite[Proposition 3.6]{K6}, \cite[Lemma 3.6]{K8}. 

\begin{remark}\label{rmk:CLPT}
The obstruction can be similarly defined even when the normal bundle is $\mathbb{C}^*$-flat if once we fix a system of type $m$. 
However, it is not the case on the well-definedness of $u_m(Y, X)$, see \cite[Remark 2.2]{CLPT}. 

\end{remark}

Finally, in the rest of this subsection, let us summarize the situation when $(Y, X)=(Y_Z, X_Z)$ is the pair as in \S 1: 
i.e. $X$ is the blow-up of $\mathbb{P}^2$ at nine points $Z=\{p_1, p_2, \dots, p_9\}$ 
such that $Z\subset C\setminus C_{\rm sing}$, where $C$ is either a smooth elliptic curve or a cycle of rational curves in $\mathbb{P}^2$, 
$C_{\rm sing}$ is the singular part of $C$, 
and $Y$ is the strict transform of $C$. 
Assume that the anti-canonical bundle $K_X^{-1}$ is nef. 
In this case, $N_{Y/X}$ is topologically trivial. 
By the arguments in \cite[\S 1.1]{U}, $N_{Y/X}$ is unitary flat if $Y$ is a smooth elliptic curve. 
As will be seen in \S \ref{section:setting_nodal}, when $Y$ is a cycle of rational curves, there exists an isomorphism $\alpha\colon {\rm Pic}^0(Y)\to \mathbb{C}^*$ such that an element $L\in {\rm Pic}^0(Y)$ is unitary flat if and only if $|\alpha(L)|=1$ (see also \cite[\S 1]{U91} for the case where $Y$ is a rational curve with a node). 
If $N_{Y/X}$ is not unitary flat, $Y$ admits a strongly pseudoconcave neighborhood and the line bundle $[Y]$ is not semi-positive. 
In this case, $T_Y$ as in the previous subsection is only the element of the set $\{T\in c_1(K_{X_Z}^{-1})\mid T: \text{closed semi-positive}\ (1, 1)-\text{current}\}$ 
(It follows by combining the arguments in the proof of \cite[Theorem 1. 1]{K3} and \cite[Theorem 1.6 $(ii)$]{K6}). 
In what follows, we will investigate the case where $N_{Y/X}$ is unitary flat. 
In this case, it follows from the argument as in \cite{N} that $(Y, X)$ is of infinite type even when $N_{Y/X}$ is torsion (it is clear that the pair is of infinite type if $N_{Y/X}$ is non-torsion, since $H^1(Y, \mathcal{O}_Y(N_{Y/X}^{-m}))=0$ holds for any positive integer $m$ in this case). 
Therefore, by Theorem \ref{thm:ueda_thm_for_infin_type}, it follows that there exists a neighborhood $V$ of $Y$ such that $[Y]|_V$ is unitary flat if $N_{Y/X}$ is torsion or Diophantine. 
Thus, by using the regularized minimum construction as we explained in the previous subsection (see also \cite[\S 5]{K8}), $K_X^{-1}$ is semi-positive if $N_{Y/X}$ is torsion or Diophantine. 

When $Y$ is an elliptic curve, 
denote by $p= p(N_{Y/X})$ and $q= q(N_{Y/X})$ the real number such that the holonomy of the foliation defined by the flat metric along $\gamma_1$ and $\gamma_2$ is equal to 
$\exp(2\pi\sqrt{-1}p)$ and $\exp(2\pi\sqrt{-1}q)$, respectively, where $\gamma_1$ and $\gamma_2$ are generators of the fundamental group $\pi_1(Y, *)$ of $Y$. 
The normal bundle $N_{Y/X}$ is torsion if and only if both $p$ and $q$ are rational. 
The normal bundle $N_{Y/X}$ is Diophantine if either $p$ or $q$ is a Diophantine irrational number. 
According to Theorem \ref{thm:main_9pt_b-up_of_P^2_smooth}, $K_X^{-1}$ is semi-positive for $(Y, X)$ as above if $Y$ is smooth and either $p(N_{Y_Z/X_Z})$ or $q(N_{Y_Z/X_Z})$ is rational. 
Thus the remaining problem is as follows: 

\begin{question}
Is $K_X^{-1}$ semi-positive when $(Y, X)=(Y_Z, X_Z)$ as above if $Y_Z$ is smooth and neither $p(N_{Y_Z/X_Z})$ nor $q(N_{Y_Z/X_Z})$ is rational or Diophantine for any choice of the generator $(\gamma_1, \gamma_2)$ of the fundamental group $\pi_1(Y, *)$ of $Y$?  
\end{question}

Note that, when $N_{Y/X}$ is neither torsion nor Diophantine, 
Ueda constructed an example of $(Y, X)$ which is of type $(\gamma)$ \cite[\S 5.4]{U}. 
For his example, we showed that $[Y]$ is not semi-positive \cite{K7} (see also \cite{KO2}). 
However, as this Ueda's example or some examples in \cite{KO2} are essentially only the known examples of type $(\gamma)$ and $X$ is non-compact surface in these examples, 
we know nothing on the existence of such an example when $X$ is compact. 

\subsection{Outline of the proof of Ueda type linearization theorems}\label{section:toy_model_single}

In this subsection, we explain the outline of the proof of Ueda type linearization theorem such as Theorem \ref{thm:ueda_thm_for_infin_type} by using a toy model. 

Let $Y$ be a smooth elliptic curve $\mathbb{C}/\langle 1, \sqrt{-1}\rangle$: i.e. $Y$ is the quotient $\mathbb{C}/\sim$, where ``$\sim$'' is the relation generated by $z\sim z+1\sim z+\sqrt{-1}$. Denote by $[z]$ the image of $z$ by the quotient map $\mathbb{C}\to Y$. 
Define an open covering $\{U_j\}_{j=1, 2, 3}$ by 
$U_j:=\{[z]\in\mathbb{R}/\mathbb{Z}\mid 2j<{\rm Re}\,z<(2j+1)/6\}$. 
Let $X$ be a non-singular surface which includes $Y$ as a subvariety such that 
$N_{Y/X}^{-1}=[\{(U_{jk}, t_{jk})\}]\in\check{H}^1(\{U_j\}, \mathcal{O}_Y^*)$ holds, where 
\[
t_{jk}=\begin{cases}
s & \text{if}\ (j, k)=(3, 1)\\
s^{-1} & \text{if}\ (j, k)=(1, 3)\\
1 & \text{otherwise}
\end{cases}
\]
for some $s\in \mathrm{U}(1)$. 
Additionally, we assume for simplicity that there exists a neighborhood $V$ of $Y$ in $X$ and a holomorphic map ${\rm Pr}_Y\colon V\to Y$ such that ${\rm Pr}_Y|_Y$ is the identity, 
and that there exists a defining function $w_j$ on $V_j$ of $U_j$ for each $j$, where $V_j:={\rm Pr}_Y^{-1}(U_j)$. 
By a simple argument, it follows that we may assume $dw_j=t_{jk}dw_k$ holds on each $U_{jk}$. 
Fix a local coordinate $z_j$ of $Y$ on $U_j$, and regard $(z_j, w_j)$ as coordinates of $V_{jk}$ (Here we extend the domain of the function $z_j$ to $V_j$ by pulling buck by ${\rm Pr}_Y$). 
Then, by letting $f_{kj, m}:=\frac{\partial^m}{\partial w_j^m}|_{w_j=0}(t_{jk}\cdot w_k)$, 
one have that 
\[
t_{jk}\cdot w_k=w_j+f_{kj, 2}(z_j)\cdot w_j^2+f_{kj, 3}(z_j)\cdot w_j^3+\dots
\]
holds on each $V_{jk}$. 
Our goal is to construct a new system $\{(V_j, \widehat{w}_j)\}$ of local defining functions such that $\widehat{w}_j=t_{jk}\widehat{w}_k$ by modifying $w_j$'s (and by shrinking $V_j$'s if necessary). 

The strategy of Ueda's proof of \cite[Theorem 3]{U} can be explained as follows: 
Define a suitable holomorphic functions $F_{j, m}\colon U_j\to \mathbb{C}$ for each $j=1, 2, 3$ and for each $m\geq 2$ so that the solution of the functional equation  
\begin{equation}\label{eq:func_eq_toy_model}
w_j=\widehat{w}_j+\sum_{m=2}^\infty F_{j, m}(z_j)\cdot \widehat{w}_j^m
\end{equation}
satisfies $\widehat{w}_j=t_{jk}\widehat{w}_k$ on a neighborhood of $U_{jk}$ (if exists). 
Here we are regarding $F_{j, m}$ as a function on $V_j$ by pulling back by ${\rm Pr}_Y$. 

In order for $\widehat{w}_j$'s to satisfy $\widehat{w}_j=t_{jk}\widehat{w}_k$, both of the coefficients of $\widehat{w}_j^m$ in two expansions
\[
t_{jk}w_k= w_j+\sum_{m=2}^\infty f_{kj, m}\cdot w_j^m
=\widehat{w}_j+\sum_{m=2}^\infty F_{j, m}\cdot \widehat{w}_j^m
+\sum_{\ell=2}^\infty f_{kj, \ell}\cdot \left(\widehat{w}_j+\sum_{n=2}^\infty F_{j, n}\cdot \widehat{w}_j^n\right)^\ell
\]
and 
\[
t_{jk}w_k
=t_{jk}\widehat{w}_k+\sum_{m=2}^\infty t_{jk}F_{k, m}\cdot \widehat{w}_k^m
=t_{jk}\widehat{w}_k+\sum_{m=2}^\infty t_{jk}^{1-m}F_{k, m}\cdot \widehat{w}_j^m
\]
should coincide for each $m$. 
This condition can be reworded as $-F_{j, m}+t_{jk}^{1-m}F_{k, m}=h_{kj, m}$, 
where $h_{kj, m}$ is the coefficient of $\widehat{w}_j^m$ in the expansion of 
\[
\sum_{\ell=2}^\infty f_{kj, \ell}\cdot \left(\widehat{w}_j+\sum_{n=2}^\infty F_{j, n}\cdot \widehat{w}_j^n\right)^\ell \in \mathcal{O}_{Y}(U_j)[[\widehat{w}_j]]. 
\]
Note that each $h_{kj, m}$ only depends on $f_{kj, \mu}$'s and $\{F_{j, \mu}\}_{\mu<m, j=1, 2, 3}$, and does not depend on $\{F_{j, \mu}\}_{\mu\geq m, j=1, 2, 3}$: Indeed, $h_{kj, 2}=f_{kj, 2}$ and $h_{kj, 3}=f_{kj, 3}+2f_{kj, 2}\cdot F_{j, 2}$ hold for example. 
Conversely, it is observed by relatively simple inductive argument that a formal solution $\widehat{w}_j$ of the functional equation $(\ref{eq:func_eq_toy_model})$ satisfies $t_{jk}\widehat{w}_k=\widehat{w}_j$ in any order jet along $U_{jk}$ if $F_{j, m}$'s satisfies $-F_{j, m}+t_{jk}^{1-m}F_{k, m}=h_{kj, m}$. 
Therefore what we should do is the following: 
({\bf Step $1$}) Solve the equation 
\[
\delta\{(U_j, F_{j, m})\} (:=\{(U_{jk}, -F_{j, m}+t_{jk}^{1-m}F_{k, m}\}) = \{(U_{jk}, h_{kj, m})\}\in \check{Z}^1(\{U_j\}, \mathcal{O}_Y(N_{Y/X}^{1-m}))
\]
to obtain $F_{j, m}$'s inductively on $m$. 
({\bf Step $2$})  Estimate $B_m:=\max_j \sup_{U_j}|F_{j, m}|$ and show that the formal power series $\sum_{m=2}^\infty B_mX^m\in \mathbb{C}[[X]]$ has a positive radius of convergence (if so, one actually show the existence of the solution $\widehat{w}_j$ of the functional equation $(\ref{eq:func_eq_toy_model})$ by shrinking $V_j$ by using the implicit function theorem, which completes the proof). 
Note that the assumption that the pair is of infinite type is needed in {\bf Step $1$} (since the class $[\{(U_{jk}, h_{kj, m})\}]$ coincides with $m$-th Ueda class $u_m(Y, X)$, see \cite[p. 598]{U}), 
and that the normal bundle is either torsion or Diophantine is needed in {\bf Step $2$}. 

In the rest of this subsection, we will focus on {\bf Step $2$}. 
In what follows, we assume that each $f_{kj, m}$ is a constant function for simplicity. 
By a simple inductive observation, each $F_{j, m}$ is also constant in this case. 

\begin{remark}
When each $f_{kj, m}$ is a constant function, clearly there is a holomorphic foliation such that each leaves are defined by $\{w_j=\text{constant}\}$. 
In this case, the problem is reduced to the linearization problem of the holonomy function, which explains the strategy of original proof of Arnol'd's theorem \cite{A} in our simple model. 
\end{remark}

For $\sigma\in\mathbb{C}^*$, denote by $L_\sigma$ the unitary flat line bundle on $Y$ defined by 
\[
L_\sigma:= [\{(U_{12}, 1), (U_{21}, 1), (U_{23}, 1), (U_{32}, 1), (U_{31}, \sigma^{-1}), (U_{13}, \sigma)\}]\in \check{H}^1(\{U_j\}, \mathbb{C}^*), 
\]
and by $\mathbb{C}(L_\sigma)$ the sheaf of locally constant sections of $L_\sigma$. 
In the case where each $f_{kj, m}$ is constant, $F_{j, m}$'s are inductively defined by solving 
\[
\delta\{(U_j, F_{j, m})\} (:=\{(U_{jk}, -F_{j, m}+t_{jk}^{1-m}F_{k, m}\}) = \{(U_{jk}, h_{kj, m})\}\in \check{Z}^1(\{U_j\}, \mathbb{C}(L_s^{m-1})). 
\]
The following lemmata are needed for the inductive estimate of $|F_{j, m}|$'s: 

\begin{lemma}\label{lem:ueda_lem_analogue_toymodel}
There exists a constant $K$ which does not depend on $\sigma$ such that, 
for any 
$\alpha=\{(U_{jk}, \alpha_{jk})\}\in \check{Z}^1(\{U_j\}, \mathbb{C}(L_{\sigma}))$ and 
$\beta=\{(U_j, \beta_j)\}\in \check{C}^0(\{U_j\}, \mathbb{C}(L_{\sigma}))$ with 
$\delta\beta=\alpha$, it holds that 
\[
d(1, \sigma)\cdot \max_{j}|\beta_j|\leq K\cdot \max_{j, k}|\alpha_{j, k}|, 
\]
where $d$ is the distance of $\mathrm{U}(1)=\mathbb{R}/\mathbb{Z}$ induced by the Euclidean distance of the universal covering $\mathbb{R}$. 
\end{lemma}

\begin{proof}
If $\sigma=1$, nothing is non-trivial. 
Otherwise, by a simple calculation, one have that 
$\beta_1 = -\frac{A}{1-\sigma}$, $\beta_2 = -\alpha_{31} -\alpha_{23} - \frac{\sigma A}{1-\sigma}$, and $\beta_3 = -\alpha_{31} - \frac{\sigma A}{1-\sigma}$, where $A:=\alpha_{12}+\alpha_{23}+\alpha_{31}$. 
The assertion follows from this and the equivalence of $d$ and the distance of $\mathrm{U}(1)$ obtained by restricting the Euclidean distance of $\mathbb{C}$ ($\supset \mathrm{U}(1)$). 
\end{proof}

\begin{lemma}\label{lem:KS_lem_analogue_toymodel}
For each $\sigma\in \mathrm{U}(1)$, there exists a constant $K_0(\sigma)$ which satisfies the following property:  
For any $\alpha=\{(U_{jk}, \alpha_{jk})\}\in \check{Z}^1(\{U_j\}, \mathbb{C}(L_{\sigma}))$ with $[\alpha]=0\in \check{H}^1(\{U_j\}, \mathbb{C}(L_{\sigma}))$, 
there exists an element $\beta=\{(U_j, \beta_j)\}\in \check{C}^0(\{U_j\}, \mathbb{C}(L_{\sigma}))$ such that $\delta\beta=\alpha$ and 
\[
\max_{j}|\beta_j|\leq K_0(\sigma)\cdot \max_{j, k}|\alpha_{j, k}|
\]
hold. 
\end{lemma}

\begin{proof}
When $\sigma\not=1$, the assertion follows from Lemma \ref{lem:ueda_lem_analogue_toymodel} (Note that the solution $\beta$ is unique in this case). 
When $\sigma=1$, it follows by simple argument that $[\alpha]=0$ holds 
if and only if $\alpha_{12}+\alpha_{23}+\alpha_{31}=0$. 
As the solution $\beta$ can be constructed by letting 
$\beta_1 := 0$, $\beta_2 := -\alpha_{31} -\alpha_{23}$, and $\beta_3 := -\alpha_{31}$, 
the assertion holds if one let $K_0(1):=2$. 
\end{proof}

Note that Lemma \ref{lem:ueda_lem_analogue_toymodel} can be regarded as an analogue of \cite[Lemma 4]{U}, and Lemma \ref{lem:KS_lem_analogue_toymodel} can be regarded as an analogue of \cite[Lemma 3]{U} (=\cite[Lemma 2]{KS}). 

Let $M:=\max_j\sup_{V_j}|w_j|$, 
and assume that there exists a positive constant $R$ such that 
$\{(z_j, w_j)\in V_j\mid z_j\in V_k, |w_j|\leq 1/R\}\Subset V_k$ holds for each $j, k$ for simplicity (we mean by ``$\Subset$'' the relatively compact subset). 
Then, by an inductive argument as in \cite[\S 4]{U}, one have that the formal power series $A(X)=X+\sum_{m=2}^\infty A_m X^m\in \mathbb{R}[[X]]$ defined by the equation 
\[
\sum_{m=2}^\infty \frac{A_m}{K_0(s^{m-1})}\cdot X^m = \frac{MA(X)^2}{1-A(X)}
\]
satisfies $\max_j \sup_{U_j}|F_{j, m}|\leq A_m$ for any $m$, where $K_0(s^{m-1})$ is the one as in Lemma \ref{lem:KS_lem_analogue_toymodel}. 

When $s$ is a torsion element of $\mathrm{U}(1)$, $\{K_0(s^{m-1})\}_{m=2}^\infty$ is a finite set, and thus there exists a constant which is larger than any $K_0(s^{m-1})$. 
When $s=e^{2\pi\sqrt{-1}\theta}$ for a Diophantine irrational number, we have a suitable type of estimate of $K_0(s^{m-1})$'s by Lemma \ref{lem:ueda_lem_analogue_toymodel}. 
In these cases, one can show that the formal power series $A(X)$ has a positive radius of convergence (by using the implicit function theorem suitably in the torsion cacse \cite[\S 4.4]{U}, 
and by using the estimate as in \cite{Sie} in the Diophantine case \cite[\S 4.6]{U}). 

On the other hand, when $s=e^{2\pi\sqrt{-1}\theta}$ for a real number $\theta$ which is neither rational nor Diophantine, Lemma \ref{lem:ueda_lem_analogue_toymodel} is not enough to show the convergence of the formal power series $\sum_{m=2}^\infty \max_j |F_{j, m}|\cdot X^m\in \mathbb{C}[[X]]$ (Indeed, there is a counterexample by Ueda, see \cite[\S 5.4]{U}). 
We will explain our idea to improve the estimates under some special situations in the next subsection. 

\subsection{Our idea to improve the estimates}\label{section:our_strategy_in_toy_model}

In order to improve the estimates of $|F_{j, m}|$'s in the previous subsection, 
we will consider a deformation of the complex structure of $X$. 
Let $S$ be a neighborhood of $\mathrm{U}(1)$ in $\mathbb{C}^*$, 
$\mathcal{X}$ a complex manifold of dimension $3$, 
and $\pi\colon \mathcal{X}\to S$ be a surjective holomorphic submersion. 
Assume that there exist a submanifold $\mathcal{Y}\subset \mathcal{X}$ and a biholomorphism $p\colon \mathcal{Y}\to Y\times S$ such that the diagram 
\[
  \xymatrix{
    \mathcal{Y} \ar[rr]^{p}_{\cong} \ar[rd]_{\pi|_{\mathcal{Y}}} & \ar@{}[d]^{\circlearrowleft} & Y\times S \ar[ld]^{{\rm Pr}_2} \\
     & S & 
  }
\]
commutes, where $Y=\mathbb{C}/\langle 1, \sqrt{-1}\rangle$ and ${\rm Pr}_2$ is the second projection. 
For each $s\in S$, denote by $X_s$ the fiber $\pi^{-1}(s)$ and by $Y_s$ the submanifold $\mathcal{Y}\cap X_s$. 
In what follows, we will identify $\mathcal{Y}$ with $Y\times S$ via $p$. 

Letting $\mathcal{U}_j:=U_j\times Y$, where $U_j$ is the one in the previous subsection, we define a $\mathbb{C}^*$-flat line bundle $\mathcal{L}$ on $\mathcal{Y}=Y\times S$ by 
\[
\mathcal{L}=
[\{(\mathcal{U}_{12}, 1), (\mathcal{U}_{21}, 1), (\mathcal{U}_{23}, 1), (\mathcal{U}_{32}, 1), (\mathcal{U}_{31}, {\rm Pr}_2^{-1}), (\mathcal{U}_{13}, {\rm Pr}_2)\}]\in \check{H}^1(\{\mathcal{U}_j\}, \mathcal{O}_{\mathcal{Y}}^*), 
\]
where ${\rm Pr}_2\colon Y\times S\to S (\subset \mathbb{C}^*)$ is the second projection. 
Note that $\mathcal{L}|_{Y_s}=L_s$ holds via natural isomorphism between $Y_s=Y\times\{s\}$ and $Y$. 
Here we assume that $N_{\mathcal{Y}/\mathcal{X}}=\mathcal{L}$. 
Then one can regard each pair $(Y_s, X_s)$ as the one we observed in the previous subsection for each $s\in\mathrm{U}(1)$ under suitable additional assumptions. 

\begin{remark}
As will be seen in \S \ref{example:smooth}, the pair $(Y_Z, X_Z)$ as in \S 1 is settled in a fiber of such a deformation $\pi\colon \mathcal{X}\to S$ when $Y_Z$ is smooth. 
On the other hand, we constructed in \cite[Example 4.3]{K7} a pathological example of such $\pi\colon \mathcal{X}\to S$, in which the line bundle $[Y_s]$ on $X_s$ is semi-positive for almost every $s\in S$ in the sense of Lebesgue measure, whereas it is not semi-positive for uncountably many $s\in \mathrm{U}(1)$. 
In order to distinguish these two cases, {\bf Condition $(*)$} we will add below is important. 
Note that the pair $(Y_s, X_s)$ is of infinite type for any torsion element $s$ in $\mathrm{U}(1)$ in the former example \cite{N}, 
and that the pair $(Y_s, X_s)$ is of finite type for any torsion element $s$ in $\mathrm{U}(1)$ in the latter example \cite[Example 4.3]{K7}. 
\end{remark}

Assume, again for simplicity, that there exists a holomorphic retraction ${\rm Pr}_\mathcal{Y}$ from a neighborhood of $\mathcal{Y}$ onto $\mathcal{Y}$ such that ${\rm Pr}_\mathcal{Y}|_\mathcal{Y}$ is the identity, 
and that there exists a neighborhood $\mathcal{V}_j$ of each $\mathcal{U}_j$ in $\mathcal{X}$ and a defining function $w_j\colon \mathcal{V}_j\to \mathbb{C}$ of $\mathcal{U}_j$ such that $\{(V_{j, s}, w_j|_{V_{j, s}})\}$ is a system as in the previous subsection for each $s\in S$, where $V_{j, s}:=\mathcal{V}_j\cap X_s$. 
Under such a situation, one can consider the {\it ``simultaneous linearization problem''} for $\{(\mathcal{V}_j, w_j)\}$, which is the problem on constructing a new defining functions system $\{(\mathcal{V}_j, \widehat{w}_j)\}$ by shrinking $\mathcal{V}_j$'s if necessary such that 
$t_{jk}w_k=w_j$ holds on each $\mathcal{V}_{jk}$, where $t_{jk}\colon \mathcal{V}_{jk}\to \mathbb{C}^*$ is the function defined by 
\[
t_{jk}(x):=\begin{cases}
{\rm Pr}_2\circ {\rm Pr}_\mathcal{Y}(x) & \text{if}\ (j, k)=(3, 1)\\
({\rm Pr}_2\circ {\rm Pr}_\mathcal{Y}(x))^{-1}& \text{if}\ (j, k)=(1, 3)\\
1 & \text{otherwise}
\end{cases}, 
\]
where ${\rm Pr}_2\colon Y\times S\to S (\subset \mathbb{C}^*)$ is the second projection. 
Under some additional technical assumptions including the following {\bf Condition $(*)$} and {\bf $(**)$}, this simultaneous linearization problem is reduced to the estimate of the $L^\infty$ operator norm of the function 
\[
\delta\colon \check{C}^0(\{\mathcal{U}_j\}, \mathcal{O}_\mathcal{Y}(\mathcal{L}^{m-1}))\to \check{C}^1(\{\mathcal{U}_j\}, \mathcal{O}_\mathcal{Y}(\mathcal{L}^{m-1}))
\]
for each $m\geq 2$ 
by the same argument as in the previous subsection and by an analogue of Lemma \ref{lem:KS_lem_analogue_toymodel} (Proposition \ref{prop:mainlemma_toy} below for a toy model case, Proposition \ref{prop:main_estim} for the actual configuration). 

\begin{description}
\item[Condition $(*)$] The pair $(Y_s, X_s)$ is of infinite type for any $s\in S$. 
\item[Condition $(**)$] For each $m\geq 2$, $[\{(U_{jk}\times \{s\}, h_{kj, m}|_{U_{jk}\times\{s\}})\}]=0\in \check{H}^1(\{U_{j, s}\}, \mathcal{O}_Y(L_s^{m-1}))$ holds for any $s\in S$, 
where $h_{kj, m}$ is the function on $\mathcal{U}_{jk}$ defined in the same manner as in the previous subsection. 
\end{description}
 
In fact, it turns out that {\bf Condition $(**)$} implies {\bf Condition $(*)$}, 
since the class $[\{(U_{jk}\times \{s\}, h_{kj, m}|_{U_{jk}\times\{s\}})\}]$ coincides with $m$-th Ueda class for $s\in \mathrm{U}(1)$, as is mentioned in the previous subsection. 
On the other hand, even when {\bf Condition $(*)$} holds, it may possible that {\bf Condition $(**)$} does not hold, since the class $[\{(U_{jk}\times \{s\}, h_{kj, m}|_{U_{jk}\times\{s\}})\}]$ does not necessarily coincide with Ueda class (see also Remark \ref{rmk:CLPT}). 
In the actual situation (i.e. in the case where each fiber $X_s$ is realized as $X_Z$ as in \S 1), 
{\bf Condition $(*)$} holds. 
Although {\bf Condition $(**)$} is much more useful for running an argument we will explain as Proposition \ref{prop:mainlemma_toy} below for a toy model case (in order to explain the idea of the proof of Proposition \ref{prop:main_estim} for actual configuration), 
we do not know whether or not {\bf Condition $(**)$} holds in the actual situation. 
Thus, we need to run the inductive argument by carefully shrinking $S$ depending on $m$ so that {\bf Condition $(**)$} holds at each step of the induction 
(Note that the cohomology group $\check{H}^1(\{U_{j, s}\}, \mathcal{O}_Y(L_s^{m-1}))$ itself vanishes if $s^{m-1}\not=1$, which means that {\bf Condition $(*)$} implies {\bf Condition $(**)$} for $m=m_0$ if $S$ is so small that $\{s\in S\mid s^{m_0-1}=1\}\subset \mathrm{U}(1)$). 

In the rest of this subsection, we will explain the idea how to establish an analogue of Lemma \ref{lem:KS_lem_analogue_toymodel} for the family configuration and estimate the operator norm of 
$\delta\colon \check{C}^0(\{\mathcal{U}_j\}, \mathcal{O}_\mathcal{Y}(\mathcal{L}^{m-1}))\to \check{C}^1(\{\mathcal{U}_j\}, \mathcal{O}_\mathcal{Y}(\mathcal{L}^{m-1}))$. 
In what follows, we replace $S$ and $\mathcal{O}_\mathcal{Y}(\mathcal{L}^{m-1})$ with $\mathrm{U}(1)$ and $\mathcal{A}_{Y\times \mathrm{U}(1)}(\mathcal{L}^{m})$, respectively, and consider the coboundary map 
\[
\delta\colon \check{C}^0(\{U_j\times \mathrm{U}(1)\}, \mathcal{A}_{Y\times \mathrm{U}(1)}(\mathcal{L}^{m}))\to \check{C}^1(\{U_j\times \mathrm{U}(1)\}, \mathcal{A}_{Y\times \mathrm{U}(1)}(\mathcal{L}^{m})). 
\]
as a toy model, where $\mathcal{A}_{Y\times \mathrm{U}(1)}(\mathcal{L}^{m})$ is 
the sheaf of the sections of $\mathcal{L}^{m}$ which are locally constant in $Y$ direction and real analytic in $\mathrm{U}(1)$ direction. 

For each positive integer $m$ and for each $\nu=0, 1, \dots, m-1$, define a subset $W_{m, \nu}$ of $\mathrm{U}(1)$ by 
\[
W_{m, \nu}:=\left\{\exp(2\pi\sqrt{-1}\theta)\in \mathrm{U}(1)\left|\ \theta\in\mathbb{R}\,  \left|\theta-\frac{\nu}{m}\right|\leq \frac{1}{2m}\right.\right\}. 
\]
Note that $s=\zeta_m^\nu$ is the only element $s$ of $W_{m, \nu}$ such that $s^m=1$, where $\zeta_m:=\exp(2\pi\sqrt{-1}/m)$. 
For an open neighborhood $\widetilde{W}_{m, \nu}$ defined by 
\[
\widetilde{W}_{m, \nu}:=\left\{\exp(2\pi\sqrt{-1}\theta)\in \mathrm{U}(1)\left|\ \theta\in\mathbb{R}\,  \left|\theta-\frac{\nu}{m}\right|\leq \frac{3}{4m}\right.\right\}
\]
of $W_{m, \nu}$ in $\mathrm{U}(1)$, we first show the following: 

\begin{proposition}\label{prop:mainlemma_toy}
Let 
$\alpha=\{(U_{jk}\times\widetilde{W}_{m, \nu}, \alpha_{jk})\}$ 
be an element of $\check{Z}^1(\{U_j\times \widetilde{W}_{m, \nu}\}, \mathcal{A}_{Y\times \mathrm{U}(1)}(\mathcal{L}^{m}))$. 
Denote by $\alpha_s=\{(U_{jk}\times\{s\}, \alpha_{jk}|_{U_{jk}\times\{s\}})\} \in \check{Z}^1(\{U_{j, s}\}, \mathbb{C}(L_s^{m}))$ the element obtained by restricting $\alpha$ to $Y_s$ for each $s\in \widetilde{W}_{m, \nu}$. 
Then the followings are equivalent: \\
$(i)$ $[\alpha_{\zeta_m^\nu}]=0\in \check{H}^1(\{U_{j, \zeta_m^\nu}\}, \mathbb{C}(L_{\zeta_m^\nu}^m))$. \\
$(ii)$ $[\alpha|_{Y\times \widetilde{W}_{m, \nu}}]=0\in \check{H}^1(\{U_j\times \widetilde{W}_{m, \nu}\}, \mathcal{A}_{Y\times \mathrm{U}(1)}(\mathcal{L}^m))$. 
\end{proposition}

\begin{proof}
As clearly $(ii)$ implies $(i)$, we show the converse. 
Assume $(i)$ holds. 
Denote by $\widetilde{W}_{m, \nu}^*$ the set $\widetilde{W}_{m, \nu}\setminus\{\zeta_m^\nu\}$. 
It follows from an argument in the proof of Lemma \ref{lem:ueda_lem_analogue_toymodel} that
the element 
$\beta=\{(U_j\times\widetilde{W}_{m, \nu}^*, \beta_j)\}
\in\check{C}^0(\{U_j\times \widetilde{W}_{m, \nu}^*\}, \mathcal{A}_{Y\times \mathrm{U}(1)}(\mathcal{L}^{m}))$ 
defined by 
$\beta_1(s) = -\frac{A(s)}{1-s^m}$, 
$\beta_2(s) = -\alpha_{31}(s) - \alpha_{23}(s) - \frac{s^m A(s)}{1-s^m}$, 
and 
$\beta_3(s) = -\alpha_{31}(s) - \frac{s^m A(s)}{1-s^m}$
satisfies $\delta\beta=\alpha|_{Y\times \widetilde{W}_{m, \nu}^*}$, 
where $A(s):=\alpha_{12}(s)+\alpha_{23}(s)+\alpha_{31}(s)$. 
As it follows from the argument in  the proof of Lemma \ref{lem:KS_lem_analogue_toymodel} that 
$A(\zeta_m^\nu)=0$, 
one can naturally extend $\beta_j$ to define a primitive $\beta\in\check{C}^0(\{U_j\times \widetilde{W}_{m, \nu}\}, \mathcal{A}_{Y\times \mathrm{U}(1)}(\mathcal{L}^{m}))$. 
\end{proof}

Under {\bf Condition $(*)$}, one have that the assertion $(i)$ of Proposition \ref{prop:mainlemma_toy} holds for any $\nu$. 
If the function 
\[
W_{m, \nu}\ni s\mapsto \left|\frac{s^m A(s)}{1-s^m}\right| \in \mathbb{C}
\]
is convex, whose complex analytical counterpart is actually the case in the actual situation in some sense, one have that
\[
\max_{s\in W_{m, \nu}}\left|\frac{s^m A(s)}{1-s^m}\right|
=\max_{s\in \partial W_{m, \nu}}\left|\frac{s^m A(s)}{1-s^m}\right|, 
\]
where $A$ is the one in the proof of Proposition \ref{prop:mainlemma_toy}. 
In this case, for a point $s_*\in \partial W_{m, \nu}$ which attains the maximum of the function above, it follows from 
Lemma \ref{lem:ueda_lem_analogue_toymodel} that 
\begin{eqnarray}
\max_{s\in W_{m, \nu}}\left|\frac{s^m A(s)}{1-s^m}\right|
&=&\left|\frac{s_*^m A(s_*)}{1-s_*^m}\right|
\leq \left|\frac{s_*^m A(s_*)}{\varepsilon d(1, s_*^m)}\right|\nonumber \\
&=&\frac{|A(s_*)|}{\varepsilon \cdot md(\zeta_m^\nu, s_*)}
=\frac{2}{\varepsilon}\cdot |A(s_*)|
\leq \frac{6}{\varepsilon} \cdot\max_{j, k}\sup_{s\in W_{m, \nu}}|\alpha_{jk}(s)|, \nonumber
\end{eqnarray}
where $\varepsilon$ is a positive constant such that $\varepsilon d(1, \sigma)\leq |1-\sigma|$. Note that $\varepsilon$ depends on neither $s$ nor $m$. 
As therefore the constant $6/\varepsilon$ depends on neither $s$ nor $m$, 
one can regard this estimate as an improved variant of Lemma \ref{lem:ueda_lem_analogue_toymodel} under {\bf Condition $(*)$}. 

\section{Main result for a suitable deformation configuration}\label{section:general}

According to the observation in \S \ref{section:our_strategy_in_toy_model}, 
we will pose a generalized configuration, which can be regarded as a generalization of a configuration of the blow-up model of $\mathbb{P}^2$ at suitably chosen nine points as we described in \S 1 (see the next section), and state our main result in this section. 

Let $S\subset \mathbb{C}^*$ be a neighborhood of ${\rm U}(1)$, 
$\mathcal{X}$ a complex manifold, 
$\pi\colon \mathcal{X}\to S$ a surjective holomorphic submersion, 
and $\mathcal{Y}\subset \mathcal{X}$ be a reduced subvariety of codimension $r>0$ such that 
there exist a connected reduced compact complex variery $Y$ and a biholomorphism $p\colon \mathcal{Y}\cong Y\times S$ such that the following diagram commutes. 
\[
  \xymatrix{
    \mathcal{Y} \ar[rr]^{p}_{\cong} \ar[rd]_{\pi|_{\mathcal{Y}}} & \ar@{}[d]^{\circlearrowleft} & Y\times S \ar[ld]^{{\rm Pr}_2} \\
     & S & 
  }
\]
Assume that $Y$ is either a manifold (i.e. a non-singular variety) or an analytic space of dimension one with only nodal singularities. 
Also assume that $r=1$ in the latter case (i.e. when $Y$ is singular). 
Take a finite open covering $\{U_j\}$ of $Y$ such that each $U_j$ is Stein, connected and simply connected which satisfies the following properties: 
For each $p\in Y_{\rm sing}$, there {\it uniquely} exists $U_j$ such that $p\in U_j$, and it holds that $\bigcup_{U_j\cap Y_{\rm sing}=\emptyset}U_j=Y\setminus Y_{\rm sing}$. 
Note that especially it holds that $Y_{\rm sing}\cap U_{jk}=\emptyset$ holds for any $j$ and $k$ if $j\not= k$. 
By considering a refinement of $\{U_j\}$ if necessary, we may assume that it also holds that $U_j\cap Y_{\rm sing}=\emptyset$ if 
$U_{jk}\not=\emptyset$ and $U_k\cap Y_{\rm sing}\not=\emptyset$. 

Define an open covering $\{\mathcal{U}_j\}$ of $\mathcal{Y}$ by letting $\mathcal{U}_j:=p^{-1}(U_j\times S)$. 
Take a neighborhood $\mathcal{V}_j$ of $\mathcal{U}_j$ in $\mathcal{X}$ such that $\mathcal{V}_j\cap \mathcal{Y}=\mathcal{U}_j$. 
We assume that $\mathcal{V}_{jk}:=\mathcal{V}_j\cap \mathcal{V}_k$ is empty if and only if $\mathcal{U}_{jk}:=\mathcal{U}_j\cap \mathcal{U}_k$ is empty 
by shrinking $\mathcal{V}_j$'s if necessary. 
We use the following notation for each $s\in S$: 
$X_s:=\pi^{-1}(s)$, $Y_s:=X_s\cap \mathcal{Y}=p^{-1}(Y\times\{ s\})$, 
$U_{j, s}:=X_s\cap \mathcal{U}_j = p^{-1}(U_j\times\{ s\})$, 
$V_{j, s}:=X_s\cap \mathcal{V}_j$, 
$U_{jk, s}:=U_{j, s}\cap U_{k, s}$, and 
$V_{jk, s}:=V_{j, s}\cap V_{k, s}$. 
We denote by $\mathbb{I}_s$ the holomorphically trivial line bundle on $Y_s$. 

In what follows, we give five assumptions. 
First one is the following: 

\begin{description}
\item[Assumption $1$] (Cohomology vanishing assumption) 
For any topologically trivial holomorphic line bundle $L$ on $Y$ which is not holomorphically trivial, 
it holds that $H^1(Y, \mathcal{O}_Y(L))=0$. 
\qed
\end{description}

The second one is: 

\begin{description}
\item[Assumption $2$](Flatness assumption for the normal bundles) 
There exists a holomorphic function 
\[
w_j=\left(
    \begin{array}{c}
      w_j^{(1)} \\
      w_j^{(2)} \\
      \vdots \\
      w_j^{(r)}
    \end{array}
  \right)\colon \mathcal{V}_j\to \mathbb{C}^r
\] 
such that the following six conditions hold: 
$(i)$ $\{w_j^{(\lambda)}\}_{\lambda=1}^r$ is a system of defining functions of $\mathcal{U}_j$ (i.e. $\{w_j^{(\lambda)}\}_{\nu=1}^r$ generates the defining ideal sheaf $I_{\mathcal{U}_{j}}\subset\mathcal{O}_{\mathcal{V}_{j}}$ of $\mathcal{U}_{j}$). 
$(ii)$ $\max_\lambda\max_j\sup_{\mathcal{V}_j}|w_j^{(\lambda)}|$  is bounded. 
$(iii)$ For each $j$ and $k$ with $\mathcal{U}_{jk}\not=\emptyset$, there exists a holomorphic function $t_{jk}^{(\lambda)}\colon S\to \mathbb{C}^*$ for each $\lambda=1, 2, \dots, r$ such that $dw_j^{(\lambda)} = t_{jk}^{(\lambda)}dw_k^{(\lambda)}$ holds on $\mathcal{U}_{jk}$, where we denote by the same letter $t_{jk}^{(\lambda)}$ the composition of $p\colon \mathcal{Y}\to Y\times S$, the second projection, and $t_{jk}^{(\lambda)}$. 
$(iv)$ When $U_k$ is singular, there exist coordinates $(x_k, y_k, s)$ of $\mathcal{V}_k$ such that $w_k^{(1)}=x_k\cdot y_k$ and $\pi(x_k, y_k, s)=s$. 
$(v)$ When $U_j$ is smooth (i.e. non-singular), $\mathcal{V}_j$ is embedded into $\mathbb{C}^r\times \mathcal{U}_j$ in the following manner: 
There exists a holomorphic retraction ${\rm Pr}_{\mathcal{U}_j}\colon \mathcal{V}_j\to \mathcal{U}_j$ (i.e. ${\rm Pr}_{\mathcal{U}_j}|_{\mathcal{U}_j}$ is the identity) such that $\pi\circ{\rm Pr}_{\mathcal{U}_j}=\pi$ and that 
$\mathcal{V}_j\ni x\mapsto (w_j(x), z_j\circ{\rm Pr}_{\mathcal{U}_j}(x), \pi(x))\in \mathbb{C}^r\times U_j\times S$ 
defines coordinates of $\mathcal{V}_j$, 
where $z_j$ is a coordinate of $U_j$. 
$(vi)$ When $U_k$ is singular, $\mathcal{V}_k$ is embedded into $\mathbb{C}^2\times S$ in the following manner: 
The images of the maps $\mathcal{V}_k\ni (x_k, y_k, s)\mapsto (x_k, 0, s)\in \mathcal{U}_k$ and $\mathcal{V}_k\ni (x_k, y_k, s)\mapsto (0, y_k, s)\in \mathcal{U}_k$ coincide with each of the irreducible components of $\mathcal{U}_k$. 
\qed
\end{description}

In what follows, we often denote simply by $w_j$ the function $w_j^{(1)}$ when $r=1$. 
Denote by $\mathcal{N}^*$ the holomorphic vector bundle on $\mathcal{Y}$ of rank $r$ 
defined by $\mathcal{N}^*:=[\{(\mathcal{U}_{jk}, T_{jk})\}]
\in H^1(\{\mathcal{U}_j\}, {\rm GL}_r(\mathcal{O}_{\mathcal{Y}}))$, 
where $T_{jk}:={\rm diag}(t_{jk}^{(1)}, t_{jk}^{(2)}, \dots, t_{jk}^{(r)})$ is the diagonal matrix, 
and, for each $s\in S$, by $N^*_s$ the holomorphic vector bundle  
$\mathcal{N}^*|_{Y_s}=[\{(U_{jk, s}, T_{jk}(s))\}]
\in H^1(\{U_{j, s}\}, {\rm GL}_r(\mathbb{C}))$ on $Y_s$ of rank $r$. 
We regard them as conormal bundles of $\mathcal{Y}$ and $Y_s$, respectively. 
Note that 
$\mathcal{N}^*=
\bigoplus_{\lambda=1}^r\mathcal{L}_\lambda$ and 
$N^*_s=
\bigoplus_{\lambda=1}^r L_{\lambda, s}$ for each $s\in S$, 
where $\mathcal{L}_\lambda$ is the holomorphic line bundle on $\mathcal{Y}$ defined by 
$\mathcal{L}_\lambda:=[\{(\mathcal{U}_{jk}, t_{jk}^{(\lambda)})\}]
\in H^1(\{\mathcal{U}_j\}, \mathcal{O}_{\mathcal{Y}}^*)$, 
and $L_{\lambda, s}$ is the holomorphic line bundle on $Y_s$ defined by 
$L_{\lambda, s}:=\mathcal{L}_\lambda|_{Y_s}=[\{(U_{jk, s}, t_{jk}^{(\lambda)}(s))\}]
\in H^1(\{U_{j, s}\}, \mathbb{C}^*)$ for each 
$\lambda=1, 2, \dots, r$. 

For a multi-index $a=(a_1, a_2, \dots, a_r)\in\mathbb{Z}^r$, 
we denote by $|a|$ the sum $a_1+a_2+\cdots+a_r$, 
by $\mathcal{L}^a$ the line bundle $\mathcal{L}_1^{a_1}\otimes \mathcal{L}_2^{a_2}\otimes \cdots \otimes \mathcal{L}_r^{a_r}$, 
by $L_s^a$ the line bundle $L_{1, s}^{a_1}\otimes L_{2, s}^{a_2}\otimes \cdots \otimes L_{r, s}^{a_r}$ for each $s\in S$, 
by $t_{jk}^a$ the product $\prod_{\lambda=1}^r(t_{jk}^{(\lambda)})^{a_\lambda}$, 
and by $dw_j^a$ the tensor $\bigotimes_{\lambda=1}^r(dw_j^{(\lambda)})^{\otimes a_\lambda}$.

Fix an integer $M_0\geq 1$. 
Denote by $\widehat{S}_m$ the subset of $S$ defined by 
\[
\widehat{S}_m:=\left\{s\in S\left| d(s, {\rm U}(1))<\frac{1}{M_0m}\right.\right\}
\]
 for each $m\in \mathbb{Z}_{>0}$, where $d$ is the distance of $S$ attached by restricting the distance of $\mathbb{C}^*$ which is induced by the Euclidean distance of $\mathbb{C}$ via the covering map $\mathbb{C}\in\xi\mapsto \exp(2\pi\sqrt{-1}\xi)\in \mathbb{C}^*$.

\begin{description}
\item[Assumption $3$] (Assumption on torsion points and the bound of the transition functions) 
The following three conditions hold: 
$(i)$ The set $\widehat{S}_1$ is a relatively compact subset of $S$. 
$(ii)$ For each positive integer $m$ and each $a=(a_1, a_2, \dots, a_r)\in\mathbb{Z}^r$ with $|a|=m$, 
$\{s\in \widehat{S}_m\mid L_s^a\cong \mathbb{I}_s\}$ is a subset of $\{s\in {\rm U}(1)\mid s^m=1\}$. 
$(iii)$ There exists a positive constant $\Theta>1$ such that $\max_{j, k}\sup_{s\in \widehat{S}_m}|t_{jk}^a(s)|\leq \Theta$ holds for any positive integer $m$ and any $a\in\mathbb{Z}^r$ with $|a|=m$. 
\qed
\end{description}

Note that {\bf Assumption $3$} $(iii)$ implies $|t_{jk}^{(\lambda)}(s)|=1$ for each $s\in {\rm U}(1)$ and each $\lambda$. 
For each positive integer $m$, denote by $S_m$ the subset of $\widehat{S}_m$ defined by 
\[
S_m:=\left\{s\in S\left| d(s, {\rm U}(1))<\frac{1}{4M_0m}\right.\right\}, 
\]
by $\zeta_{M_0 m}$ the point $\exp(2\pi\sqrt{-1}/(M_0 m))$ of $\mathrm{U}(1)\subset S$, 
by $W_{m, \nu}$ the closed subset of $S$ defined by 
\[
W_{m, \nu}:=\left\{\exp(2\pi\sqrt{-1}\xi)\in \widehat{S}_m\left|\ \xi\in\mathbb{C},\ \left|{\rm Re}\,\xi - \frac{\nu}{M_0m}\right| \leq \frac{1}{2M_0m},\ |{\rm Im}\, \xi| \leq \frac{1}{4M_0 m}\right.\right\}
\]
for $\nu=0, 1, 2, \dots, M_0 m-1$, 
and by $d_{a, \nu}\colon W_{m, \nu}\to \mathbb{R}_{\geq 0}$ the function defined by 
\[
d_{a, \nu}(s):=
\begin{cases}
  d(s, \zeta_{M_0 m}^\nu) &\text{if}\ L_{\zeta_{M_0 m}^\nu}^a\text{is holomorphically trivial}\\
  \frac{1}{2M_0 m} &\text{otherwise}
\end{cases}
\]
for each $m, \nu$, and $a$ with $|a|=m$. 
Note that one simply have that $\bigcup_{\nu=0}^{M_0m-1}W_{m, \nu}=\overline{S_m}$, 
$d_{a, \nu}(s)\leq
d(s, \{\sigma\in S\mid L_{\sigma}^a\cong\mathbb{I}_\sigma\})$, 
and that $m\cdot d_{a, \nu}(s)\leq d(1, s^m)$ for each positive integer $m$, multi-index $a$ with $|a|=m$, $\nu=0, 1, 2, \dots, M_0 m-1$, and $s\in W_{m, \nu}$ (Here we use {\bf Assumption $3$} $(ii)$). 

\begin{description}
\item[Assumption $4$] (Ueda-type Lemma) 
There exists a positive constant $K$ such that the following holds: 
For any positive integer $m$, multi-index $a=(a_1, a_2, \dots, a_r)\in\mathbb{Z}^r$ with $|a|=m$, 
$\nu=0, 1, 2, \dots, M_0 m-1$, 
$s\in W_{m, \nu}$, 
a $1$-cochain $\alpha=\{(U_{jk, s}, \alpha_{jk, s}dw_j^a)\}\in \check{Z}^1(\{U_{jk, s}\}, \mathcal{O}_{Y_s}(L_s^a))$, and 
a $0$-cochain $\beta=\{(U_{j, s}, \beta_{j, s}dw_j^a)\}\in \check{C}^0(\{U_{j, s}\}, \mathcal{O}_{Y_s}(L_s^a))$ such that 
$\alpha=\delta \beta:=\{(U_{jk, s}, (-\beta_{j, s}+t_{kj}^a(s)\cdot \beta_{k, s})dw_j^a)\}\in \check{Z}^1(\{U_{jk, s}\}, \mathcal{O}_{Y_s}(L_s^a))$, 
it holds that 
\[
m\cdot d_{a, \nu}(s)\cdot \|\beta\|_{s, a}\leq K\cdot \|\alpha\|_{s, a}, 
\]
where $\|\alpha\|_{s, a}:=\max_{j, k}\sup_{U_{jk, s}}|\alpha_{jk, s}|$ and $\|\beta\|_{s, a}:=\max_{j}\sup_{U_{j, s}}|\beta_{j, s}|$. 
\qed
\end{description}

The final assumption is the following: 

\begin{description}
\item[Assumption $5$] (Assumption on the type of $(Y_s, X_s)$ for $s\in {\rm U}(1)$) 
For each $s\in {\rm U}(1)$, the pair $(Y_s, X_s)$ is of infinite type in the sense of \cite{K6} when $Y$ is singular, and of \cite{K8} when $Y$ is smooth (see also \S \ref{section:review_ueda_theory}). 
\qed
\end{description}

The following theorem is the main result: 

\begin{theorem}\label{thm:main_general}
Let $s$ be an element of ${\rm U}(1) (\subset S)$. 
Then, under {\bf Assumption} $1, 2, 3, 4$, and $5$, the following holds by shrinking $V_{j, s}$'s if necessary: 
there exists a function 
\[
\widehat{w}_{j, s}=\left(
    \begin{array}{c}
      \widehat{w}_{j, s}^{(1)} \\
      \widehat{w}_{j, s}^{(2)} \\
      \vdots \\
      \widehat{w}_{j, s}^{(r)}
    \end{array}
  \right)\colon V_{j, s}\to \mathbb{C}^r
\] 
on each $V_{j, s}$ 
such that 
$\{\widehat{w}_j^{(\lambda)}\}_{\lambda=1}^r$ generates the defining ideal sheaf $I_{U_{j, s}}\subset\mathcal{O}_{V_{j, s}}$ of $U_{j, s}$, 
and that $\widehat{w}_{j, s}^{(\lambda)} = t_{jk}^{(\lambda)}(s)\cdot \widehat{w}_{k, s}^{(\lambda)}$ 
holds on each $V_{jk, s}$ for $\lambda=1, 2, \dots, r$. 
\end{theorem}


\section{Two examples of configurations which satisfies five assumptions in Theorem \ref{thm:main_general}}

Before we give the proof of Theorem \ref{thm:main_general} in 
\S \ref{section:prf_main_thm}, 
we will give two examples of configurations which satisfy five assumptions in \S \ref{section:general}, so that one can apply Theorem \ref{thm:main_general} to the examples in \S 1. 

\subsection{A configuration in which $Y_s$'s are smooth elliptic curves}\label{section:setting_smooth}

Let $Z$ be a manifold of dimension $r+1$, $L$ be a holomorphic line bundle on $Z$, 
$D_1^0, D_2^0, \dots, D_r^0$ be prime divisors of $Z$ such that $\mathcal{O}_Z(D_\lambda^0)\cong \mathcal{O}_Z(L)$ for any $\lambda=1, 2, \dots, r$. 
Assume that $Y^0:=D_1^0\cap D_2^0\cap \cdots \cap D_r^0$ is a smooth elliptic curve,  $D_\lambda^0$'s intersects to each other transversally along $Y^0$, 
and that the intersection number $(L. Y^0)$ is equal to $1$. 
Denote by $p_0\in Y^0$ the point such that $\mathcal{O}_{Y^0}(L|_{Y^0})\cong \mathcal{O}_{Y^0}(p_0)$. 

In what follows, we identify $\Sigma:={\rm Pic}^0(Y^0)$ with $Y^0$ via the map 
$Y^0\ni q\mapsto \mathcal{O}_{Y^0}(q-p_0)\in \Sigma$. 
Let $\overline{\mathcal{X}}$ be the blow-up of $Z\times \Sigma$ along the subvariety 
\[
\{(z, s)\in Y^0\times \Sigma\mid z\ \text{coincides with}\ s\ \text{via the identification above}\}
\subset Z\times \Sigma, 
\]
and $\pi\colon \overline{\mathcal{X}}\to \Sigma$ be the morphism obtained by composing the blow-up morphism and the second projection $Z\times \Sigma\to \Sigma$. 
Note that $\pi$ is a surjective holomorphic submersion, 
and that each fiber $X_s:=\pi^{-1}(s)$ is the blow-up of $Z$ at $s$ for each $s\in \Sigma$. 
Denote by $\overline{\mathcal{D}}_\lambda\subset \overline{\mathcal{X}}$ the strict transform of $D^0_\lambda\times \Sigma$ for each $\lambda=1, 2, \dots, r$, 
and by $\overline{\mathcal{Y}}$ the intersection $\bigcap_{\lambda=1}^r\overline{\mathcal{D}}_\lambda$. 
It is easily observed that $\{D_{\lambda, s}:=X_s\cap \overline{\mathcal{D}}_\lambda\}_{\lambda=1}^r$ intersects transversally along $Y_s:=X_s\cap \overline{\mathcal{Y}}$ for each $s\in \Sigma$. 
Note also that $N_{D_{\lambda, s}/X_s}|_{Y_s}\cong L_s^{-1}$ holds for any $\lambda$ (and thus it holds that $N_{Y_s/X_s}^*\cong L_s^{\oplus r}$) for each $s\in \Sigma$ ($=Y^0$ by the identification we fixed in the above), 
where $L_s$ is the line bundle on $Y_s$ such that $\mathcal{O}_{Y_s}(L_s)\cong \mathcal{O}_{Y^0}(s-p_0)$ holds via the natural isomorphism between $Y_s$ and $Y^0$ 
(i.e. via the composition of the blow-up morphism and the first projection $Z\times \Sigma\to Z$). 

Take an element $\tau$ of the upper half plane 
$\mathbb{H}:=\{t\in\mathbb{C}\mid {\rm Im}\,t>0\}$ such that 
$Y:=\mathbb{C}/\langle 1, \tau \rangle$ is isomorphic to $Y^0$. 
Let $\gamma_1$ and $\gamma_2$ be generators of $\pi_1(Y^0, p_0)$ which corresponds to the deck transformations $+1$ and $+\tau$ of the universal covering $\mathbb{C}\to Y^0$. 
Denote by $\mathcal{L}$ the line bundle on 
$Y^0\times Y$ such that, for each $y\in Y$, the restriction $L_y:=\mathcal{L}|_{Y^0\times\{y\}}$ of it to $Y^0\times\{y\}$ is the unitary flat line bundle which corresponds to the unitary representation $\pi_1(Y^0, *)\to \mathrm{U}(1)$ defined by 
$\gamma_1\mapsto \exp(2\pi\sqrt{-1}\cdot(-q(y)))$ and 
$\gamma_2\mapsto \exp(2\pi\sqrt{-1}p(y))$ 
via the first projection $Y^0\times\{y\}\to Y^0$, 
where 
$p(y)$ and $q(y)$ are the elements of $\mathbb{R}/\mathbb{Z}$ such that $y=[p(y)+q(y)\cdot \tau]$. 
Note that we can naturally regard  $\mathcal{L}$ as a holomorphic line bundle, 
since each line bundle $L_y$ is isomorphic (on $Y^0\times \{y\}$ via the first projection as a holomorphic line bundle) to the $\mathbb{C}^*$-flat line bundle on $Y^0$ which corresponds to the $\mathbb{C}^*$-representation $\pi_1(Y^0, *)\to \mathbb{C}^*$ defined by 
$\gamma_1\mapsto 1$ and $\gamma_2\mapsto \exp(2\pi\sqrt{-1}(p(y)+q(y)\cdot \tau))$. 
It follows from the fact that ${\rm Pic}^0(Y^0)$ has a property as the coarse moduli (see \cite[\S A.1]{KU} for example), 
there exists an isomorphism $i\colon Y\to \Sigma$ such that 
$({\rm id}_{Y^0}\times i)^*N_{\overline{\mathcal{D}}_\lambda/\overline{\mathcal{X}}}^{-1}|_{\overline{\mathcal{Y}}}\cong \mathcal{L}$ via the isomorphism between $\mathcal{Y}$ and $Y^0\times \Sigma$ which is defined by using the blow-up morphism (Note that it does not depend on $\lambda$, since $N_{\overline{\mathcal{D}}_\lambda/\overline{\mathcal{X}}}^{-1}|_{\overline{\mathcal{Y}}}$'s are isomorphic to each other for $\lambda=1, 2, \dots, r$). 

In what follows, we identify $Y$ with $Y^0$ (and thus also with $\Sigma$) via $i$. 
Note that the point $p_0\in Y^0$ is identified with the point $[0]\in Y$. 
Note also that it follows from a simple argument that the following two conditions are equivalent to each other for an element of $y\in \Sigma=Y^0=Y$ under our identifications: 
$(i)$ $y$ is torsion as an element of $Y$ (i.e. $[m(p(y)+q(y)\cdot \tau)]=[0]\in Y$ for some $m\in\mathbb{Z}\setminus\{0\}$), 
$(ii)$ $y$ is torsion as an element of $\Sigma$ (i.e. $L_y^m$ is holomorphically trivial for some $m\in\mathbb{Z}\setminus\{0\}$). 
In such case, we simply say that $y$ is {\it torsion}. 

Fix a rational number $q_0\in \mathbb{Q}$ with $0\leq q_0<1$. 
Set 
\[
\ell_0:=\{[p+q_0\cdot \tau]\in Y\mid p\in \mathbb{R}\}\ \ 
\text{and}\ 
\ell_\infty:=\left\{\left.\left[p+\left(q_0+\frac{1}{2}\right)\cdot \tau\right]\in Y\right| p\in \mathbb{R}\right\}, 
\]
and regard them as subsets of $\Sigma$ via the identification. 
Define a base space $S$ by $S:=\Sigma\setminus \ell_\infty$. 
We regard $S$ as a neighborhood $\{s\in\mathbb{C}^*\mid \exp(-\pi{\rm Im}\,\tau)<|s|<\exp(\pi{\rm Im}\,\tau)\}$ 
of $\mathrm{U}(1)$ in $\mathbb{C}^*$ by using the embedding 
\[
S\ni [z]\mapsto \exp(2\pi\sqrt{-1}(z-q_0\cdot \tau))\in \mathbb{C}^*, 
\]
where we are regarding $z$ as an element of $\{p+q\tau\in \mathbb{C}\mid p, q\in\mathbb{R},\ |q-q_0|<1/2\}$. 
In what follows, for each element $s\in S$, we let $p(s)$ and $q(s)$ be the real numbers 
such that $|q(s)-q_0|<1/2$ and $s=[p(s)+q(s)\cdot \tau]$ hold as elements of $Y$ 
($p(s)$ is determined modulo $\mathbb{Z}$, whereas $q(s)$ is determined uniquely). 

We define the total space $\mathcal{X}$ by $\mathcal{X}:=\pi^{-1}(S)\subset \overline{\mathcal{X}}$, 
and the subvariety $\mathcal{Y}$ by $\mathcal{Y}:=\overline{\mathcal{Y}}\cap \mathcal{X}$. 
Denote the restriction $\pi|_{\mathcal{X}}$ simply by the same letter $\pi$, 
and by $\mathcal{D}_\lambda$ the intersection $\overline{\mathcal{D}}_\lambda\cap \mathcal{X}$ for each $\lambda=1, 2, \dots, r$. 
Let $\{U_j\}, \{\mathcal{U}_j\}$, and $\{U_{j, s}\}$ be those as in \S \ref{section:general}. 
In the rest of this subsection, we will show that $\pi\colon\mathcal{X}\to S$ satisfies {\bf Assumption} $1, 2, 3, 4$, and $5$ under the following: 

\begin{description}
\item[Assumption $5$'] The following holds for any torsion element $s\in \Sigma$: 
There exist a complex manifold $B_s$ of dimension $r$, 
$r$ divisors $\{E_\lambda\}_{\lambda=1}^r$ of $B_s$, 
a neighborhood $V_s$ of $Y_s$ in $X_s$, 
and a surjective proper holomorphic map $b_s\colon V_s\to B_s$ such that 
$Y_s$ is a fiber of $b_s$, 
$E_\lambda$'s intersects transversally to each other at the point $b_s(Y_s)$, 
and that $b_s^*E_\lambda=m_\lambda D_{\lambda, s}$ holds as divisors for some positive integer $m_\lambda$ for each $\lambda$.  
\qed
\end{description}

\subsubsection{Assumption $1$}

As $Y$ is an elliptic curve, the assertion of {\rm Assumption $1$} clearly holds. 

\subsubsection{Assumption $2$}

As is obtained by a simple observation, the line bundle $L_s$ is holomorphically isomorphic to the $\mathbb{C}^*$-flat line bundle which corresponds to the $\mathbb{C}^*$-representation $\pi_1(Y^0, *)\to \mathbb{C}^*$ defined by $\gamma_1\mapsto \exp(2\pi\sqrt{-1}(-q_0))$ and $\gamma_2\mapsto \exp(2\pi\sqrt{-1}(p(s)+(q(s)-q_0)\cdot \tau))$. 
Thus we have that, by taking a refinement of $\{U_j\}$ if necessary, there exists a holomorphic map $t_{jk}\colon S\to \mathbb{C}^*$ for each $j$ and $k$ such that 
each $t_{jk}(s)$ coincides with either $1$, $\exp(\pm2\pi\sqrt{-1}q_0)$, $\exp(\pm 2\pi\sqrt{-1}(p(s)+(q(s)-q_0)\cdot \tau))$, or $\exp(\pm2\pi\sqrt{-1}q_0)\cdot \exp(\pm 2\pi\sqrt{-1}(p(s)+(q(s)-q_0)\cdot \tau))$, and that $L_s=[\{(U_{jk, s}, t_{jk}(s))\}]$ holds for each $s\in S$. 

Take a neighborhood $\mathcal{V}_j$ of each $\mathcal{U}_j$ by using \cite[Corollary 1]{Siu}. 
Then we have that each $\mathcal{V}_j$ is Stein, and there exists a defining functions system $\{w_j\}$ and a holomorphic retraction ${\rm Pr}_{\mathcal{U}_j}$ such that the assertions $(ii)$ and $(v)$ in {\bf Assumption} $2$ hold. 
By a standard argument, one can modify $\{w_j\}$ 
so that $w_j^{(\lambda)}$ is a defining function of $\mathcal{D}_\lambda\cap \mathcal{V}_j$ on $\mathcal{V}_j$, and that $dw_j^{(\lambda)}=t_{jk}dw_k^{(\lambda)}$ holds on each $\mathcal{U}_{jk}$, since $N_{\mathcal{D}_\lambda/\mathcal{X}}^{-1}|_{\mathcal{Y}}=[\{(\mathcal{U}_{jk}, t_{jk})\}]$ as elements of $H^1(\{\mathcal{U}_j\}, \mathcal{O}_{\mathcal{Y}}^*)$ for each $\lambda$ (see \cite[\S 2.2]{K8} for example). 
Thus we have that the assertions $(i)$ and $(iii)$ in {\bf Assumption} $2$ also hold by letting $t_{jk}^{(\lambda)}:=t_{jk}$ for each $\lambda$. 

\subsubsection{Assumption $3$}

Fix a sufficiently large integer $m_*$ such that $m_*> \max\left\{2,\ ({\rm Im}\,\tau)^{-1}\right\}$. 
By letting $M_0:=m_*\cdot \min\{m\in\mathbb{Z}\mid m>0, mq_0\in\mathbb{Z}\}$, we here show the assertions $(i)$, $(ii)$, and $(iii)$ in Assumption $3$. 

First, the assertion $(i)$ follows directly by our definition of $M_0$. 
Here let us note that $\mathrm{U}(1)$ is identified with $\ell_0$ in the configuration in this subsection. 

Next, for the assertion $(ii)$, we will show that 
$\{s\in \widehat{S}_m\mid L_s^m\cong \mathbb{I}_s\}\subset \{s\in {\rm U}(1)\mid s^m=1\}$ holds 
for each positive integer $m$ 
(Note that $t_{jk}^a=t_{jk}^{|a|}$ and $L_s^a=L_s^{|a|} (=L_s^{\otimes |a|})$ hold for any $a\in\mathbb{Z}^r$). 
As $L_s^m$ is holomorphically isomorphic to the unitary flat line bundle which corresponds to the unitary representation $\pi_1(Y_0, *)\to \mathrm{U}(1)$ defined by 
$\gamma_1\mapsto \exp(2\pi\sqrt{-1}(mq(s)))$ and 
$\gamma_2\mapsto \exp(2\pi\sqrt{-1}(mp(s)))$, 
one have that it is holomorphically trivial if and only if both $mp(s)$ and $mq(s)$ are integers, since a unitary flat line bundle on a compact complex manifold is trivial as a unitary flat line bundle if and only if it is holomorphically trivial (see also \cite[Proposition 2.2]{K8}). 
Therefore one have that $M_0mq(s)$ is an integer which is a multiple of $m_*$ in this case. 
As it follows by definition that $M_0q_0$ is also an integer which is a multiple of $m_*$, we have that either $q(s)=q_0$ or $|mM_0q(s)-mM_0q_0|\geq m_*$ holds if $L_s^m$ is holomorphically trivial. 
The assertion $(ii)$ follows from this, since 
\[
|mM_0q(s)-mM_0q_0|
= \frac{mM_0}{{\rm Im}\,\tau}\cdot (|q(s)-q_0|\cdot {\rm Im}\,\tau)<
mM_0 m_*\cdot d(s, \mathrm{U}(1))<m_*
\]
holds for any $s\in \widehat{S}_m$ (Note that $d(s, \mathrm{U}(1))=|{\rm Im}(p(s)+(q(s)-q_0)\tau)|$ by definition). 

Finally, we show the assertion $(iii)$. 
Let $m$ be a positive integer. 
Take an element $s\in \widehat{S}_m$. 
As  $t_{jk}(s)$ is either $1$, $\exp(\pm2\pi\sqrt{-1}q_0)$, $\exp(\pm 2\pi\sqrt{-1}(p(s)+(q(s)-q_0)\cdot \tau))$, or $\exp(\pm2\pi\sqrt{-1}q_0)\cdot \exp(\pm 2\pi\sqrt{-1}(p(s)+(q(s)-q_0)\cdot \tau))$, one have that 
\[
|t_{jk}(s)|^m\leq
\max_\pm|\exp(\pm 2m\pi\sqrt{-1}(q(s)-q_0)\cdot \tau)|
\leq \exp(2m\pi |q(s)-q_0|\cdot {\rm Im}\,\tau)). 
\]
As $|q(s)-q_0|\cdot {\rm Im}\,\tau\leq (M_0m)^{-1}$ holds for any $s\in \widehat{S}_m$, 
one have the assertion by letting $\Theta:=\exp\left(2\pi/M_0\right)$.

\subsubsection{Assumption $4$}

Take a positive integer $m$, $\nu\in \{0, 1, 2, \dots, M_0 m-1\}$, and an element $s\in W_{m, \nu}$. 
When $s\in \mathrm{U}(1)$, 
it follows by \cite[Lemma 4]{U} that there exists a constant $K$ which depends only on the choice of $Y$ and $\{U_j\}$ such that the inequality 
\[
d(1, s^m)\cdot\max_{j}\sup_{U_{j, s}}|\beta_{j, s}|
\leq K\cdot \max_{j, k}\sup_{U_{jk, s}}|\alpha_{jk, s}|
\]
holds for 
any multi-index $a$ with $|a|=m$, 
any $1$-cochain $\alpha=\{(U_{jk, s}, \alpha_{jk, s}dw_j^a)\}\in \check{Z}^1(\{U_{jk, s}\}, \mathcal{O}_{Y_s}(L_s^m))$ and 
any $0$-cochain $\beta=\{(U_{j, s}, \beta_{j, s}dw_j^a)\}\in \check{C}^0(\{U_{j, s}\}, \mathcal{O}_{Y_s}(L_s^m))$ such that 
$\alpha=\delta \beta$, 
since each $t_{jk}(s)$ is an element of $\mathrm{U}(1)$ by {\bf Assumption $3$} $(iii)$ (Note that such $\beta$ is unique when $d(1, s^m)\not=0$ by {\bf Assumption $3$} $(ii)$, since $L_s^m$ is not holomorphically trivial in such case). 
As $md_{a, \nu}(s)\leq = d(1, s^m)$, one have the assertion of {\bf Assumption $4$} for such $s$. 

In what follows, 
we investigate the case where $s$ need not to be an element of $\mathrm{U}(1)$ by fixing such an element $s\in W_{m, \nu}$. 
Take a local trivialization $e_j$ of $L_s$ one each $U_{j, s}$ such that 
$e_j=t_{jk}(s)\cdot e_k$ holds on each $U_{jk, s}$. 

As $L_s$ is isomorphic to the line bundle which corresponds to the unitary representation 
$\pi_1(Y^0, *)\to \mathrm{U}(1)$ generated by 
$\gamma_1\mapsto \exp(2\pi\sqrt{-1}\cdot(-q(s)))$ and 
$\gamma_2\mapsto \exp(2\pi\sqrt{-1}p(s))$, 
it follows that there exists another local trivialization $\widehat{e}_j$ of $L_s$ on each $U_{j, s}$ such that $\widehat{e}_j=\widehat{t}_{jk}\cdot \widehat{e}_k$ holds on each $U_{jk, s}$, 
where $\widehat{t}_{jk}$ is a constant function valued in $\mathrm{U}(1)$. 
Note that such $\widehat{e}_j$'s are constructed by considering the function 
$\sigma\colon z\mapsto \exp(2\pi\sqrt{-1}\cdot (q(s)-q_0)z)$ on the universal covering of $Y$. 
Therefore we may assume that the ratio $\sigma_j:=\widehat{e}_j/e_j$ coincides with this function $\sigma$ restricted to a suitable open subset of $\mathbb{C}$ via the covering map. 
As $|\sigma(z)|=\exp(-2\pi(q(s)-q_0){\rm Im}\,z)$, 
one have the existence of a constant $M_1>1$ which only depends on $M_0$ and the manner how to choose domains of the universal covering of $Y_s$ such that the restriction of $\sigma$ to it coincides with $\sigma_j$ (and thus depends on neither $m$ nor $s\in \widehat{S}_m$) such that 
\[
M_1^{-\frac{1}{m}} \leq \inf_{U_{j, s}}|\sigma_j(z)|\leq \max_{U_{j, s}}|\sigma_j(z)| \leq M_1^{\frac{1}{m}}
\]
holds for each $j$. 
Again by the inequality $md_{a, \nu}(s)\leq = d(1, s^m)$, it is sufficient to show the following: 

\begin{lemma}\label{lem:ueda_4_new}
Tere exists a constant $K$ which depends on neigher  $m$ nor $s\in \overline{S_m}$ such that, 
for any $1$-cochain $\alpha=\{(U_{jk, s}, \alpha_{jk, s}e_j^m)\}\in \check{Z}^1(\{U_{jk, s}\}, \mathcal{O}_{Y_s}(L_s^m))$ and 
any $0$-cochain $\beta=\{(U_{j, s}, \beta_{j, s}e_j^m)\}\in \check{C}^0(\{U_{j, s}\}, \mathcal{O}_{Y_s}(L_s^m))$ with 
$\alpha=\delta \beta:=\{(U_{jk, s}, (-\beta_{j, s}+t_{kj}^m(s)\cdot \beta_{k, s})e_j^m)\}\in \check{Z}^1(\{U_{jk, s}\}, \mathcal{O}_{Y_s}(L_s^m))$, it holds that 
\[
m\cdot \min_{a, b\in\frac{1}{m}\mathbb{Z}}|(p(s)+q(s)\tau)-(a+ b\tau)|\cdot \|\beta\|_{s, m}\leq K\cdot \|\alpha\|_{s, m}, 
\]
where
$\|\alpha\|_{s, m}:=\max_{j, k}\sup_{U_{jk, s}}|\alpha_{jk, s}|$ and 
$\|\beta\|_{s, m}:=\max_{j}\sup_{U_{j, s}}|\beta_{j, s}|$. 
\end{lemma}

\begin{proof}
One may assume that $L_s^m$ is not holomorphically trivial, since otherwise both $mp(s)$ and $mq(s)$ are integers by {\bf Assumption} $3$ $(ii)$. 
Note that $H^0(\{U_{j, s}\}, \mathcal{O}_{Y_s}(L_s^m))=H^1(\{U_{j, s}\}, \mathcal{O}_{Y_s}(L_s^m))=0$ in this case, which meas that $\beta$ as in the assertion is uniquely exists for any $\alpha$. 

Take a $1$-cocycle $\alpha=\{(U_{jk, s}, \alpha_{jk, s}e_j^m)\}\in \check{Z}^1(\{U_{j, s}\}, \mathcal{O}_{Y_s}(L_s^m))$ such that $[\alpha]=0\in \check{H}^1(\{U_{j, s}\}, \mathcal{O}_{Y_s}(L_s^m))$. 
By letting $\widehat{\alpha}_{jk, s}:=\sigma_j^{-m}\cdot \alpha_{jk, s}$, 
one have that 
$\widehat{\alpha}_{jk, s}\widehat{e}_j^m=\alpha_{jk, s}e_j^m$. 
Take a primitive 
$\beta= \{(U_{j, s}, \beta_{j, s}e_j^m)\}=\{(U_{j, s}, \widehat{\beta}_{j, s}\widehat{e}_j^m)\}$, which uniquely exists as we mentioned above. 
Note that $\widehat{\beta}_{j, s}\cdot \sigma_j^m=\beta_{j, s}$ for each $j$. 
It follows from 
\cite[Lemma 4]{U} that there exists a constant $K_1$ which depends only on $Y$ and $\{U_j\}$ and on neither $s$, $m$, $\alpha$, nor $\beta$ such that 
\[
\max_j\sup_{U_{j, s}}|\widehat{\beta}_{j, s}| \leq \frac{K_1}{d_{\Sigma}(\mathbb{I}_s, L_s^m)}\cdot\max_{j, k}\sup_{U_{jk, s}}|\widehat{\alpha}_{jk, s}|, 
\]
where $d_{\Sigma}$ is an invariant distance on $\Sigma$. 
As all the invariant distances on $\Sigma$ are equivalent to each other, one can use our distance $d$ on $Y$ as $d_\Sigma$ via the identification. 
As 
\[
\max_{j, k}\sup_{U_{jk, s}}|\widehat{\alpha}_{jk, s}|
=\max_{j, k}\sup_{U_{jk, s}}\left(|\alpha_{jk, s}|\cdot |\sigma_j|^{-m}\right)
\leq M_1\cdot \max_{j, k}\sup_{U_{jk, s}}|\alpha_{jk, s}|
\]
and 
\[
\max_{j}\sup_{U_{j, s}}|\beta_{j, s}|
=\max_{j}\sup_{U_{j, s}}\left(|\widehat{\beta}_{j, s}|\cdot |\sigma_j|^m\right)
\leq M_1\cdot \max_{j}\sup_{U_{j, s}}|\widehat{\beta}_{j, s}|
\]
hold, one have the lemma by letting $K:=K_1\cdot M_1^2$.  
\end{proof}

\subsubsection{Assumption $5$}

For any non-torsion element 
$s\in \Sigma$, one have that the pair $(Y_s, X_s)$ is of infinite type, since it follows from {\rm Assumption} $1$ that $H^1(Y_s, \mathcal{O}_{Y_s}(N_s\otimes S^{m+1}N_s^*))=0$ for any $m\geq 1$. 
When $s\in\mathrm{U}(1)=\ell_0$ is a torsion element, 
consider the foliation $\mathcal{F}_s$ on $X_s$ defined by $b_s$ as in Assumption $5$'. 
Then, by a simple observation, it follows that the holonomy of $\mathcal{F}_s$ along a leaf $Y_s$ is a diagonal matrix ${\rm diag}(t_1, t_2, \dots, t_r)$ for $m_\lambda$-th roots of the unity $t_\lambda$'s. 
This means that, for a suitable foliation chart 
$(z_j, \widehat{w}_j^{(1)}, \widehat{w}_j^{(2)}, \dots, \widehat{w}_j^{(r)})$ on a neighborhood of each $U_{j, s}$, the transition functions of $(\widehat{w}_j^{(1)}, \widehat{w}_j^{(2)}, \dots, \widehat{w}_j^{(r)})$ are unitary, which shows the assertion of {\rm Assumption} $5$. 

\subsection{A configuration in which $Y_s$'s are cycles of rational curves}\label{section:setting_nodal}

Let $Z$ be a smooth complex surface and $Y^0\subset Z$ be a holomorphically embedded cycle of $N$ rational curves (i.e. $Y^0$ is a reduced subvariety of $X$ with only nodal singularities whose normalization consists of $N$ copies of $\mathbb{P}^1$ such taht the dual graph is a cycle graph). 
Denote by $i\colon C\to Y^0$  the normalization, 
and by $C_1, C_2, \dots, C_N$ the irreducible components of $C$. 
We may assume that $i(C_\nu)\cap i(C_\mu)\not=\emptyset$ if $(\nu, \mu)=(1, 2), (2, 3), \dots, (N-1, N)$ or $(N, 1)$. 
Let $p_0$ be the intersection $i(C_N)\cap i(C_1)$, and 
$p_\nu$ be the intersection $i(C_\nu)\cap i(C_{\nu+1})$ for $\nu=1, 2, \dots, N-1$ when $N>1$. 
In this case, we often identify each $C_\nu$ with the image $i(C_\nu)$, and regard each $p_\nu$ also as a point of $C_{\nu}$ or $C_{\nu-1}$ for $\nu=0, 1, \dots, N-1$ ($C_0:=C_N$). 
When $N=1$, we denote by $\{p_0, p_1\}$ the preimage of the unique nodal point of $Y^0$ by $i$. 
We sometimes denotes the unique nodal point also by $p_0$ in this case. 

Denote by $L$ the normal bundle $N_{Y^0/Z}=[Y^0]|_{Y^0}$. 
Assume that the intersection number $(i^*L. C_\nu)$ is equal to $1$ for $\nu=1$ and to $0$ for $\nu=2, 3, \dots, N$. 

Set $\Sigma:=C_1\setminus \{p_0, p_1\}$. 
In this subsection, we identify $\Sigma$ with $\mathbb{C}^*$ by using the non-homogeneous coordinate $s$ of $C_1$ such that $s$ maps $p_0$ to $0\in\mathbb{P}^1$, $p_1$ to $\infty\in \mathbb{P}^1$, 
and the unique zero of a non section $f_L$ of $H^0(Y^0, \mathcal{O}_{Y^0}(L))$ to $1\in \mathbb{P}^1$. 
Note that it follows by a standard argument that $H^0(Y^0, \mathcal{O}_{Y^0}(L))\cong\mathbb{C}, H^1(Y^0, \mathcal{O}_{Y^0}(L))=0$, 
and that, for an element $f\in H^0(Y^0, \mathcal{O}_{Y^0}(L))$, the following three conditions are equivalent to each other: $i^*f(p_0)=0$, $i^*f(p_1)=0$, and $f\equiv 0$. 

We use a finite open covering $\{U_j\}$ of $Y^0$ such that $\bigcup_{U_j\cap Y_{\rm sing}=\emptyset}U_j=Y\setminus Y_{\rm sing}$ and that, for each $p_\nu \in Y^0_{\rm sing}$, it holds that $\#\{k\mid p_\nu\in U_k\}=1$. 
We denote by $U_{(N1)}$ the unique element of $\{U_j\}$ which includes $p_0$, 
by $U_{(\nu, \nu+1)}$ the unique element of $\{U_j\}$ which includes $p_\nu$ for each $\nu=1, 2, \dots, N-1$, 
and by $U_{(N1)}^\pm$ the irreducible components of $U_{(N1)}$. 
We may assume that $U_{(N1)}^+\cap \{s\in \Sigma\mid |s|<1\}\not=\emptyset$. 
We also assume that $U_j\cap Y^0_{\rm sing}=\emptyset$ holds if $U_k\cap Y^0_{\rm sing}\not=\emptyset$ and $U_j\cap U_k\not=\emptyset$. 

According to the proof of \cite[Lemma 2.1]{K6} and \cite[Remark 2.2]{K6}, the map 
\[
\Sigma\ni s\mapsto \mathcal{O}_{Y^0}(L)\otimes\mathcal{O}_{Y^0}(-s)\in{\rm Pic}^0(Y^0)
\]
is an isomorphism, by which we will identify $\Sigma$ with ${\rm Pic}^0(Y^0)$ in this subsection. 
Let us see how are the transition functions of $\mathcal{O}_{Y^0}(L)\otimes\mathcal{O}_{Y^0}(-s_0)$ for $s_0\in \Sigma$. 
Denote by $F_{s_0}\in H^0(C, \mathcal{O}_{C}(i^*L)\otimes \mathcal{O}_C(\{1\}-\{s_0\}))$ the section obtained by tensoring $i^*f_L$ and the element of $H^0(C, \mathcal{O}_C(\{1\}-\{s_0\}))$ which coincides with the rational map $s\mapsto \frac{s_0-s}{1-s}$ on $C_1$. 
Then it holds that $F_{s_0}(p_0)=s_0\cdot F_{s_0}(p_1)$. 
Denote by $G\in H^0(U_{(N1)}, \mathcal{O}_{Y^0}(L)\otimes \mathcal{O}_{Y^0}(-i(s)))$ the element defined by 
\[
G:=\begin{cases}
F_{s_0} & \text{on}\ U_{(N1)}^+ \\
s_0\cdot F_{s_0}  & \text{on}\ U_{(N1)}^-
\end{cases}.
\]
By using $G$ as a local frame on $U_{(N1)}$ and a section $e_j$ such that $i^*e_j=F_{s_0}|_{U_j}$ as that on $U_j$ for $j\not=(N1)$, one have that there exists a holomorphic function $t_{jk}$ on $\Sigma$ for each $j$ and $k$ such that the element of ${\rm Pic}^0(Y^0)$ which corresponds to $s_0\in \Sigma$ coincides with $[\{(U_{jk}, (t_{jk}(s_0))^{-1})\}]$. 
Note that $t_{jk}(s)$ is either $s^{\pm 1}$ or $1$. 

Let $\overline{\mathcal{X}}$ be the blow-up of $Z\times \Sigma$ along the subvariety 
\[
\{(z, s)\in Y^0\times \Sigma\mid z\ \text{coincides with}\ s\ \text{via the identification above}\}
\subset Z\times \Sigma, 
\]
and $\pi\colon \overline{\mathcal{X}}\to \Sigma$ be the morphism obtained by composing the blow-up morphism and the second projection $Z\times \Sigma\to \Sigma$. 
Note that $\pi$ is a surjective holomorphic submersion, 
and that the fiber $X_s:=\pi^{-1}(s)$ is the blow-up of $Z$ at $s\in Y^0\subset Z$ for each $s\in \Sigma$. 
Denote by $\overline{\mathcal{Y}}$ the strict transform of  $Y^0\times \Sigma$, 
by $Y_s$ the intersection $X_s\cap \overline{\mathcal{Y}}$, 
and by $N_s$ the the line bundle $\mathcal{O}_{Y_s}(L)\otimes \mathcal{O}_{Y_s}(-s)$ on $Y_s$ for each $s\in \Sigma$, where we are regarding $L$ as a line bundle on $Y_s$ by the natural identification of $Y_s$ with $Y^0$. 
Note that $N_{Y_s/X_s}\cong N_s$. 

In what follows, we identify $\Sigma$ with $\mathbb{C}^*$ by using the coordinate $s$, 
and will use the distance $d$ of $\Sigma$ which is the restriction of that of $\mathbb{C}^*$ induced by the Euclidean distance of the universal covering 
$\mathbb{C}\ni \xi\mapsto \exp(2\pi\sqrt{-1}\xi)\in \Sigma$. 
Denote by $S$ the subset defined by $S:=\{s\in \Sigma\mid d(\mathrm{U}(1), s)<2\} (=\{\exp(2\pi\sqrt{-1}\xi)\in \Sigma\mid |{\rm Im}\,\xi|<2\})$. 
We define the total space $\mathcal{X}$ by $\mathcal{X}:=\pi^{-1}(S)\subset \overline{\mathcal{X}}$, 
and the subvariety $\mathcal{Y}$ by $\mathcal{Y}:=\overline{\mathcal{Y}}\cap \mathcal{X}$. 
Denote the restriction $\pi|_{\mathcal{X}}$ simply by the same letter $\pi$. 
Let $\{U_j\}, \{\mathcal{U}_j\}$, and $\{U_{j, s}\}$ be those as in \S \ref{section:general}. 
In the rest of this subsection, we will show that $\pi\colon\mathcal{X}\to S$ satisfies {\bf Assumption} $1, 2, 3, 4$, and $5$ under the following: 

\begin{description}
\item[Assumption $5$''] Assume that an element $s\in \Sigma$ is torsion (i.e. $s^m=1$ for some positive integer $m$). 
Then there exists a neighborhood $V_s$ of $Y_s$ in $X_s$ 
and a surjective proper holomorphic map $b_s\colon V_s\to B_s$ onto a neighborhood $B_s$ of the origin of $\mathbb{C}$ 
such that $b_s^*\{0\}=m\cdot Y_s$ holds as divisors for some positive constant $m$. 
\qed
\end{description}

\subsubsection{Assumption $1$}

One can easily show that the assertion of Assumption $1$ holds by the calculation as in the proof of \cite[Proposition 2.5]{K6}. 

\subsubsection{Assumption $2$}

First we construct $\mathcal{V}_k$ for $k=(N1), (12), \dots, (N, N-1)$. 
Denote by $\overline{\mathcal{U}}_k\subset \overline{\mathcal{X}}$ the strict transform of $U_k\times \Sigma$. 
Let $\overline{\mathcal{U}}_k^\pm$ be each of the irreducible component of $\overline{\mathcal{U}}_k$. 
As each $\overline{\mathcal{U}}_k^+$ is stein, it follows from \cite{Siu} that there exists a Stein neighborhood $\mathcal{W}_k$ of $\overline{\mathcal{U}}_k^+$ in $\overline{\mathcal{X}}$. 
Take a defining function $y_k$ of $\overline{\mathcal{U}}_k^+$ on $\mathcal{W}_k$, 
and $x_k$ of $\overline{\mathcal{U}}_k^-\cap \mathcal{W}_k$ on $\mathcal{W}_k$. 
Set $w_k:=x_k\cdot y_k$. 
Then, as $S\Subset \Sigma$, the assertion $(ii), (iv)$, and $(vi)$ of {\bf Assumption $2$} holds for each $k=(N1), (12), \dots, (N, N-1)$ by shrinking $U_k$'s and $\mathcal{W}_k$'s and by letting $\mathcal{V}_k:=\mathcal{W}_k\cap \mathcal{X}$. 
Note that $\{U_j\}$ is still an open covering of $Y^0$ after shrinking $U_k$'s in such manner, since it holds that $\bigcup_{U_j\cap Y_{\rm sing}=\emptyset}U_j=Y\setminus Y_{\rm sing}$. 

For $j$ such that $U_j\cap Y_{\rm sing}=\emptyset$, we define $\mathcal{V}_j$ by using \cite[Corollary 1]{Siu} in the same manner as in \S \ref{section:setting_smooth}. 
Then, by the same argument as in , one have that the assertion $(i), (ii), (iii)$, and $(v)$ of {\bf Assumption $2$} also holds by letting $t_{jk}^{(1)}:=t_{jk}$ (Note that $N_{\mathcal{Y}/\mathcal{X}}^{-1}=[\{\mathcal{U}_{jk}, t_{jk}(s)\}]$). 

\subsubsection{Assumption $3$}

Here we show the assertion {\bf Assumption $3$} holds by letting $M_0:=1$. 
First let us note that the assertion $(i)$ clearly holds. 

It follows from a calculation by using a long exact sequence as in \cite[Proposition 2.5]{K6} that, for $s\in S$, 
$H^0(Y_s, \mathcal{O}_{Y_s}(N_s))=0$ holds if and only if $s=1$. 
Thus one have that the set 
$\{s\in S\mid N_s^m\cong \mathbb{I}_s \}$ coincides with $\{s\in {\rm U}(1)\mid s^m=1\}$, which proves the assertion $(ii)$. 

Denote by $p(s)$ and $q(s)$ the real numbers such that $\exp(2\pi\sqrt{-1}(p(s)+q(s)\sqrt{-1})) = s$ for each $s\in S$. 
Note that $p(s)$ is determined modulo $\mathbb{Z}$. 
Then, as
$s\in \widehat{S}_m$ if and only if $|q(s)|<1/m$, 
one have that 
$\max_{j, k}\sup_{s\in \widehat{S}_m}|t_{jk}(s)^m|
\leq \exp(2\pi/m)^m=\exp(2\pi)$, since $t_{jk}(s)$ is either $1$ or $s^{\pm 1}$. 
Thus the assertion $(iii)$ holds by letting $\Theta:=\exp(2\pi)$.

\subsubsection{Assumption $4$}

Take a local frame $e_j$ of $N_{\mathcal{Y}/\mathcal{X}}^{-1}$ on each $\mathcal{U}_j$ such that $e_j=t_{jk}e_k$ holds on each $\mathcal{U}_{jk}$ (or just let $e_j:=dw_j|_{\mathcal{U}_j}$). 
It is sufficient to show the following: 

\begin{lemma}\label{lem:ueda_4_new_nodal}
There exists a constant $K$ such that the following holds for any positive integer $m$ and any element $s\in \widehat{S}_m$: 
for a $1$-cochain $\alpha=\{(U_{jk, s}, \alpha_{jk, s}e_j^m)\}\in \check{Z}^1(\{U_{jk, s}\}, \mathcal{O}_{Y_s}(N_s^{-m}))$ and 
a $0$-cochain $\beta=\{(U_{j, s}, \beta_{j, s}e_j^m)\}\in \check{C}^0(\{U_{j, s}\}, \mathcal{O}_{Y_s}(N_s^{-m}))$ with 
$\alpha=\delta \beta:=\{(U_{jk, s}, (-\beta_{j, s}+t_{kj}^m(s)\cdot \beta_{k, s})e_j^m)\}$, 
it holds that
\[
m\cdot \min_{a\in\frac{1}{m}\mathbb{Z}}|(p(s)+q(s)\sqrt{-1})-a|\cdot \|\beta\|_{s, m}\leq K\cdot \|\alpha\|_{s, m}, 
\]
where 
$\|\alpha\|_{s, m}:=\max_{j, k}\sup_{U_{jk, s}}|\alpha_{jk, s}|$ 
and 
$\|\beta\|_{s, m}:=\max_{j}\sup_{U_{j, s}}|\beta_{j, s}|$. 
\qed
\end{lemma}

\begin{proof}
It is sufficient to show the lemma only when $s^m\not=1$. 

Let $J\colon \widetilde{C}\to Y^0$ be the morphism obtained by considering the normalization only at the singular point $p_0$. 
Note that $\widetilde{C}$ is a connected variety whose dual graph is tree, and that $J^{-1}(p_0)$ consists of two points, say $p_0^\pm$. 
Denote by $\widetilde{U}_j$ the preimage $J^{-1}(U_j)$ for each $j\not=(N1)$, 
and by $\widetilde{U}_{(N1\pm)}$ the neighborhood of $p_0^\pm$, respectively, 
such that $J^{-1}(U_{(N1)})=\widetilde{U}_{(N1-)}\cup \widetilde{U}_{(N1+)}$ holds. 
We regard 
$\{U_j\}_{j\not=(N1)}\cup\{\widetilde{U}_{(N1+)}, \widetilde{U}_{(N1-)}\}$ as an open covering of $\widetilde{C}$. 

Define a $0$-cochain $\widehat{\beta}=\{(\widetilde{U}_j, \widehat{\beta}_j)\}\in
\check{C}^0(\{\widetilde{U}_j\}, \mathcal{O}_{\widetilde{C}}(J^*N_s^{-m}))$ by 
$\widehat{\beta}_{(N1+)}:=J^*\beta_{(N1)}$, 
$\widehat{\beta}_{(N1-)}:=J^*\beta_{(N1)}\cdot s^m$, and 
by $\widehat{\beta}_j:=J^*\beta_j$ for the other $j$'s, where we are regarding $N_s$ as a line bundle on $Y^0$ by using the natural identification of $Y^0$ and $Y_s$. 
As all the transition functions of $J^*N_s$ are trivial, one have that 
$\max_{j, k}\sup_{\widetilde{U}_{jk}}|\widehat{\alpha}_{jk}|=\|\alpha\|_{s, m}$ 
holds, where we are denoting $\delta\widehat{\beta}_j (=J^*\alpha)$ by $\widehat{\alpha}=\{(\widetilde{U}_{jk}, \widehat{\alpha}_{jk})\}$. 
Again by the fact that all the transition functions of $J^*N_s$ are trivial, 
it follows from \cite[Lemma 4.1]{K6} that 
there exists an element $\widehat{\beta}'=\{(\widetilde{U}_j, \widehat{\beta}_j')\}\in
\check{C}^0(\{\widetilde{U}_j\}, \mathcal{O}_{\widetilde{C}}(J^*N_s^{-m}))$ such that 
$\delta\widehat{\beta}'=\widehat{\alpha}$ and 
$\max_{j}\sup_{\widetilde{U}_{j}}|\widehat{\beta}_{j}'|\leq K_0\|\alpha\|_{s, m}$ 
holds for some constant $K_0$ which depends only on $Y^0$ and $\{U_j\}$ and on neither $s$ nor $m$. 
Note that, as is clearly followed by the compactness of $Y^0$, one can take a constant $\ell\in\mathbb{C}$ such that $\widehat{\beta}_{j}=\widehat{\beta}_{j}'+\ell$. 
As it holds for $a_\pm := \widehat{\beta}'_{(N1\pm)}(p_0^\pm)$ that $a_++\ell=s^{-m}(a_-+\ell)$, one have that 
$\ell=\frac{s^ma_+-a_-}{1-s^m}$. 
Therefore, by 
\[
\|\beta\|_s\leq \max\{1, |s|^m\}\cdot \left(\max_{j}\sup_{\widetilde{U}_{j}}|\widehat{\beta}'_{j}|+\frac{|s^ma_+-a_-|}{|1-s^m|}\right)
\leq e^{2\pi}\left(K_0\|\alpha\|_{s, m}+\frac{2e^{2\pi}K_0\|\alpha\|_{s, m}}{|1-s^m|}\right), 
\]
one have the inequality $|1-s^m|\cdot \|\beta\|_{s, m}\leq K_1\cdot \|\alpha\|_{s, m}$ holds for a constant 
$K_1:=e^{2\pi}K_0\cdot (1+e^{2\pi}+2e^{2\pi})$. 
The lemma follows form this, since $d$ and the distance of $S$ induced by restricting the Euclidean distance of $\mathbb{C} (\supset \mathbb{C}^*\supset S)$ are equivalent. 
\end{proof}

\subsubsection{Assumption $5$}

Take $s\in \Sigma$. 
When $s$ is non-torsion, it follows from {\bf Assumption $1$} that the pair $(Y_s, X_s)$ is of infinite type, since $H^1(Y_s, \mathcal{O}_{Y_s}(N_s^{-m}))=0$ holds for any $m\geq 1$. 
When $s$ is non-torsion, consider the fibration $b_s\colon V_s\to B_s$ as in 
Assumption $5$''. 
Then one can construct a system of local defining functions $\{(V_{j, s}\cap V_s, w_{j, s})\}$ of $Y_s$ by considering $m$-th root of $b_s^*w|_{V_{j, s}\cap V_s}$, where $w$ is the coordinate of $B_s$. 
As clearly the ratio $w_{j, s}/w_{k, s}$ is a constant function whose value is a $m$-th root of the unity for each $j$ and $k$, one have that the assertion {\bf Assumption $5$} holds. 


\section{Proof of Theorem \ref{thm:main_general}}\label{section:prf_main_thm}

\subsection{Outline of the proof}\label{section:proof_outline}

In this section, for proving Theorem \ref{thm:main_general}, 
we will try to construct a new system of defining functions 
\[
\widehat{w}_{j}=\left(
    \begin{array}{c}
      \widehat{w}_{j}^{(1)} \\
      \widehat{w}_{j}^{(2)} \\
      \vdots \\
      \widehat{w}_{j}^{(r)}
    \end{array}
  \right)
\] 
by modifying $w_j$'s as in {\bf Assumption $2$}. 
Modification is done in the following manner: 
construct a suitable holomorphic function $F^{(\lambda)}_{j, a}$ for each $\lambda=1, 2, \dots, r$ and for each multi-index $a\in (\mathbb{Z}_{\geq 0})^r$ with $|a|\geq 2$, and solve the functional equation 
\begin{equation}\label{eq:def_w_hat}
w_j^{(\lambda)}=\widehat{w}_j^{(\lambda)}+\sum_{|a|\geq 2}F^{(\lambda)}_{j, a}\cdot \widehat{w}_j^a, 
\end{equation}
where 
$\widehat{w}_j^a:=\prod_{\lambda=1}^r(\widehat{w}_j^{(\lambda)})^{a_\nu}$. 

We will construct the function $F^{(\lambda)}_{j, a}$ suitably as a function defined on $p^{-1}(Y\times S_{|a|-1})$ inductively on $|a|$, and extend it to a holomorphic function on a neighborhood of $p^{-1}(Y\times S_{|a|-1})$ in the manner we will explain in Remark \ref{rmk:rule_extension} below. 

\begin{remark}\label{rmk:rule_extension}
Here we explain our rule in this section how to extend a function $F=F^{(\lambda)}_{j, a}$ defined on $S_{|a|-1}$ to its neighborhood. 
For simplicity, we will explain on each $U_{j, s}$ ($s\in S_{|a|-1}$). 
Let $F$ be a holomorphic function defined on $U_{j, s}$. 
When $U_j$ is non-singular, we use the pull-back $({\rm Pr}_{\mathcal{U}_j}|_{V_{j, s}})^*F=F\circ {\rm Pr}_{\mathcal{U}_j}|_{V_{j, s}}$, which will be denoted by same letter $F$. 
For singular $U_k$, denote by $p_k$ its singular point, and 
by $U_{k, s}^+$ and $U_{k, s}^-$ the irreducible components $\{(x_k, y_k)\in V_{j, s}\mid y_k=0\}$ and $\{(x_k, y_k)\in V_{j, s}\mid x_k=0\}$, respectively, of
$U_{k, s}$. 
Letting $c:=F(p_k)$, it is easily observed that there uniquely exist holomorphic functions 
$F^\pm$ on $U_{k, s}^\pm$ such that $F^\pm(p_k)=0$ and 
\[
F-c=\begin{cases}
F^+ & \text{on}\ U_{k, s}^+\\
F^- & \text{on}\ U_{k, s}^-
\end{cases}
\]
hold. 
In this case, we use the function $\widetilde{F}$ on  $V_{j, s}$ defined by $\widetilde{F}(x_k, y_k):=F^+(x_k)+F^-(y_k)+c$ as an extension of $F$. 
In what follows, we often denote $\widetilde{F}$ simply by the same letter $F$. 
Note that $\sup_{V_{j, s}}|F|\leq 3\sup_{U{j, s}}|F|$ holds in both of the cases. 
\qed
\end{remark}

The functions $\{F^{(\lambda)}_{j, a}\}_{|a|= m}$ will be constructed so that they enjoy the following {\bf (Property)$_m$} inductively when $\{F^{(\lambda)}_{j, a}\}_{|a|< m}$ have already been constructed in the manner that each of them enjoys {\bf (Property)$_{|a|}$}. 

\begin{description}
\item[(Property)$_m$] For any choice of the remaining $\{F^{(\lambda)}_{j, a}\}_{|a|>m}$, the formal solution $\widehat{w}_j$ of the functional equation $(\ref{eq:def_w_hat})$ satisfies 
$T_{jk}\widehat{w}_k =\widehat{w}_j + O(|\widehat{w}_j|^{m+1})$ as elements of $\mathcal{O}_{\mathcal{Y}}(p^{-1}(U_{jk}\times S_{m-1}))[[\widehat{w}_j^{(1)}, \widehat{w}_j^{(2)}, \dots, \widehat{w}_j^{(r)}]]$ for each $j$ and $k$ with $U_{jk}\not=\emptyset$ and $U_j\cap Y_{\rm sing}=\emptyset$. \qed
\end{description}

We will describe how to construct $\{F^{(\lambda)}_{j, a}\}$ in \S \ref{section:constr_Fs}. 
We here remark that we never shrink $\mathcal{V}_j$ and $\mathcal{U}_j$ anymore in \S \ref{section:constr_Fs}. 
After finishing the construction of them, 
we will estimate each $|F^{(\lambda)}_{j, a}|$ suitably on each $U_{j, s}$ with $s\in\mathrm{U}(1)$ in \S \ref{section:estim_Fs} so that one can regard the right hand side of the functional equation $(\ref{eq:def_w_hat})$ as a convergent series in a suitable sense. 

In the rest of this subsection, we will explain how to solve the functional equation $(\ref{eq:def_w_hat})$ after once the construction and the estimate of $F^{(\lambda)}_{j, a}$'s are finished. 
Take $s\in\mathrm{U}(1)$. 
We will construct a solution $\widehat{w}_j$ of the functional equation on each $V_{j, s}$ 
by shrinking $V_{j, s}$ to a smaller neighborhood of $U_{j, s}$ if $U_j\cap Y_{\rm sing}=\emptyset$, and $V_{k, s}$ to a smaller neighborhood of the nodal point if $U_j\cap Y_{\rm sing}\not=\emptyset$. 
Note that $\{V_{j, s}\}$ is an open covering of $Y_s$ even after such shrinking, 
since it holds that $\bigcup_{U_j\cap Y_{\rm sing}=\emptyset}U_j=Y\setminus Y_{\rm sing}$. 

First let us consider on $V_{j, s}$ for $j$ such that 
$Y_{\rm sing}\cap U_j=\emptyset$. 
It follows by {\rm Assumption $2$} $(v)$ that one can embed $V_{j, s}$ into $U_{j, s}\times \mathbb{C}^r$ by regarding $w_j$ as the coordinate of $\mathbb{C}^r$. 
For each $\lambda=1, 2, \dots, r$, 
it will be turned out by the estimate we will make in \S \ref{section:estim_Fs} that 
\[
G(z_j, w_j, \widehat{w}_j):=
-w_j+\widehat{w}_j+\sum_{|a|\geq 2}
\left(
    \begin{array}{c}
    F^{(1)}_{j, a}(z_j, s)\\
    F^{(2)}_{j, a}(z_j, s)\\
      \vdots \\
    F^{(r)}_{j, a}(z_j, s)
    \end{array}
  \right)
\cdot \widehat{w}_j^a. 
\]
can be regarded as a $\mathbb{C}^r$-valued holomorphic function defined on 
a neighborhood of $U_{j, s}\times\{(0, 0)\}$ in $U_{j, s}\times \mathbb{C}^r\times \mathbb{C}^r$. 
As the Jacobian matrix $(\partial G/\partial \widehat{w}_j)$ is invertible on each point of $U_{j, s}\times\{(0, 0)\}$, one have by the implicit function theorem that 
there exists a holomorphic function $\widehat{w}_j(z_j, w_j)$ on a neighborhood of $U_{j, s}$ in $V_{j, s}$ such that $G(z_j, w_j, \widehat{w}_j(z_j, w_j)))\equiv 0$, which mean that $\widehat{w}_j$ is a solution of the functional equation $(\ref{eq:def_w_hat})$. 

Next, let us consider on $V_{k, s}$ for $k$ such that $Y_{\rm sing}\cap U_k\not=\emptyset$. 
According to {\rm Assumption $2$} $(iv)$, one may regard $V_{k, s}$ as a subset of $\mathbb{C}^2$. 
Denoting $w_k^{(1)}$ simply by $w_k$, 
consider 
\[
G(x_k, y_k, \widehat{x}_k):=
-x_k+\widehat{x}_k+\sum_{m=2}^\infty F^{(1)}_{k, m}(x_k, y_k, s)\cdot \widehat{x}_k^my_k^{m-1}. 
\]
By the estimate we will make in \S \ref{section:estim_Fs}, 
it will be turned out $G$ defines a holomorphic function on a neighborhood of the point $(0, 0, 0)$ in $V_{k, s}$. 
As 
$\frac{\partial}{\partial \widehat{x}_k}G(0, 0, 0)=1\not=0$, 
it follows from the implicit function theorem that 
there exists a holomorphic function $\widehat{x}_k=\widehat{x}_k(x_k, y_k)$ defined on a neighborhood of the nodal point in $V_{k, s}$ such that $G(x_k, y_k, \widehat{x}_k(x_k, y_k))\equiv 0$ holds, which means that $\widehat{w}_k:=\widehat{x}_ky_k$ is a solution of the functional equation $(\ref{eq:def_w_hat})$. 

\subsection{Construction of $F^{(\lambda)}_{j, a}$'s}\label{section:constr_Fs}

\subsubsection{Outline of the construction}\label{section:constr_Fs_outline}

Take $w_j$'s as in {\bf Assumption $2$}. 
Let 
\[
t_{jk}w_k^{(\lambda)} = w_j^{(\lambda)}+\sum_{|a|\geq 2} f^{(\lambda)}_{kj, a}(z_j, s)\cdot w_j^a
\]
be the expansion of $t_{jk}w_k^{(\lambda)}$ on $\mathcal{V}_{jk}$ for each $\lambda=1, 2, \dots, r$, where $a=(a_1, a_2, \dots, a_r)$ runs all the elements of $(\mathbb{Z}_{\geq 0})^r$ with $|a|\geq 2$. 
We always assume that $U_j$ is non-singular whenever we consider such kind of expansion (it may possible that $U_k$ is singular). 
We note that each coefficient functions are obtained by 
\[
f^{(\lambda)}_{kj, a}(z_j, s):=\frac{1}{|a|!}\left.\frac{\partial^{|a|}(t_{jk}w_k^{(\lambda)})}{\partial (w_j^{(1)})^{a_1}\partial (w_j^{(2)})^{a_2}\cdots \partial (w_j^{(r)})^{a_r}}\right|_{(w_j, z_j, s)=(0, 0, \dots, 0, z_j, s)}. 
\]
We are also regarding this function as the one defined on $\mathcal{V}_{jk}$ according to the rule we mentioned in Remark \ref{rmk:rule_extension} (i.e. by pulling back by ${\rm Pr}_{\mathcal{U}_j}$). 

In order for all $F^{(\lambda)}_{j, a}$'s to satisfy {\bf (Property)$_m$}, 
the following two formal expansions of $t_{jk}w_k^{(\lambda)}$ around a point $(0, 0, \dots, 0, z_j, s)$ for each $z_j\in U_{jk}$ and each $s\in{\rm U}(1)$ should coincide: 
\begin{eqnarray}
t_{jk}^{(\lambda)}w_k^{(\lambda)} &=&w_j^{(\lambda)}+\sum_{|a|\geq 2} f^{(\lambda)}_{kj, a}(z_j, s)\cdot w_j^a \nonumber \\
&=&\widehat{w}_j^{(\lambda)}+\sum_{|a|\geq 2}F^{(\lambda)}_{j, a}(z_j, s)\cdot \widehat{w}_j^a+\sum_{|a|\geq 2} f^{(\lambda)}_{kj, a}(z_j, s)\cdot \prod_{\mu=1}^r\left(\widehat{w}_j^{(\mu)}+\sum_{|b|\geq 2}F^{(\mu)}_{j, b}(z_j, s)\cdot \widehat{w}_j^b\right)^{a_\mu} \nonumber
\end{eqnarray}
and
\begin{eqnarray}
t_{jk}^{(\lambda)}w_k^{(\lambda)} &=& t_{jk}^{(\lambda)}\widehat{w}_k^{(\lambda)}+\sum_{|a|\geq 2}t_{jk}^{(\lambda)}F^{(\lambda)}_{k, a}(z_k, s)\cdot \widehat{w}_k^a\nonumber \\
&=&t_{jk}^{(\lambda)}\widehat{w}_k^{(\lambda)}+\sum_{|a|\geq 2}t_{jk}^{(\lambda)}F^{(\lambda)}_{k, a}(z_k(w_j, z_j, s), s)\cdot t_{kj}^a\widehat{w}_j^a\nonumber \\
&=&t_{jk}^{(\lambda)}\widehat{w}_k^{(\lambda)}+\sum_{|a|\geq 2}t_{jk}^{(\lambda)}t_{kj}^a\left(F^{(\lambda)}_{k, a}(z_k(0, z_j, s), s)+\sum_{|b|\geq 2}F^{(\lambda)}_{kj, a, b}(z_j, s)\cdot w_j^b\right)\cdot \widehat{w}_j^a\nonumber \\
&=&t_{jk}^{(\lambda)}\widehat{w}_k^{(\lambda)}+\sum_{|a|\geq 2}t_{jk}^{(\lambda)}t_{kj}^aF^{(\lambda)}_{k, a}(z_k(0, z_j, s), s)\cdot \widehat{w}_j^a\nonumber \\
&&+\sum_{|a|\geq 2}t_{jk}^{(\lambda)}t_{kj}^a\sum_{|b|\geq 2}F^{(\lambda)}_{kj, a, b}(z_j, s)\cdot \widehat{w}_j^a\cdot \prod_{\mu=1}^r\left(\widehat{w}_j^{(\mu)}+\sum_{|c|\geq 2}F^{(\mu)}_{j, c}(z_j, s)\cdot \widehat{w}_j^c\right)^{b_\mu} \nonumber 
\end{eqnarray}
when $U_k$ is smooth, 
where $F^{(\lambda)}_{kj, a, b}(z_j, s)$'s are the function defined by the expansion 
\[
F^{(\lambda)}_{k, a}(z_k(w_j, z_j, s), s)
=F^{(\lambda)}_{k, a}(z_k(0, z_j, s), s)+\sum_{|b|\geq 2}F^{(\lambda)}_{kj, a, b}(z_j, s)\cdot w_j^b. 
\]
When $U_k$ is singular, replace the second expansion with 
\begin{eqnarray}
t_{jk}^{(\lambda)}w_k^{(\lambda)} &=& t_{jk}^{(\lambda)}\widehat{w}_k^{(\lambda)}+\sum_{|a|\geq 2}t_{jk}^{(\lambda)}F^{(\lambda)}_{k, a}(x_k, y_k, s)\cdot \widehat{w}_k^a\nonumber \\
&=&t_{jk}^{(\lambda)}\widehat{w}_k^{(\lambda)}+\sum_{|a|\geq 2}t_{jk}^{(\lambda)}F^{(\lambda)}_{k, a}(x_k(w_j, z_j, s), y_k(w_j, z_j, s), s)\cdot t_{kj}^a\widehat{w}_j^a\nonumber \\
&=&t_{jk}^{(\lambda)}\widehat{w}_k^{(\lambda)}+\sum_{|a|\geq 2}t_{jk}^{(\lambda)}t_{kj}^a\left(F^{(\lambda)}_{k, a}(x_k(0, z_j, s), y_k(0, z_j, s), s)+\sum_{|b|\geq 2}F^{(\lambda)}_{kj, a, b}(z_j, s)\cdot w_j^b\right)\cdot \widehat{w}_j^a\nonumber \\
&=&t_{jk}^{(\lambda)}\widehat{w}_k^{(\lambda)}+\sum_{|a|\geq 2}t_{jk}^{(\lambda)}t_{kj}^aF^{(\lambda)}_{k, a}(x_k(0, z_j, s), y_k(0, z_j, s), s)\cdot \widehat{w}_j^a\nonumber \\
&&+\sum_{|a|\geq 2}t_{jk}^{(\lambda)}t_{kj}^a\sum_{|b|\geq 2}F^{(\lambda)}_{kj, a, b}(z_j, s)\cdot \widehat{w}_j^a\cdot \prod_{\mu=1}^r\left(\widehat{w}_j^{(\mu)}+\sum_{|c|\geq 2}F^{(\mu)}_{j, c}(z_j, s)\cdot \widehat{w}_j^c\right)^{b_\mu}, \nonumber 
\end{eqnarray}
where 
$F^{(\lambda)}_{kj, a, b}(z_j, s)$'s are the function defined by the expansion 
\[
F_{k, a}^{(\lambda)}(x_k(w_j, z_j, s), y_k(w_j, z_j, s), s)
=F^{(\lambda)}_{k, a}(x_k(0, z_j, s), y_k(0, z_j, s), s)+\sum_{|b|\geq 2}F^{(\lambda)}_{kj, a, b}(z_j, s)\cdot w_j^b
\]
in this case. 
In the following, we will construct $F^{(\lambda)}_{j, a}$'s according to the observation based on the comparison of these expansions. 

\begin{remark}
On each $X_s$, the idea of the construction of $F^{(\lambda)}_{j, a}$'s is the same one as in the proof of \cite[Theorem 3]{U}. 
Indeed, one can run just the same argument when $s\in {\rm U}(1)$ 
as the one described in \cite{U}, \cite{K6}, and \cite{K8}. 
In the construction of $F^{(\lambda)}_{j, a}$'s, the condition that $s\in {\rm U}(1)$ is important since otherwise the transition functions $t_{jk}(s)$ need not to be elements of ${\rm U}(1)$ and this cause a serious problem when we compare a cohomology class what we will denote by $[\{(U_{jk, s}, h_{kj, a}^{(\lambda)}(-, s))\}]$ in the notation below with Ueda classes. 
This problem is caused by the difficulty on the well-definedness of Ueda classes when $t_{jk}(s)$'s are not unitary (see also Remark \ref{rmk:CLPT}). 
However, in our configuration, it follows by {\bf Assumption} $1$ and $3$ that 
$\check{H}^1(\{U_{j, s}\}, S^{m}N^*_s\otimes N_s)=0$ holds for any $s\in S_{m-1}\setminus\mathrm{U}(1)$, which helps us to overcome this kind of difficulty. 
\end{remark}

\subsubsection{Construction on $S_m^*$}

Denote by $h_{1, kj, a}^{(\lambda)}(z_j, s)$ the coefficient of $\widehat{w}_j^a$ in the expansion 
\[
\sum_{|a|\geq 2} f^{(\lambda)}_{kj, a}(z_j, s)\cdot \prod_{\mu=1}^r\left(\widehat{w}_j^{(\mu)}+\sum_{|b|\geq 2}F^{(\mu)}_{j, b}(z_j, s)\cdot \widehat{w}_j^b\right)^{a_\mu}. 
\]
Then one have that 
\begin{eqnarray}
&&\widehat{w}_j^{(\lambda)}+\sum_{|a|\geq 2}F^{(\lambda)}_{j, a}(z_j, s)\cdot \widehat{w}_j^a+\sum_{|a|\geq 2} f^{(\lambda)}_{kj, a}(z_j, s)\cdot \prod_{\mu=1}^r\left(\widehat{w}_j^{(\mu)}+\sum_{|b|\geq 2}F^{(\mu)}_{j, b}(z_j, s)\cdot \widehat{w}_j^b\right)^{a_\mu} \nonumber \\
&=& \widehat{w}_j^{(\lambda)}+\sum_{|a|\geq 2}\left(F^{(\lambda)}_{j, a}(z_j, s)+h_{1, kj, a}^{(\lambda)}(z_j, s)\right)\cdot \widehat{w}_j^a. \nonumber
\end{eqnarray}
Denote by $h_{2, kj, a'}^{(\lambda)}(z_j, s)$ the coefficient of $\widehat{w}_j^{a'}$ in the expansion 
\[
\sum_{|a|\geq 2}t_{jk}^{(\lambda)}t_{kj}^a\sum_{|b|\geq 1}F^{(\lambda)}_{kj, a, b}(z_j, s)\cdot \widehat{w}_j^a\cdot \prod_{\mu=1}^r\left(\widehat{w}_j^{(\mu)}+\sum_{|c|\geq 2}F^{(\mu)}_{j, c}(z_j, s)\cdot \widehat{w}_j^c\right)^{b_\mu}. 
\]
Note that 
\begin{eqnarray}
&&\widehat{w}_j^{(\lambda)}+\sum_{|a|\geq 2}t_{jk}^{(\lambda)}t_{kj}^aF^{(\lambda)}_{k, a}(z_k(0, z_j, s), s)\cdot \widehat{w}_j^a\nonumber \\
&&+\sum_{|a|\geq 2}t_{jk}^{(\lambda)}t_{kj}^a\sum_{|b|\geq 1}F^{(\lambda)}_{kj, a, b}(z_j, s)\cdot \widehat{w}_j^a\cdot \prod_{\mu=1}^r\left(\widehat{w}_j^{(\mu)}+\sum_{|c|\geq 2}F^{(\mu)}_{j, c}(z_j, s)\cdot \widehat{w}_j^c\right)^{b_\mu}\nonumber \\
&=&\widehat{w}_j^{(\lambda)}+\sum_{|a|\geq 2}\left(t_{jk}^{(\lambda)}t_{kj}^aF^{(\lambda)}_{k, a}(z_k(0, z_j, s), s)+h_{2, kj, a}^{(\lambda)}(z_j, s)\right)\cdot \widehat{w}_j^a. \nonumber
\end{eqnarray}
It is simply observed that each $h_{1, kj, a}^{(\lambda)}(z_j, s)$ and $h_{2, kj, a}^{(\lambda)}(z_j, s)$ depends only on 
$\{F^{(\lambda)}_{j, b}\}_{|b|<|a|}$ and some known functions, and thus especially it does not depend on $\{F^{(\lambda)}_{j, b}\}_{|b|\geq |a|}$. 
Therefore, according to the observation we made above, we define $\{F^{(\lambda)}_{j, a}\}_{|a|=m+1}$ by Lemma \ref{lem:induction_main}  and Lemma \ref{prop:main_estim} below so that they are the solution of the equation 
\[
F^{(\lambda)}_{j, a}(z_j, s)+h_{1, kj, a}^{(\lambda)}(z_j, s)
=t_{jk}^{(\lambda)}t_{kj}^aF^{(\lambda)}_{k, a}(z_k(0, z_j, s), s)+h_{2, kj, a}^{(\lambda)}(z_j, s), 
\]
or equivalently, 
\[
\delta\{(\mathcal{U}_{j}^{(m)}, F_{j, a}^{(\lambda)})\}
=\{(\mathcal{U}_{jk}^{(m)}, h_{kj, a}^{(\lambda)}(z_j, s))\}
\in \check{Z}^1(\{(\mathcal{U}_{j}^{(m)}\}, \mathcal{O}_{\mathcal{Y}_m}(\mathcal{L}^{a-e_\lambda})), 
\]
where $h_{kj, a}^{(\lambda)}(z_j, s):=h_{1, kj, a}^{(\lambda)}(z_j, s)-h_{2, kj, a}^{(\lambda)}(z_j, s)$, 
when $\{F_{j, b}\}_{|b|\leq m}$ is already decided in a manner such that each {\bf (Property)$_{m'}$} holds for $m'=1, 2, \dots, m$ (inductive assumption). 
See the proof of Lemma \ref{lem:induction_main} 
for the fact that $\{(\mathcal{U}_{jk}^{(m)}, h_{kj, a}^{(\lambda)}(z_j, s))\}
\in \check{Z}^1(\{(\mathcal{U}_{j}^{(m)}\}, \mathcal{O}_{\mathcal{Y}_m}(\mathcal{L}^{a-e_\lambda}))$. 
Here we are denoting 
by $e_\lambda$ the multi-index $(0, 0, \dots, 0, 1, 0, \dots, 0)\in \mathbb{Z}^r$ whose $\lambda$-th entry is one and the other entries are zero, 
by $\mathcal{Y}_m$ the set $p^{-1}(Y\times S_m)$, 
by $\mathcal{U}^{(m)}_j$ the set $p^{-1}(U_j\times S_m)$, 
and by $\mathcal{U}^{(m)}_{jk}$ the intersection $\mathcal{U}^{(m)}_j\cap \mathcal{U}^{(m)}_k$. 

First we construct $F_{j, a}^{(\lambda)}$'s on $S_m^*:=S_m\setminus\mathrm{U}(1)$. 

\begin{lemma}\label{lem:induction_main}
Let $m$ be a positive integer. 
Assume that 
$\{F_{j, b}\}_{|b|\leq m}$ is already decided in a manner such that each {\bf (Property)$_{m'}$} holds for $m'=1, 2, \dots, m$. 
Then there uniquely exist holomorphic functions $\{F^{(\lambda)}_{j, a}\}_{|a|=m+1}$ on $p^{-1}(Y\times S_m^*)$ such that 
\[
F^{(\lambda)}_{j, a}(z_j, s)-t_{jk}^{(\lambda)}t_{kj}^aF^{(\lambda)}_{k, a}(z_k(0, z_j, s), s)
=-h_{1, kj, a}^{(\lambda)}(z_j, s)+h_{2, kj, a}^{(\lambda)}(z_j, s) 
\]
holds on each $U_{jk, s}$ with $s\in S_m^*$, 
and that $\{F^{(\lambda)}_{j, a}\}_{|a|=m+1}$ satisfies {\bf (Property)$_{m+1}$}. 
\end{lemma}

\begin{proof}
As it follows by the same argument as in the end of \S \ref{section:proof_outline} that one can solve the functional equation
\[
w_j^{(\lambda)}=u_j^{(\lambda)}+\sum_{2\leq |a| \leq m}F^{(\lambda)}_{j, a}\cdot u_j^a
\]
to define a new system $\{(\widetilde{\mathcal{V}}_j, u_j)\}$ of local defining functions of $p^{-1}(Y\times S_{m-1}^*)$ by using the implicit function theorem. 
As it holds that
$T_{jk}u_k =u_j + O(|u_j|^{m+1})$ by {\bf (Property)$_{m}$}, 
it follows from the calculation as we described above that 
\[
T_{jk}u_k =u_j - \sum_{|a|=m+1}h_{kj, a}^{(\lambda)}(z_j, s)\cdot u_j^a + O(|u_j|^{m+2}). 
\]
Denote by $h_{jk, a}^{(\lambda)}$ the function $-t_{jk}^{a-e_\lambda}\cdot h_{kj, a}^{(\lambda)}$ when $U_k\cap Y_{\rm sing}\not=\emptyset$. 
Then one easily obtain, by comparing the both hands sides of the expansion of 
$T_{jk}\widehat{w}_k -  \widehat{w}_j = 
T_{j\ell}(T_{\ell k}\widehat{w}_k -  \widehat{w}_\ell) +(T_{j\ell} \widehat{w}_\ell-  \widehat{w}_j )$, that 
$\{(\mathcal{U}_{jk}^{(m)}, h_{kj, a}^{(\lambda)}(z_j, s))\}
\in \check{Z}^1(\{(\mathcal{U}_{j}^{(m)}\}, \mathcal{O}_{\mathcal{Y}_m}(\mathcal{L}^{a-e_\lambda}))$. 

It follows from {\bf Assumption} $1$ and $3$ that $H^1(Y_s, L_s^{a-e_\lambda})=0$ for any 
$s\in S_m^*$. 
Thus one have that 
$R^j({\rm Pr}_2\circ p|_{\mathcal{Y}_m^*})_*\mathcal{L}^{a-e_\lambda}=0$ for each $j>0$ for $\mathcal{Y}_m^*:=p^{-1}(Y\times S_m^*)$, 
by which it follows that 
$H^1(\mathcal{Y}_m^*, \mathcal{L}^{a-e_\lambda})
=H^1(S_m^*, ({\rm Pr}_2\circ p)_*\mathcal{L}^{a-e_\lambda})$. 
As $S_m^*$ is Stein and $({\rm Pr}_2\circ p)_*\mathcal{L}^{a-e_\lambda}$ is coherent by Grauert's theorem, 
one have that $H^1(\mathcal{Y}_m^*, \mathcal{L}^{a-e_\lambda})=0$, which implies the lemma. 
Note that the uniqueness of the solution follows by $H^0(Y_s, L_s^a)=0$ for each $s\in S_m^*$ (Here we used {\bf Assumption $3$} $(ii)$). 
\end{proof}

\subsubsection{Construction on $S_m$ and preliminary $L^\infty$ estimates}

In this subsection, we fix a positive integer $m$ and $\nu\in\{0, 1, 2, \dots, mM_0-1\}$, 
and use the following simple notation: $\xi_0:=\nu/(M_0m)$, $s_0:=\exp(2\pi\sqrt{-1}\xi_0)$. 
Let $W_0$ be the set defined by 
\[
W_0:=\left\{\xi\in \mathbb{C}\left| \left|{\rm Re}\,\xi \leq \xi_0\right|\leq\frac{1}{2M_0m},\ |{\rm Im}\,\xi|<\frac{1}{4M_0m}\right.\right\}. 
\]
Note that it holds that 
$\exp(2\pi\sqrt{-1}W_0)=W_{m, \nu}$. 

In the previous subsection, under the inductive assumption and Lemma \ref{lem:induction_main}, 
we have seen that there exists a holomorphic function 
$F_{j, a}^{(\lambda)}\colon p^{-1}(Y\times S_m^*)\to\mathbb{C}$ for each $a$ with $|a|=m+1$ which enjoys {\bf (Property)$_{m+1}$}, which is the solution of the functional equation 
\[
-F^{(\lambda)}_{j, a}(z_j, s)+t_{jk}^{(\lambda)}t_{kj}^aF^{(\lambda)}_{k, a}(z_k(0, z_j, s), s)
=h_{kj, a}^{(\lambda)}(z_j, s). 
\]
Note that $h_{kj, a}^{(\lambda)}(z_j, s)$ is holomorphic on $\mathcal{Y}_{m-1}$. 
In what follows, we show that $F_{j, a}^{(\lambda)}$ holomorphically extends to $\mathcal{Y}_{m}$, or equivalently, that $F_{j, a}^{(\lambda)}|_{p^{-1}(Y\times W_{m, \nu}^*)}$ holomorphically extends to $p^{-1}(Y\times W_{m, \nu})$, where $W_{m, \nu}^*:=W_{m, \nu}\setminus\{s_0\}$ (for each $\nu$). 

Fist, let us note that, 
when $L^{a-e_\lambda}_{s_0}$ is not holomorphically trivial, 
one can construct 
the function $F_{j, a}^{(\lambda)}$ as the one defined on $p^{-1}(Y\times W_{m, \nu})$ by exactly the same argument as in the proof of Lemma \ref{lem:induction_main}. 
Note also that, in this case, 
it follows from {\bf Assumption $4$} that 
\begin{equation}\label{eq_estim_on_W_Ls0hol_non_triv_case}
\max_j\sup_{U_{j, s}}|F^{(\lambda)}_{j, a}|
\leq \frac{K}{m\cdot (2M_0m)^{-1}}\max_{j, k}\sup_{U_{jk, s}}|h_{kj, a}^{(\lambda)}|
= 2M_0K\cdot \max_{j, k}\sup_{U_{jk, s}}|h_{kj, a}^{(\lambda)}|. 
\end{equation}

When $L^{a-e_\lambda}_{s_0}$ is holomorphically trivial, we show the following: 

\begin{proposition}\label{prop:main_estim}
Let $a$ be a multi-index with $|a|=m+1$, and $\lambda$ be an element of $\{1, 2, \dots, r\}$. 
Denote by $a'$ the multi-index $\alpha-e_\lambda$. 
Assume that $L^{a'}_{s_0}$ is holomorphically trivial. 
Let 
$\alpha=\{(U_{jk}\times W_{m, \nu}, \alpha_{jk}dw_j^{a'})\}\in\check{Z}^1(\{U_j\times W_{m, \nu}\}, \mathcal{O}_{\mathcal{Y}}(\mathcal{L}^{a'}))$
be a $1$-cocycle with $M:=\max_{j, k}\sup_{U_{jk}\times W_{m, \nu}}|\alpha_{jk}|<\infty$, 
where we are identifying $\mathcal{L}$ with $p^*\mathcal{L}$ via $p$. 
For each $\zeta\in W_0$, denote by 
$\alpha_\xi=\{(U_{jk}\times \{\exp(2\pi\sqrt{-1}\xi)\}, \alpha_{\zeta, jk}dw_j^{a'})\}\in\check{Z}^1(\{U_j\times \{\exp(2\pi\sqrt{-1}\xi)\}, \mathcal{O}_{Y_{\exp(2\pi\sqrt{-1}\xi)}}(L^{a'}_{\exp(2\pi\sqrt{-1}\xi)}))$ 
the $1$-cocycle obtained by restricting $\alpha$ 
to $Y\times \{\exp(2\pi\sqrt{-1}\xi)\}$. 
Then the followings are equivalent: \\
$(i)$ $[\alpha_{\xi_0}]=0\in \check{H}^1(\{U_{j, s_0}\}, \mathcal{O}_{Y_{s_0}})$. \\
$(ii)$ $[\alpha]=0\in \check{H}^1(\{U_j\times W_{m, \nu}\}, \mathcal{O}_\mathcal{Y}(\mathcal{L}^{a'}|_{Y\times W_{m, \nu}}))$. 
\end{proposition}

\begin{proof}
As clearly $(ii)$ implies $(i)$, we show the converse.  
Assume that the assertion $(i)$ holds. 
In what follows, we use the estimate $\max_{j, k}\sup_{s\in W_{m, \nu}}|t_{jk}^{(\lambda)}(s)|\leq \Theta^{1/m}\leq\Theta$ follows from {\bf Assumption $3$} $(iii)$. 

By the assertion $(i)$, it follows from \cite[Lemma 3]{U} (=\cite[Lemma 2]{KS}, for the smooth case) or \cite[Lemma 4.1]{K6} (for the singular case) that 
there exists a $0$-cochain 
$\beta_{\xi_0}=\{(U_{j}\times \{s_0\}, \beta_{\xi_0, j}dw_j^{a'})\}\in\check{C}^0(\{U_j\times \{s_0\}, \mathcal{O}_{Y_{s_0}}(L^{a'}_{s_0}))$ 
with $\delta \beta_{\xi_0}=\alpha_{\xi_0}$ such that 
\[
\max_j\sup_{U_{j, s_0}}|\beta_{\xi_0, j}|\leq  K_{\rm KS}M 
\]
holds, where $K_{\rm KS}$ is the constant which depends only on $Y$ and $\{U_j\}$. 
Define a holomorphic function $\beta_j\colon U_j\times W_{m, \nu}\to \mathbb{C}$ by 
$\beta_j(z_j, s):=\beta_{\xi_0, j}(z_j)$, a $0$-cochain $\beta$ by 
\[
\beta:=\{(U_{j}\times W_{m, \nu}, \beta_jdw_j^{a'})\}\in\check{C}^0(\{U_j\times  W_{m, \nu}\}, \mathcal{O}_{\mathcal{Y}}(\mathcal{L}^{a'})), 
\]
and a $1$-cocycle 
\[
\gamma=\{(U_{jk}\times W_{m, \nu}, \gamma_{jk})\}\in\check{Z}^1(\{U_j\times  W_{m, \nu}\}, \mathcal{O}_{\mathcal{Y}}(\mathcal{L}^{a'}))
\]
by $\gamma:=\alpha-\delta\beta$. 
In what follows, we will construct a primitive $\widetilde{\beta}$ of $\gamma$ 
(Then clearly it holds that $\delta(\beta+\widetilde{\beta})=\alpha$, which proves the lemma). 
Note that 
\[
\max_{j, k}\sup_{U_{jk}\times W_{m, \nu}}|\gamma_{jk}|\leq M+(1+\Theta)\cdot K_{\rm KS}M=(1+K_{\rm KS}+\Theta K_{\rm KS})M. 
\]
Denote by 
$W_0^*$ the set $W_0\setminus\{\zeta_0\}$ and by $W_{m, \nu}^*$ the image $\exp(2\pi\sqrt{-1}W_0^*)$. 
It follows by the same argument as in the proof of 
Lemma \ref{lem:induction_main}, one have that 
there uniquely exists a holomorphic function $\widetilde{\beta}_j\colon U_j\times W_{m, \nu}^*\to \mathbb{C}$ such that the $0$-cochain $\widetilde{\beta}$ defined by 
\[
\widetilde{\beta}:=\{(U_{j}\times W_{m, \nu}^*, \widetilde{\beta}_jdw_j^{a'})\}\in\check{C}^0(\{U_j\times W_{m, \nu}^*\}, \mathcal{O}_{\mathcal{Y}}(\mathcal{L}^{a'}))
\]
satisfies $\delta\widetilde{\beta}=\alpha|_{Y\times W_{m, \nu}^*}$. 
According to {\bf Assumption $4$}, one have the inequality 
\[
m\cdot |\xi-\xi_0|\cdot \|\widetilde{\beta}\|_{\exp(2\pi\sqrt{-1}\xi), a'}\leq K\cdot \|\gamma\|_{\exp(2\pi\sqrt{-1}\xi), a'}
\]
for each $\xi\in W_0^*$. 
Therefore, for each element $z_j\in U_j$ and each $\xi\in W_0^*$, one have that
\begin{eqnarray}
|\widetilde{\beta}(z_j, \exp(2\pi\sqrt{-1}\xi))|&\leq& \frac{K}{m|\xi-\xi_0|}\cdot \max_{j, k}\sup_{p\in U_{jk}}|\gamma_{jk}(p, \exp(2\pi\sqrt{-1}\xi))|\nonumber \\
&=&\frac{K}{m}\cdot \max_{j, k}\sup_{p\in U_{jk}}\left|\frac{\gamma_{jk}(p, \exp(2\pi\sqrt{-1}\xi))}{\xi-\xi_0}\right|\nonumber
\end{eqnarray}
holds. As $\gamma_{jk}|_{\xi=\xi_0}\equiv 0$ by construction, one have that the function 
\[
U_{jk}\times W_0^*\ni (p, \xi)\mapsto \frac{\gamma_{jk}(p, \exp(2\pi\sqrt{-1}\xi))}{\xi-\xi_0}\in \mathbb{C}
\]
can also be regarded as a holomorphic function defined on $U_{jk}\times W_0$. 
Thus it follows by the maximum principle that there exists a point $(p_*, \xi_*)$ of the boundary $\partial U_{jk}\times \partial W_0$ which attains the maximum 
\[
\sup_{(p, \xi)\in U_{jk}\times W_0}\left|\frac{\gamma_{jk}(p, \exp(2\pi\sqrt{-1}\xi))}{\xi-\xi_0}\right|. 
\]
Therefore, by $|\xi_*- \xi_0|\geq \frac{1}{4M_0m}$, we have that 
\begin{eqnarray}
\sup_{(p, \xi)\in U_{jk}\times W_0}\left|\frac{\gamma_{jk}(p, \exp(2\pi\sqrt{-1}\xi))}{\xi-\xi_0}\right|
&=& \frac{|\gamma_{jk}(p_*, \exp(2\pi\sqrt{-1}\xi_*))|}{|\xi_*-\xi_0|} \nonumber \\ 
&\leq&\frac{(1+K_{\rm KS}+\Theta K_{\rm KS})M}{(4M_0m)^{-1}}\nonumber \\
&=&4M_0m\cdot (1+K_{\rm KS}+\Theta K_{\rm KS})M\nonumber
\end{eqnarray}
holds. 
Thus we obtain the estimate 
\[
|\widetilde{\beta}(z_j, \exp(2\pi\sqrt{-1}\xi))|\leq 
\frac{K}{m}\cdot \max_{j, k}\sup_{p\in U_{jk}}\left|\frac{\gamma_{jk}(p, \exp(2\pi\sqrt{-1}\xi))}{\xi-\xi_0}\right|
\leq 4M_0K\cdot (1+K_{\rm KS}+\Theta K_{\rm KS})M
\]
for each $z_j\in U_j$ and $\xi\in W_0^*$, 
by which one can apply Riemann's extension theorem to conclude that $\widetilde{\beta}$ can be holomorphically extended to $\xi=\xi_0$. 
\end{proof}

As is seen in \S \ref{section:estim_Fs}, 
the assumption ``$M<\infty$'' in Proposition \ref{prop:main_estim} will be inductively assured for the case where $\alpha$ is the one defined by $h_{kj, a}^{(\lambda)}$. 
For such case, the assertion $(i)$ of Proposition \ref{prop:main_estim} actually holds by {\bf Assumption $5$} and the argument as in \cite[p. 598]{U} (see also \cite[Remark 3.5]{K6} and \cite[\S 4.2.1]{K6} for singular case, \cite[\S 3.2]{K8} and \cite[Claim 4.3, 4.4]{K8} for higher codimensional case), by which we obtain that the solution $F_{j, a}^{(\lambda)}$'s of 
\[
-F^{(\lambda)}_{j, a}(z_j, s)+t_{jk}^{(\lambda)}t_{kj}^aF^{(\lambda)}_{k, a}(z_k(0, z_j, s), s)
=h_{kj, a}^{(\lambda)}(z_j, s)
\]
actually exist on $\mathcal{Y}_m$. 
Note also that it follows by the argument in the proof of Proposition \ref{prop:main_estim} that the inequality 
\begin{equation}\label{eq_estim_on_W_Ls0hol_triv_case}
\max_j\sup_{U_{j, s}}|F^{(\lambda)}_{j, a}|
\leq ((1+\Theta)K_{\rm KS}+4M_0K\cdot (1+K_{\rm KS}+\Theta K_{\rm KS}))\cdot \max_{j, k}\sup_{U_{jk, s}}|h_{kj, a}^{(\lambda)}|
\end{equation}
holds for each $s\in W_{m, \nu}$. 

\subsection{Estimates of $|F^{(\lambda)}_{j, a}|$}\label{section:estim_Fs}

Here we will construct positive constants $\{A_m\}_{m=2}^\infty$ such that 
\[
\max_{j}\max_{|a|=m+1}\sup_{p^{-1}(U_j\times S_m)}|F^{(\lambda)}_{j, a}|\leq A_{m+1}
\]
holds and that the formal series 
\[
A(X^1, X^2, \dots, X^r):=\sum_{|a|\geq 2}A_{|a|}X^a
:=\sum_{|a|\geq 2}A_{|a|}\prod_{\nu=1}^r(X^\nu)^{a_\nu}
\]
is convergent. 

First, we define a new open covering $\{U_j^*\}$ of $Y$ whose index set coincides with that of $\{U_j\}$ as follows: 
Set $U_k^*:=U_k$ if $U_k\cap Y_{\rm sing}\not=\emptyset$, and 
$U_j^*$ is a relatively compact subset of $U_j$ if $U_j\cap Y_{\rm sing}=\emptyset$ such that $U_{jk}\not=\emptyset$ if and only if $U_j^*\cap U_k^*\not=\emptyset$. 
By {\bf  Assumption} $3$ $(i)$, one have that 
\[
\{(z_j, w_j, s)\in \mathcal{V}_j\mid s\in S_1,\ z_j\in U_j\cap U_k^*,\ |w_j^{(\lambda)}|<1/R\ \text{for each}\ \lambda\in\{1, 2, \dots, r\}\}\Subset \mathcal{V}_k
\]
holds for sufficiently large constant $R$ if both $U_j$ and $U_k$ are smooth. 

When $U_k$ is singular, we need to modify $U_k$ and some of the elements of $\{U_j\}$ as follows if necessary (see also \cite[Remark 4.3]{K6}). 
Take a sufficiently small positive constant $\varepsilon$ such that 
$\mathcal{V}_{(k, 3)}\Subset \mathcal{V}_k$, where 
$\mathcal{V}_{(k, \mu)}:=\{(x_k, y_k, s)\in\mathcal{V}_k\mid \max\{|x_k|,\ |y_k|\}<\mu\cdot \varepsilon,\ s\in S_1\}$ for $\mu=1, 2, 3$. 
Denote by $U_{(k, \mu)}$ the subset $\{(x_k, y_k)\in U_k\mid \max\{|x_k|,\ |y_k|\}<\mu\cdot \varepsilon\}$ for $\mu=1, 2, 3$. 
For non-singular $U_j$ such that $U_{(k, 3)}\cap U_j\not=\emptyset$ for some singular $U_k$, we will divide $U_j$ into $U_{(j, 3)}:=U_{(k, 3)}\cap U_j$ and $U_{(j, 2)}:=U_j\setminus \overline{U_{(k, 2)}}$, and use 
\begin{eqnarray}
&&\{U_j\mid U_j\ \text{is smooth},\ U_{(k, 3)}\cap U_j=\emptyset\ \text{for any singular}\ U_k\}\nonumber \\
&\cup&\{U_{(j, 2)}\mid U_{(k, 3)}\cap U_j\not=\emptyset\ \text{for some singular}\ U_k\}\nonumber \\
&\cup&\{U_{(j, 3)}\mid U_{(k, 3)}\cap U_j\not=\emptyset\ \text{for some singular}\ U_k\}\nonumber \\
&\cup&\{U_{(k, 1)}\mid U_k\ \text{is singular}\}\nonumber
\end{eqnarray}
as a new open covering of $Y$. 
We will use the restriction of the original coordinate $(x_k, y_k, s)$ on a neighborhood $\mathcal{V}_{(k, 1)}$ of $p^{-1}(U_{(k, 1)}\times S)$ for each $k$ such that $U_k$ is singular, 
the restriction of the original coordinate $(z_j, w_j, s)$ on a neighborhood of $\mathcal{V}_{(j, 2)}$ of $p^{-1}(U_{(j, 2)}\times S)$ for each $U_{(j, 2)}$ ($\mathcal{V}_{(j, 2)}$ is, for example, a subset of $\mathcal{V}_j\setminus \overline{\mathcal{V}_{(k, 2)}}$). 
For each $U_{(j, 3)}$, we use the restriction of $(x_k, y_k, s)$ as the coordinate of a neighborhood. 
In this case, either $x_k\equiv 0$ or $y_k\equiv 0$ holds on $p^{-1}(U_{(j, 3)}\times S)$. 
When, for example, $y_k\equiv 0$ holds on $p^{-1}(U_{(j, 3)}\times S)$, 
we use the set $\mathcal{V}_{(j, 3)}:=\{(x_k, y_k, s)\in \mathcal{V}_k\mid x_k\in U_{(j, 3)},\ |y_k|<\varepsilon\}$ as a neighborhood of $p^{-1}(U_{(j, 3)}\times S)$, and use $w_{(j, 3)}:= y_k$ and $z_{(j, 3)}:=x_k$ as coordinates. 

After modifying $\{U_j\}$ and $\{\mathcal{V}_j\}$ in such manner, we may assume that 
\[
\{(z_j, w_j, s)\in \mathcal{V}_j\mid s\in S_1,\ z_j\in U_j\cap U_k,\ |w_j^{(\lambda)}|\leq 1/R\ \text{for each}\ \lambda\in\{1, 2, \dots, r\}\}\subset \mathcal{V}_k
\]
holds for any smooth $U_j$ even if $U_k$ is singular (by letting $1/R<\varepsilon/2$). 

Denote by  $M$ the constant $\Theta\cdot \max_\lambda\max_j\sup_{\mathcal{V}_j}|w_j^{(\lambda)}|$, which is finite by {\bf Assumption $1$} $(ii)$. 
By Cauchy's inequality, one have that 
\[
\max_{j, k}\sup_{(U_j\cap U_k^*)\times S} |f_{kj, a}^{(\lambda)}|\leq MR^{|a|}. 
\]
For any $a$. 
It follows by $h_{kj, a}^{(\lambda)}=f_{kj, a}^{(\lambda)}$ that 
$\max_{j, k}\sup_{(U_j\cap U_k^*)\times S_1} |h_{kj, a}^{(\lambda)}|\leq MR^2$. 
Therefore one have by the cocycle condition and {\bf Assumption $3$} $(iii)$ that, for each $(p, s)\in U_{jk}\times S_1$, it holds that 
\begin{eqnarray}
|h_{kj, a}^{(\lambda)}(p, s)|&\leq& |t_{jk}^{(\lambda)}t_{kj}^a h_{k\ell, a}^{(\lambda)}(p, s)| + |t_{j\ell}^{(\lambda)}t_{\ell j}^a h_{\ell j, a}^{(\lambda)}(p, s)|\nonumber \\
&=& |h_{\ell k, a}^{(\lambda)}(p, s)| + |t_{j\ell}^{(\lambda)}t_{\ell j}^a h_{\ell j, a}^{(\lambda)}(p, s)|
\leq (1+\Theta)MR^2\nonumber
\end{eqnarray}
even if $p\not \in U_k^*$, where $\ell$ is the one such that $p\in U_\ell^*$ 
(Note that $U_k$ is smooth if $p\not \in U_k$ by construction of $\{U_j^*\}$). 
Thus we obtain the inequality 
\[
\max_{jk}\max_{|a|=2}\sup_{p^{-1}(U_{jk}\times S_1)}|h_{kj, a}^{(\lambda)}|\leq B_1:=(1+\Theta)MR^2. 
\]

Consider the division $\overline{S_m}=\bigcup_{\nu=0}^{mM_0-1}W_{m, \nu}$ for $m=1$. 
Take a multi-index $a$ with  
$|a|=2$, $\lambda\in\{1, 2, \dots, \lambda\}$, 
and $\nu\in\{0, 1, \dots, M_0-1\}$. 
When $L^{a-e_\lambda}_{\exp(2\pi\sqrt{-1}\nu/M_0)}$ is holomorphically trivial, 
it follows from the inequality $(\ref{eq_estim_on_W_Ls0hol_triv_case})$ that 
\[
\max_{j}\sup_{p^{-1}(U_j\times W_{1, \nu})}|F_{j, a}^{(\lambda)}|\leq 
\left((1+\Theta)K_{\rm KS}+4M_0K\cdot (1+K_{\rm KS}+\Theta K_{\rm KS})\right)B_1
\]
holds (Note again that here we used {\bf Assumption $5$} to assure the assertion $(i)$ of Proposition \ref{prop:main_estim}. Here we used the argument as in \cite[p. 598]{U}, \cite[Remark 3.5]{K6}, \cite[\S 4.2.1]{K6}, \cite[\S 3.2]{K8}, \cite[Claim 4.3, 4.4]{K8} in order to see that the $1$-cocycle defined by $h_{kj, a}^{(\lambda)}$ coincides with the Ueda class, see also the proof of Lemma \ref{lem:induction_main}). 
In what follows, we let $K>K_{\rm KS}>1$. Then one have that 
\[
\max_{j}\max_{|a|=2}\sup_{p^{-1}(U_j\times W_{1, \nu})}|F_{j, a}|\leq 
14M_0\Theta K^2\cdot B_1
\]
holds. 
When $L^{a-e_\lambda}_{\exp(2\pi\sqrt{-1}\nu/M_0)}$ is not holomorphically trivial, 
it follows from the inequality $(\ref{eq_estim_on_W_Ls0hol_non_triv_case})$ that 
\[
\max_{j}\sup_{p^{-1}(U_j\times W_{1, \nu})}|F_{j, a}|\leq 2M_0KB_1
\]
holds. 
Thus one have that $\max_{j}\max_{|a|=2}\sup_{p^{-1}(U_j\times S_1)}|F_{j, a}|\leq A_2$ holds if one let 
$A_2:=14M_0\Theta K^2\cdot (1+\Theta)MR^2$. 

Assume that we have decided $A_2, A_3, \dots, A_m$ suitably. 
Then, as $h_{1, kj, a}^{(\lambda)}$'s are defined by
\[
\sum_{|a|\geq 2} f^{(\lambda)}_{kj, a}(z_j, s)\cdot \prod_{\mu=1}^r\left(\widehat{w}_j^{(\mu)}+\sum_{|b|\geq 2}F^{(\mu)}_{j, b}(z_j, s)\cdot \widehat{w}_j^b\right)^{a_\mu} 
=\sum_{|a|\geq 2}h_{1, kj, a}^{(\lambda)}(z_j, s)\cdot \widehat{w}_j^a, 
\]
it follows that the value of  
$\max_{jk}\max_{|a|=m+1}\sup_{p^{-1}((U_j\cap U_k^*)\times S_m)}|h_{1, kj, a}^{(\lambda)}|$ is bounded from above by the coefficient of $X^a$ in the expansion of 
\[
\sum_{|a|\geq 2} MR^{|a|}\cdot \prod_{\mu=1}^r\left(X^{\mu}+A(X)\right)^{a_\mu} 
=M\prod_{\nu=1}^r\frac{1}{1-R(X^{\nu}+A(X))}
-M-MR\sum_{\nu}(X^{\nu}+A(X)), 
\]
where $A(X)=A(X^1, X^2, \dots, X^r)$. 
Note that the coefficient of $X^a$ in the expansion of the above depends only on $A_2, A_3, \dots, A_m$ if $|a|=m+1$. 
Similarly, as $h_{2, kj, a}^{(\lambda)}$'s are defined by  
\begin{eqnarray}
&&\sum_{|a|\geq 2}t_{jk}^{(\lambda)}t_{kj}^a\sum_{|b|\geq 1}F^{(\lambda)}_{kj, a, b}(z_j, s)\cdot \widehat{w}_j^a\cdot \prod_{\mu=1}^r\left(\widehat{w}_j^{(\mu)}+\sum_{|c|\geq 2}F^{(\mu)}_{j, c}(z_j, s)\cdot \widehat{w}_j^c\right)^{b_\mu}\nonumber \\
&=&\sum_{|a|\geq 2}h_{2, kj, a}^{(\lambda)}(z_j, s)\cdot \widehat{w}_j^a, \nonumber
\end{eqnarray}
the value of $\max_{jk}\max_{|a|=m+1}\sup_{p^{-1}((U_j\cap U_k^*)\times S_m)}|h_{2, kj, a}^{(\lambda)}|$ is bounded from above by the coefficient of $X^a$ in the expansion of 
\[
\sum_{|a|\geq 2}\Theta\sum_{|b|\geq 1}3A_{|a|}R^{|b|}\cdot X^a\cdot \prod_{\mu=1}^r\left(X^{\mu}+A(X)\right)^{b_\mu}
=3\Theta A(X)\cdot \sum_{|b|\geq 2}\prod_{\mu=1}^r\left(RX^{\mu}+RA(X)\right)^{b_\mu}. 
\]
Here we used the estimate $|F^{(\lambda)}_{kj, c, d}|\leq 3A_{|c|}R^{|d|}$ which is obtained by Cauchy's inequality and the rule we described in Remark \ref{rmk:rule_extension}. 

Therefore, one have that the value of 
$\max_{jk}\max_{|a|=m+1}\sup_{p^{-1}((U_j\cap U_k^*)\times S_m)}|h_{kj, a}^{(\lambda)}|$ 
is bounded from above by the coefficient of $X^a$ in the expansion of 
\begin{eqnarray}
&&M\cdot \left(\prod_{\nu=1}^r\frac{1}{1-R(X^{\nu}+A(X))}
-1-R\sum_{\nu}(X^{\nu}+A(X))\right)\nonumber \\
&+&3\Theta A(X)\cdot \left(\prod_{\nu=1}^r\frac{1}{1-R(X^{\nu}+A(X))}
-1\right).\nonumber
\end{eqnarray}
Take a point $(p, s)\in U_{jk}\times S_m$. 
When $p\not \in U_k$, as $U_k$ is smooth by construction of $\{U_j^*\}$ in this case, 
it follows from the cocycle condition and {\bf Assumption $3$} $(iii)$ that 
\begin{eqnarray}
|h_{kj, a}^{(\lambda)}(p, s)|&\leq& |t_{jk}^{(\lambda)}t_{kj}^a h_{k\ell, a}^{(\lambda)}(p, s)| + |t_{j\ell}^{(\lambda)}t_{\ell j}^a h_{\ell j, a}^{(\lambda)}(p, s)|
= |h_{\ell k, a}^{(\lambda)}(p, s)| + |t_{j\ell}^{(\lambda)}t_{\ell j}^a h_{\ell j, a}^{(\lambda)}(p, s)|\nonumber \\
&\leq & (1+\Theta)\cdot \max_{jk}\max_{|a|=m+1}\sup_{p^{-1}((U_j\cap U_k^*)\times S_m)}|h_{kj, a}^{(\lambda)}|, \nonumber
\end{eqnarray}
where $\ell$ is an index such that $p\in U_\ell^*$. 

Thus one have that the value of 
$\max_{jk}\max_{|a|=m+1}\sup_{p^{-1}(U_{jk}\times S_m)}|h_{kj, a}^{(\lambda)}|$ 
is bounded by the coefficient of $X^a$ in the expansion of 
\begin{eqnarray}
&&(1+\Theta)M\cdot \left(\prod_{\nu=1}^r\frac{1}{1-R(X^{\nu}+A(X))}
-1-R\sum_{\nu}(X^{\nu}+A(X))\right)\nonumber \\
&+&3(1+\Theta)\Theta A(X)\cdot \left(\prod_{\nu=1}^r\frac{1}{1-R(X^{\nu}+A(X))}
-1\right), \nonumber 
\end{eqnarray}
which will be denoted by $B_m$. 

Take a multi-index $a$ with $|a|=m+1$, an element $\nu\in\{0, 1, \dots, mM_0-1\}$, and $\lambda\in\{1, 2, \dots, r\}$. 
When $L^{a-e_\lambda}_{\exp(2\pi\sqrt{-1}\nu/(mM_0))}$ is not holomorphically trivial, 
it follows by $(\ref{eq_estim_on_W_Ls0hol_triv_case})$ that 
\[
\max_{j}\sup_{p^{-1}(U_j\times W_{m, \nu})}|F_{j, a}^{(\lambda)}(p, s)|\leq 
14M_0\Theta K^2\cdot B_m
\]
(again we remark that here we used {\bf Assumption $5$}). 
When $L^{a-e_\lambda}_{\exp(2\pi\sqrt{-1}\nu/(mM_0))}$ is holomorphically trivial, 
it follows by $(\ref{eq_estim_on_W_Ls0hol_non_triv_case})$ that 
\[
\max_{j}\sup_{p^{-1}(U_j\times W_{m, \nu})}|F_{j, a}^{(\lambda)}(p, s)|\leq 
2M_0K\cdot B_m
\]
holds. 

Thus now we have that one can obtain constants $\{A_m\}_{m=2}^\infty$ such that 
\[
\max_{j}\max_{|a|=m+1}\sup_{p^{-1}(U_j\times S_m)}|F^{(\lambda)}_{j, a}(p, s)|\leq A_{m+1}
\]
by letting $A(X)\in\mathbb{C}[[X^1, X^2, \dots, X^r]]$ be the formal power series defined by the equation 
\begin{eqnarray}
\frac{1}{14M_0\Theta K^2}A(X)
&=&(1+\Theta)M\cdot \left(\prod_{\nu=1}^r\frac{1}{1-R(X^{\nu}+A(X))}
-1-R\sum_{\nu}(X^{\nu}+A(X))\right)\nonumber \\
&&+3(1+\Theta)\Theta A(X)\cdot \left(\prod_{\nu=1}^r\frac{1}{1-R(X^{\nu}+A(X))}
-1\right). \nonumber
\end{eqnarray}
Denoting by 
$C$ the constant $14M_0\Theta K^2(1+\Theta)$, 
we let $F(X^1, X^2, \dots, X^r, Y)\in \mathcal{O}_{\mathbb{C}^{r+1}, 0}$ be the one defined by 
\begin{eqnarray}
F(X, Y):=&&-Y+CM\cdot \left(\prod_{\nu=1}^r\frac{1}{1-R(X^{\nu}+Y)}
-1-R\sum_{\nu}(X^{\nu}+Y)\right)\nonumber \\
&+&3C\Theta Y\cdot \left(\prod_{\nu=1}^r\frac{1}{1-R(X^{\nu}+Y)}
-1\right).  \nonumber
\end{eqnarray}
Then, as 
$\frac{\partial}{\partial Y}F(0, 0)
=-1\not=0$, it follows from the implicit function theorem that $A(X)$ is convergent. 

\subsection{End of the proof}

Let $s$ be an element of $\mathrm{U}(1)$. 
Then, by the estimates in 
Remark \ref{rmk:rule_extension} and \S \ref{section:estim_Fs}, one can carry out the argument as in the end of \S \ref{section:proof_outline} on each $U_{j, s}$ to obtain the solution $\widehat{w}_j$ of the functional equation $(\ref{eq:def_w_hat})$ defined on a neighborhood $\widehat{V}_{j, s}$ of $U_{j, s}$ in $X_s$. 
Note that clearly $\{(\widehat{V}_{j, s}, \widehat{w}_j)\}$ is also a local defining functions system of $Y_s$, since $w_j$ and $\widehat{w}_j$ coincide in one order jet. 
It follows by {\bf (Property)$_m$}'s that $\widehat{w}_j=T_{jk}(s)\cdot \widehat{w}_k$ holds on each $\widehat{V}_{j, s}\cap \widehat{V}_{k, s}$, the theorem holds. 
\qed



\section{Examples and Proof of Theorem \ref{thm:main_9pt_b-up_of_P^2_nodal} and \ref{thm:main_9pt_b-up_of_P^2_smooth}}

\subsection{Some examples}

First, we give an example of the configuration we described in \S \ref{section:setting_smooth}. 

\begin{example}\label{example:smooth}

Let  $(V, F)$ be a del Pezzo manifold of degree $1$: 
i.e. $V$ is a projective manifold of dimension $n$ and $F$ is an ample line bundle on $V$ with $K_V^{-1}\cong F^{n-1}$ and the self-intersection number $(F^n)$ is equal to $1$. 
In this case, as we wrote in \cite[\S 6.3]{K8}, it follows from \cite[6.4]{F} that ${\rm dim}\,H^0(V, L)=n$. 
Take general elements $D^0_1, D^0_2, \dots, D^0_n\in |F|$. 
By \cite[4.2]{F} and $(D^0_1, D^0_2, \dots, D^0_n)=(F^n)=1$, it holds that 
the intersection $\textstyle\bigcap_{\lambda=1}^nD^0_\lambda$ is a point, which we denote by $p_0$. It is clear that $D^0_\lambda$'s intersect each other transversally at $p_0$. 
From this fact and Bertini's theorem, we may assume that each $D^0_\lambda$ is non-singular. 

Consider an sequence of the subvarieties $V_n:=V, V_{n-1}:=D^0_1, V_{n-2}:=D^0_1\cap D^0_2, \cdots, V_1:=D^0_1\cap D^0_2\cdots, \cap D^0_{n-1}$. 
Denote by $F_\lambda$ the restriction $F|_{V_\lambda}$ for each $\lambda=1, 2, \cdots, n-1$. 
Note that it follows from a simple inductive argument that $(V_\lambda, F_\lambda)$ is also a del Pezzo manifold of degree $1$ for each $\lambda$. 
Especially, for $\lambda=1$, it holds that $V_1$ is an elliptic curve and ${\rm deg}\,F_1=1$. 
Let $\pi\colon\mathcal{X}\to S$ be the one obtained by the construction 
we described in \S \ref{section:setting_smooth} starting from $Z:=V$, $L:=F$, and $Y^0:=V_1$. 
By running the same argument as in \cite[\S 5]{N} inductively on $\lambda$, one have that this model actually enjoys {\bf Assumption $5$'} (One have that the anti-canonical bundle of $X_s$ is semi-ample, and that the morphism defined by the complete linear system $K_{X_s}^{-m}$ can be used as the fibration $b_s$ for each torsion $s$, where $m$ is a suitable positive integer). 
\end{example}

%

Next, we give an example of the configuration we described in \S \ref{section:setting_nodal}. 

\begin{example}\label{example:nodal}

Let $Y\subset \mathbb{P}^2$ be a cycle of rational curves which is of degree three, and $p_1, p_2, \dots, p_8$ be points of $Y\setminus Y_{\rm sing}$. 
Assume one of the following four conditions: 
$(1)$ $Y$ is a rational curve with one node, 
$(2)$ $Y$ consists of two irreduceble components $C_1$ and $C_2$ such that the degree of $C_1$ is one and of $C_2$ is two, $p_1, p_2\in C_1$ and $p_3, p_4, \dots, p_8\in C_2$, 
$(3)$ $Y$ consists of two irreduceble components $C_1$ and $C_2$ such that the degree of $C_1$ is two and of $C_2$ is three, $p_1, p_2, \dots, p_5 \in C_1$ and $p_6, p_7, p_8\in C_2$, 
or $(4)$ $Y$ consists of three irreduceble components $C_1, C_2$ and $C_3$ , $p_1, p_2\in C_1$, $p_3, p_4, p_5\in C_2$, and $p_6, p_7, p_8\in C_3$. 
Let $\pi\colon\mathcal{X}\to S$ be the one obtained by the construction 
we described in \S \ref{section:setting_smooth} 
by using the blow-up of $\mathbb{P}^2$ at $p_1, p_2, \dots, p_8$ as $Z$, 
and the strict transform of $Y$ as $Y^0$. 
By the same argument as in \cite[\S 5]{N}, one have that this model actually enjoys {\bf Assumption $5$''}.
\end{example}


\subsection{Proof of Theorem \ref{thm:main_9pt_b-up_of_P^2_nodal}}

Let $(X_Z, Y_Z)$ be as in \S 1 such that the anti-canonical bundle $K_{X_Z}^{-1}$ is nef, and that $Y_Z$ is a cycle of rational curves. 

When $K_{X_Z}^{-1}|_{Y_Z}$ is unitary flat (i.e. $|\alpha(K_{X_Z}^{-1}|_{Y_Z})|=1$), by the argument as in \cite[\S 7]{K6}, $(X_Z, Y_Z)=(X_s, Y_s)$ holds for some deformation $\pi\colon\mathcal{X}\to S$ as in Example \ref{example:nodal} and for some $s\in\mathrm{U}(1)\subset S$. 
Thus one have by Theorem \ref{thm:main_general} that there exists a local defining functions system $\{(V_j, w_j)\}$ of $Y_Z$ such that each ratio $w_j/w_k$ is a unitary constant. 
It follows by regarding $w_j$ as a local frame of $[Y_Z]$ on each $V_j$ that $[Y_Z]$ is flat on a neighborhood of $Y_Z$. 
Thus, by the argument as in the proof of \cite[Corollary 3.4]{K2} or \cite[\S 5]{K8} (see also \S \ref{section:relation_sp_nbhd}), one have the assertion $(i)$. 
Define a function $\Phi\colon \bigcup_jV_j\to\mathbb{R}$ by $\Phi:=\log |w_j|$ on each $V_j$. By considering neighborhoods $\{\Phi<\varepsilon\}$ of $Y_Z$ for sufficiently small positive numbers $\varepsilon$, one have that the assertion $(iii)$. 
The assertion $(iv)$ follows by considering Chern curvature forms of smooth Hermitian metrics on $[Y]$ with semi-positive curvature. 

When $K_{X_Z}^{-1}|_{Y_Z}$ is not unitary flat, neither the assertion $(i)$ nor $(iv)$ holds by \cite[Theorem 1.1 $(ii)$]{K6}. 
In this case, there exists a strongly plurisubharmonic function $\Psi$ on $V\setminus Y_Z$ for some neighborhood $V$ of $Y_Z$ according to \cite[Theorem 1.6 $(i)$]{K6}, by which one have that the assertion $(iii)$ does not hold. 
\qed

\subsection{Proof of Theorem \ref{thm:main_9pt_b-up_of_P^2_smooth}}

Let $(X_Z, Y_Z)$ be as in \S 1 such that $Y_Z$ is smooth and ${\rm rank}(N_Z)=1$, where $N_Z:=K_{X_Z}^{-1}|_{Y_Z}$ ($Z=\{p_1, p_2, \dots, p_9\}$). 
Denote by $V$ the blow-up of $\mathbb{P}^2$ at the eight points $p_1, p_2, \dots, p_7$, and $p_8$. Then, as $V$ is a del Pezzo manifold of degree $1$, one can obtain the deformation $\pi\colon \mathcal{X}\to S$ as in Example \ref{example:smooth}. 
Note that, it follows by the condition ${\rm rank}(N_Z)=1$ that one may assume that $X_Z\cong X_s$ for some element $s$ of $\mathrm{U}(1)\subset S$ by choosing $\tau$ and $q_0$ suitably in the construction in \S \ref{section:setting_smooth}.
Thus one have by Theorem \ref{thm:main_general} that there exists a local defining functions system $\{(V_j, w_j)\}$ of $Y_Z$ such that each ratio $w_j/w_k$ is a unitary constant. 

It follows by regarding $w_j$ as a local frame of $[Y_Z]$ on each $V_j$ that $[Y_Z]$ is flat on a neighborhood of $Y_Z$. 
Thus, by the argument as in the proof of \cite[Corollary 3.4]{K2} or \cite[\S 5]{K8} (see also \S \ref{section:relation_sp_nbhd}), one have the assertion $(i)$. 

Define a function $\Phi\colon \bigcup_jV_j\to\mathbb{R}$ by $\Phi:=\log |w_j|$ on each $V_j$. By considering neighborhoods $\{\Phi<\varepsilon\}$ of $Y_Z$ for sufficiently small positive numbers $\varepsilon$, one have that the assertion $(ii)$. 
\qed


\appendix

\section{Appendix: A correction to the paper ``Toward a higher codimensional Ueda theory''}

We found that the main theorem \cite[Theorem 1]{K5} was not correct. 
We have to add some assumptions in \cite[Theorem 1]{K5}. 
The main application \cite[Corollary 1]{K5} needs no correction. 

\subsection{Corrected form of \cite[Theorem 1]{K5}}

Corrected form of \cite[Theorem 1]{K5} is the following: 

\begin{theorem}\label{thm}
Let $X$ be a complex manifold, $S$ a smooth hypersurface of $X$, and $C$ be a smooth compact hypersurface of $S$ such that 
$N_{S/X}|_V$ is flat, where $V$ is a sufficiently small neighborhood of $C$ in $S$. 
Assume one of the following {two} conditions holds: 
$(i)$ $N_{C/S}\in\mathcal{E}_0(C)$, $N_{S/X}|_C\in\mathcal{E}_0(C)$, 
$(ii)$ $N_{C/S}$ and $N_{S/X}|_C$ are isomorphic to each other and they are elements of $\mathcal{E}_1(C)$. 
Further assume that $u_{n, m}(C, S, X; \{w_j\})=0$ holds for all $n\geq 1, m\geq 0$ and {for all system $\{w_j\}$ of order $(n, m)$, and that there exists a system of local defining functions of $C$ in $V$ of extension type infinity}. 
Then there exists a neighborhood $W$ of $C$ in $X$ such that $\mathcal{O}_X(S)|_W$ is flat. 
{Moreover, there exists a smooth hypersurface $Y$ of $W$ which intersects $S$ transversally along $C$. }
\end{theorem}

In the above statement, we removed the case $(iii)$ from \cite[Theorem 1]{K5} and added the assumption on the existence of a system of local defining functions of $C$ in $V$ {\it of extension type infinity}, which is the notion we posed in \cite{KO}. 
As a result, we could also add the conclusion on the existence of the transversal $Y$ to \cite[Theorem 1]{K5}. 
For the proof of Theorem \ref{thm}, see \cite[\S 3.4]{KO}. 

Let us explain some terms in Theorem \ref{thm}. 
We say the line bundle $L$ on a manifold $M$ is {\it flat} if the transition functions are chosen as constant functions valued in $U(1):=\{t\in\mathbb{C}\mid |t|=1\}$ (i.e. $L\in H^1(M, U(1))$). 
We denote by $\mathcal{E}_0(C)$ the set of all flat line bundles $F$ such that there exists a positive integer $n$ with $F^n=\mathbb
{I}$, where $\mathbb{I}$ is the holomorphically trivial line bundle. 
We denote by $\mathcal{E}_1(C)$ the set of all flat line bundles $F$ which satisfies the condition $|\log d(\mathbb{I}, F^{n})|=O(\log n)$ as $n\to\infty$, 
where $d$ is an invariant distance of the Picard group ($\mathcal{E}_1(C)$ does not depend on the choice of $d$, see \cite[\S 4.1]{U}). 
Let $(C, S, X)$ be as in Theorem \ref{thm}. 
In \cite[\S 3.1]{K5}, we defined the obstruction class $u_{n, m}(C, S, X)=u_{n, m}(C, S, X; \{w_j\})\in H^1(C, N_{S/X}|_C^{-n}\otimes N_{C/S}^{-m})$ for each $n\geq 1, m\geq 0$ and for each {\it system $\{w_j\}$ of order $(n, m)$}. 
We here explain the meaning of our new assumption ``there exists a system of local defining functions of $C$ in $V$ of extension type infinity''. 
Let $V$ be a sufficiently small tubular neighborhood of $C$ in $S$ and 
$W$ be a sufficiently small tubular neighborhood of $C$ in $X$ such that $W\cap S=V$. 
Take a sufficiently fine open covering 
$\{U_j\}$ of $C$, $\{V_j\}$ of $V$, and $\{W_j\}$ of $W$ 
such that $V_j=W_j\cap S$, $U_j=V_j\cap C$, and $U_{jk}:=U_j\cap U_k=\emptyset$ iff $W_{jk}:=W_j\cap W_k=\emptyset$. 
Extend a coordinates system $x_j$ of $U_j$ to $W_j$. 
Let $y_{j}$ be a defining function of $U_{j}$ in $V_{j}$ 
and $w_j$ a defining function of $V_j$ in $W_{j}$. 
As both $N_{S/X}$ and $N_{C/S}$ are flat in our settings, we may assume that 
$t_{jk}w_k=w_j+O(w_j^2)$ holds on $W_{jk}$ and $s_{jk}y_k=y_j+O(y_j^2)$ holds on $V_{jk}$ for some constants $t_{jk}, s_{jk}\in U(1)$. 
The assumption ``there exists a system of local defining functions of $C$ in $V$ of extension type infinity'' means that we can choose such $\{y_j\}$ with the following two additional properties: 
$(a)$ $s_{jk}y_k=y_j$ holds on $V_{jk}$ for each $j$ and $k$, and
$(b)$ $\{y_{j}\}$ is {\it of extension type infinity} in the sense of \cite[Definition 3.2]{KO}: i.e. the class $v_{n, m}(C, S, X;\{z_j\})\in H^1(C, N_{S/X}|_C^{-n}\otimes N_{C/S}^{-m+1})$ is equal to zero for each $n\geq 1, m\geq 0$ and for {\it any} type $(n, m)$ extension $\{z_{j}\}$ of $\{y_{j}\}$ (the class $v_{n, m}(C, S, X;\{z_j\})$ is the obstruction class we posed in \cite{KO}). 

We here remark that, as we will see later, \cite[Corollary 1]{K5} needs no correction. 
It is because the example of general $8$ points blow-up of $\mathbb{P}^3$ automatically satisfies this condition. 

\subsection{Details of the mistakes}

There are following three mistakes in \cite{K5}: 
one is on the well-definedness of the obstruction classes, 
another one is in the statement of \cite[Lemma 1]{K5}, 
and the other one is in a Taylor expansion in Lemma 6 and Lemma 7. 
In this section, we explain the details of these mistakes. 

\subsubsection{Mistake on the well-definedness of the obstruction classes}

The first one is on the well-definedness of the obstruction classes. 
In \cite[Proposition 3]{K5}, we stated that the $(n, m)$-th Ueda class $u_{n, m}(C, S, X)$ of the triple $(C, S, X)$ is independent of the choice of a system $\{w_j\}$ of order $(n, m)$ up to non-zero constant multiples. 
However, we found a critical mistake in the proof. 
Thus, now we should denote the obstruction class by $u_{n, m}(C, S, X; \{w_j\})$. 
See also \cite[\S 2.2.2, \S 3.2.2]{KO}.

\subsubsection{Mistake in the statement of \cite[Lemma 1]{K5}}

We also found a mistake in \cite[Lemma 1]{K5}, which is an open analogue of \cite[Lemma 2]{KS}. 
The corrected form of \cite[Lemma 1]{K5} should be stated as follows: 

\begin{lemma}[Corrected form of {\cite[Lemma 1]{K5}}]\label{lem:KS_2}
Let  $C$ be a compact complex manifold embedded in a complex manifold $S$ 
Fix a sufficiently small connected neighborhood $V$ of $C$ in $S$ and a sufficiently fine open covering $\{V_j\}$ of $V$ which consists of a finite number of open sets. 
Fix also a relatively compact open domain $V_0\subset V$ which contains $C$. 
For each flat line bundle $E$ on $V$, there exists a positive constant $K=K(E)$
such that, 
for each $1$-cocycle $\alpha=\{(V_{jk}, \alpha_{jk})\}$ of $E$ which can be realized as the coboundary of some $0$-cochain, there exists a $0$-cochain $\beta=\{(V_j{ \cap V_0}, \beta_j)\}$ of $E$ such that $\alpha{ |_{V_0}}$ is equal to the coboundary $\delta (\beta)$ of $\beta$ and the inequality 
\[
\max_j\sup_{V_0\cap V_j}|\beta_j|\leq K\cdot\max_{jk}\sup_{V_0\cap V_{jk}}|\alpha_{jk}|
\]
holds. 
\end{lemma}

This mistake is critical for proving \cite[Theorem 1]{K5} for the case $(iii)$, which is why we had to remove this case. 

\subsubsection{Mistake in a Taylor expansion in Lemma 6 and Lemma 7}

Here we explain the mistake under the configuration of Lemma 7. 
Lemma 7 is the lemma for defining the system of functions $\{G_j^{(n, m)}\}$ inductively: 
i.e. assuming that $\{G_j^{(\nu, \mu)}\}$ has already defined for each $(\nu, \mu)<(n, m)$, we are stating how to define $\{G_j^{(n, m)}\}$ in this lemma. 
For the definition of $\{G_j^{(n, m)}\}$, we regard $G_j^{(\nu, \mu)}z_j^\mu$ as the function defined on $W_j$ which does not depend on the variable $w_j$ and considered the expansion 
\begin{eqnarray}\label{eq:G_exp}
G_j^{(\nu, \mu)}(x_j)\cdot z_j^\mu&=&G_j^{(\nu, \mu)}(x_j(x_k, z_k, w_k))\cdot z_j(x_k, z_k, w_k)^\mu \nonumber \\
&=&G_j^{(\nu, \mu)}(x_j(x_k, 0, 0))\cdot s_{jk}^\mu z_k^\mu\nonumber \\
&&+\sum_{q=1}^\infty G_{jk, 0, q}^{(\nu, \mu)}(x_k)\cdot z_k^{\mu+q}
+\sum_{p=1}^\infty\sum_{q=0}^\infty G_{jk, p, q}^{(\nu, \mu)}(x_k)\cdot z_k^{\mu+q}w_k^{p} \nonumber
\end{eqnarray}
on $W_{jk}$, {in which we made a mistake}. 
This expansion should be 
\begin{eqnarray}\label{eq:G_exp_2}
G_j^{(\nu, \mu)}(x_j)\cdot z_j^\mu&=&G_j^{(\nu, \mu)}(x_j(x_k, z_k, w_k))\cdot z_j(x_k, z_k, w_k)^\mu \nonumber \\
&=&G_j^{(\nu, \mu)}(x_j(x_k, 0, 0))\cdot s_{jk}^\mu z_k^\mu\nonumber \\
&&+\sum_{q=1}^\infty G_{jk, 0, q}^{(\nu, \mu)}(x_k)\cdot z_k^{\mu+q}
+\sum_{p=1}^\infty\sum_{q=0}^\infty G_{jk, p, q}^{(\nu, \mu)}(x_k)\cdot { z_k^{q}} w_k^{p}.\nonumber
\end{eqnarray}
Even after this correction, the inductive definition of $\{G_j^{(n, m)}\}$ can be executed just as in Lemma 7. 
However, the norm estimate problem occurs in this case so that we can not show the convergence of the functional equation (8) in \cite[\S 4]{K5}. 
To avoid this difficulty, we have to refine not only the system $\{w_{j}\}$, but also the extension $\{z_{j}\}$ of $\{y_{j}\}$ by using a suitable functional equation at the same time (with fixing only $\{x_{j}\}$ and $\{y_{j}\}$), see \cite[\S 3.4]{KO} for the details. 

We here remark that, by the same reason, we also have to correct {\cite[Proposition 4, Lemma 5]{K5}}. 

\subsection{Proof of \cite[Corollary 1]{K5}}

Here we prove the following: 

\begin{corollary}[={\cite[Corollary 1]{K5}}]\label{main_cor}
Let $C_0\subset\mathbb{P}^3$ be a complete intersection of two quadric surfaces of $\mathbb{P}^3$ and let $p_1, p_2, \dots, p_8\in C_0$ be $8$ points different from each other. 
Assume $\mathcal{O}_{\mathbb{P}^3}(-2)|_{C_0}\otimes \mathcal{O}_{C_0}(p_1+p_2+\dots+p_8)\in \mathcal{E}_1(C_0)$. 
Then the anti-canonical bundle of the blow-up of $\mathbb{P}^3$ at $\{p_j\}_{j=1}^8$ is not semi-ample, however admits a smooth Hermitian metric with semi-positive curvature. 
\end{corollary}

\begin{proof}[Proof of Corollary \ref{main_cor}]
We use the notations in \cite[\S 5.2]{K5}. 
We apply Theorem \ref{thm} to the triple $(C, S_0, X)$. 
We here remark that the existence of the transversal $Y$ is clear in this example (consider $Y:=S_\infty$). 

All we have to do here is to check the added condition ``there exists a system of local defining functions of $C$ in $V$ of extension type infinity''. 
Let $\{s_{jk}\}$ and $\{y_j\}$ be as in \S 1 here. 
As $u_n(C, S_0)\in H^1(C, N^{-n})=0$ for each $n\geq 1$, we can conclude from \cite[Theorem 3]{U} that we may assume the condition $(a)$ $s_{jk}y_k=y_j$ holds on $V_{jk}$ for each $j$ and $k$. 
We will check the condition 
$(b)$ the class $v_{n, m}(C, S, X;\{z_j\})\in H^1(C, N_{S/X}|_C^{-n}\otimes N_{C/S}^{-m+1})$ is equal to zero for each $n\geq 1, m\geq 0$ and for {\it any} type $(n, m)$ extension $\{z_{j}\}$ of $\{y_{j}\}$. 
First we will check the case where $(n, m)=(1, 0)$. 
Note that the class $v_{1, 0}(C, S, X; \{z_j\})$ does not depend on the choice of an extension $\{z_{j}\}$ of $\{y_{j}\}$ (nor a system $\{w_j\}$). 
It can be shown by just the same (and much more simple) argument as in \cite[\S 3.2.2]{KO}. 
Thus, it is sufficient to show that $v_{1, 0}(C, S, X; \{z_j\})=0$ for a suitably fixed extension $\{z_{j}\}$ of $\{y_{j}\}$. 
For this purpose, let us fix an extension $z_{j}$ of $y_{j}$ such that $z_j$ is a defining function of $W_j\cap S_\infty$. 
Let 
\[
s_{jk}z_k=z_j+p_{jk}^{(1)}(x_j, z_j)\cdot w_j+O(w_j^{2})
\]
be the expansion in $w_j$ and 
\[
p_{jk}^{(1)}(x_j, y_j)=q_{jk}^{(1,0)}(x_j)+O(y_{j})
\]
be the expansion of $p_{jk}^{(1)}|_{V_{jk}}$ in $y_j$ for each $j$ and $k$. 
As $s_{jk}z_k/z_j$ is holomorphic around $W_{jk}\cap S_\infty$, we obtain that $p_{jk}^{(1)}(x_j, z_j)$ can be divided by $z_j$. 
Therefore we obtain that $v_{1, 0}(C, S, X; \{z_j\})=[\{q_{jk}^{(1,0)}\}]\equiv [\{0\}]=0$. 
Next we will check the case where $(n, m)>(1, 0)$. 
In this case, as $N$ is non-torsion and $n+m-1>0$, we obtain that $H^1(C, N_{S/X}|_C^{-n}\otimes N_{C/S}^{-m+1})=H^1(C, N^{-n-m+1})=0$ holds, which proves the assertion. 
\end{proof}


\end{document}